\documentclass[reqno]{amsart}

\usepackage{xypic,hyperref,todonotes,comment}
\usepackage{graphics,amssymb}
\usepackage{pagecolor,lipsum}
\usepackage{latexsym}
\usepackage{mathrsfs}
\xyoption{curve}
\usepackage{hyperref}
\hypersetup{
colorlinks,
citecolor=violet,
filecolor=blue,
linkcolor=violet,
urlcolor=violet
}

\newcommand{\nocontentsline}[3]{}
\newcommand{\tocless}[2]{\bgroup\let\addcontentsline=\nocontentsline#1{#2}\egroup}

\newtheorem{lemma}{Lemma}[section]
\newtheorem{proposition}[lemma]{Proposition}
\newtheorem{theorem}[lemma]{Theorem}
\newtheorem{corollary}[lemma]{Corollary}

\newtheorem{conjecture}[lemma]{Conjecture}

\newtheorem*{theoremA}{Theorem}

\theoremstyle{definition}
\newtheorem{example}[lemma]{Example}
\newtheorem{definition}[lemma]{Definition}
\newtheorem{remark}[lemma]{Remark}

\newcommand{\mfk}[1]{\mathfrak{#1}}
\newcommand{\mbb}[1]{\mathbb{#1}}
\newcommand{\mcl}[1]{\mathcal{#1}}
\newcommand{\mrm}[1]{\mathrm{#1}}
\newcommand{\msc}[1]{\mathscr{#1}}

\newcommand{\msf}[1]{\mathsf{#1}}

\DeclareMathOperator{\Hom}{Hom}
\DeclareMathOperator{\End}{End}
\DeclareMathOperator{\RHom}{RHom}

\DeclareMathOperator{\Ext}{Ext}

\DeclareMathOperator{\rep}{rep}
\DeclareMathOperator{\Rep}{Rep}
\DeclareMathOperator{\Spec}{Spec}
\DeclareMathOperator{\Proj}{Proj}
\DeclareMathOperator{\Spf}{Spf}
\DeclareMathOperator{\Supp}{Supp}

\DeclareMathOperator{\ord}{ord}

\DeclareMathOperator{\res}{res}
\DeclareMathOperator{\supp}{supp}

\DeclareMathOperator{\Sym}{Sym}
\DeclareMathOperator{\Coh}{Coh}

\DeclareMathOperator{\thick}{thick}

\let\oldO\O

\newcommand{\ot}{\otimes}
\renewcommand{\1}{\mathbf{1}}
\renewcommand{\O}{\mathscr{O}}
\renewcommand{\hat}{\widehat}
\renewcommand{\tilde}{\widetilde}

\renewcommand{\b}[1]{[\!\hspace{.1mm}[{#1}]\!\hspace{.1mm}]}

\title[Support for integrable Hopf algebras]{Support for integrable Hopf algebras via noncommutative hypersurfaces}
\date{\today}

\author{Cris Negron}
\address{Department of Mathematics, University of North Carolina, Chapel Hill, NC 27599}
\email{cnegron@email.unc.edu}

\author{Julia Pevtsova}
\address{Department of Mathematics, University of Washington, Seattle, WA 98195}
\email{julia@math.washington.edu}

\begin{document}
\maketitle

 

\begin{abstract}
We consider finite-dimensional Hopf algebras $u$ which admit a smooth deformation $U\to u$ by a Noetherian Hopf algebra $U$ of finite global dimension.  Examples of such Hopf algebras include small quantum groups over the complex numbers, restricted enveloping algebras in finite characteristic, and Drinfeld doubles of height $1$ group schemes.  We provide a means of analyzing (cohomological) support for representations over such $u$, via the singularity categories of the hypersurfaces $U/(f)$ associated to functions $f$ on the corresponding parametrization space.  We use this hypersurface approach to establish the tensor product property for cohomological support, for the following examples: functions on a finite group scheme, Drinfeld doubles of certain height 1 solvable finite group schemes, bosonized quantum complete intersections, and the small quantum Borel in type $A$.
\end{abstract}

\tocless\section{Introduction}

The present paper is dedicated to analyses of cohomology and support for finite-dimensional Hopf algebras.  Given a Hopf algebra $\msf{u}$ over a field $k$, we are particularly interested in its associated tensor category of finite-dimensional representations $\rep(\msf{u})$.  At a basic level, support theory proposes that one can use a certain geometry, namely the geometry of the spectrum of the cohomology algebra $\Ext^\ast_\msf{u}(k,k)$ or some resolution thereof, to understand the representation theory $\rep(\msf{u})$ \emph{as a tensor category}.
\par

Support theory has its foundations in the modular representation theory of finite group (schemes), that is, the representation theory of finite group schemes in finite characteristic.  This geometric approach to modular representation theory arguably began with work of Quillen in the 70's \cite{quillen71} and continues into the present with strong contributions of many authors.  For a finite group scheme $\mcl{G}$ one assigns to any finite-dimensional $\mcl{G}$-representation $V$ a closed subvariety $\supp(V)$ in the spectrum of the cohomology ring $H^\ast(\mcl{G},k)=\Ext^\ast_\mcl{G}(k,k)$.  The subvariety $\supp(V)$ is defined as the support of the $H^\ast(\mcl{G},k)$-module $\Ext^\ast_\mcl{G}(V,V)$, with action defined via the tensor structure on $\rep(\mcl{G})$.  In the studies \cite{balmer05,bensoniyengarkrausepevtsova18}, for example, supports of objects are employed to provide strong analyses of the global structure on the (stable) category of $\mcl{G}$-representations.  The fundamental property which allows one to access the tensor structure on $\rep(\mcl{G})$ via support is the so-called tensor product property, which appears as the equality
\[
\supp(V\ot W)=\supp(V)\cap\supp(W).
\]
One proves the above equality by identifying the spectrum of cohomology with a certain moduli space of maps from $\rep(\mcl{G})$ to $\rep(\mbb Z/p)$ called the \emph{rank variety} for $\mcl{G}$ \cite{carlson83,suslinfriedlanderbendel97b,friedlanderpevtsova05,friedlanderpevtsova07}.
\par

Now, for a given finite-dimensional Hopf algebra $\msf{u}$ one can still define support varieties $\supp(V)$ of finite-dimensional $\msf{u}$-representation via cohomology as above.  However, for almost all examples over the complex numbers it is unknown how these support varieties acknowledge or do not acknowledge the tensor structure on $\rep(\msf{u})$.  At a basic level, one would like to know that the tensor product property for support, or some precise analog thereof, holds for $\msf{u}$-representations.  The purpose of the present paper and its sequels is to provide a framework with which one can address support theory outside of the modular setting, and also to resolve the tensor product property for support at least in some basic instances.  Our motivating examples are quantum groups over the complex numbers and Drinfeld doubles of finite group schemes in finite characteristic.  Before stepping into the world of support however, one does have some more fundamental questions to address.

Before one even starts to think about support theory for representations in terms of cohomological actions, one would like to guarantee that the cohomology ring in question is nice enough, specifically, finitely generated. We adopt the following terminology (cf.~\cite{negronplavnik}): A finite-dimensional Hopf algebra $\msf{u}$, over a field $k$, is said to have {\it finitely generated cohomology} if the self-extensions $\Ext^\ast_\msf{u}(k,k)$ are a finitely generated algebra and, for any finite-dimensional $\msf{u}$-representation $V$, the extensions $\Ext^\ast_\msf{u}(k,V)$ are a finite module over $\Ext^\ast_\msf{u}(k,k)$.  The so-called {\it finite generation conjecture}~\cite{friedlandersuslin97,etingofostrik04} proposes that all finite-dimensional Hopf algebras have finitely generated cohomology.
\par

Finite group schemes, or equivalently finite-dimensional cocommutative Hopf algebras over an arbitrary field, were shown to have finitely generated cohomology in celebrated work of Friedlander and Suslin \cite{friedlandersuslin97}.  The particular case of discrete groups had been covered in earlier works of Golod, Venkov and Evens \cite{golod59,venkov59,evens61}.  In the quantum setting over $\mbb{C}$, finite generation is known in many important cases, including small quantum groups $u_q(\mfk{g})$ and their duals \cite{ginzburgkumar93,bnpp14}, and certain fundamental classes of finite-dimensional pointed Hopf algebras \cite{mpsw10,aapw}.
\par

One observation of this work is that many situations for which finite generation is known can be united under the same umbrella of a {\it smooth integrability} property for Hopf algebras (see Definition~\ref{def:integrable}). We define and develop this property in the first part of the paper and show how it implies finite generation of cohomology, recovering several known results. Building on the ideas of Avramov-Buchweitz \cite{avramovbuchweitz00} and Avramov-Iyengar \cite{avramoviyengar18} in commutative algebra, we introduce an alternative notion of a rank variety formulated in terms of hypersurfaces on the smooth integration of the algebra $\msf{u}$ (see Definition~\ref{def:hyp_supp}). In some fundamental cases we are able to identify this new hypersurface support with the cohomological support. As an application, we show that the tensor product property holds in various situations including representations of the small quantum Borel in type $A$, and Drinfeld doubles of some naturally occurring height $1$ solvable group schemes in finite characteristic (see Section \ref{sect:PiqPi}).  Precise results for the full quantum group $u_q(\mfk{g})$ should appear in a subsequent paper \cite{negronpevtsova3}.
\par

Let us now describe the technical content of this paper in more detail. 
Throughout $k$ is an algebraically closed base field of varying characteristic.  Consider $\msf{u}$ a finite-dimensional Hopf algebra, over $k$.  We say that $\msf{u}$ is {\it smoothly integrable}, or just {\it integrable}, if $\msf{u}$ admits a deformation $U\to \msf{u}$ parametrized by a smooth, augmented, central subalgebra $Z\subset U$ such that (a) $U$ is a Noetherian Hopf algebra of finite global dimension, and $U\to \msf{u}$ is a Hopf map, and (b) $Z$ is a coideal subalgebra in $U$.  By smoothness we mean that $Z$ is smooth over our given base field $k$, and in particular of finite type over $k$.
\par

For example, the De Concini-Kac algebra $U^{DK}_q(\mfk{g})$ integrates the quantum group $u_q(\mfk{g})$, and we have a similar deformation for the quantum Borel $U^{DK}_q(\mfk{b})\to u_q(\mfk{b})$.  A restricted enveloping algebra $u^{\rm res}(\mfk{g})$, in finite characteristic, is smoothly integrated by the associated universal enveloping algebra $U(\mfk{g})\to u^{\rm res}(\mfk{g})$.  For $\mcl{G}$ a finite group scheme, any choice of embedding $\mcl{G}\to \mcl{H}$ into a smooth connected algebraic group $\mcl{H}$ provides a smooth integration $\O(\mcl{H})\to \O(\mcl{G})$ of functions on $\mcl{G}$.  The last two examples can be combined to produce a smooth integration $\tilde{D}(\mcl{G})\to D(\mcl{G})$ of the Drinfeld double of an infinitesimal height $1$ group scheme which admits a {\it normal} embedding into a smooth algebraic group.  See Section~\ref{sect:examples} for more details.
\par

As remarked above, the point of this paper is to explain how integrability of a given finite-dimensional Hopf algebra $\msf{u}$ can be employed in a meaningful way in analyses of cohomology and support.  At a relatively basic level, we have the following.
\vspace{2mm}

\begin{theoremA}[\ref{cor:fg_revisit}]
If a Hopf algebra $\msf{u}$ is integrable, then $\msf{u}$ has finitely generated cohomology.
\end{theoremA}

We note that the above theorem recovers finite generation for all of the examples above~\cite{friedlanderparshall86II,ginzburgkumar93,friedlandernegron18}, although it does not give information on the spectrum of cohomology beyond the embedding dimension.  (Our results  for height $1$ doubles are slightly stronger than those of~\cite{friedlandernegron18}.) The main technical tool here is a dg version of Koszul duality for the algebras $\Sym((m_Z/m_Z^2)^*)$ and $\wedge(m_Z/m_Z^2)$, where $m_Z$ is the kernel of the augmentation on $Z$.
\par

For $\msf{u}$ with finitely generated cohomology, and $\msf{Y}$ the reduced projective spectrum $\Proj(\Ext^\ast_\msf{u}(k,k)\big)_{\rm red}$, we consider the cohomological supports of objects
\begin{equation} 
\label{def:supp}
\supp_\msf{Y}(V):=\Supp_{\msf{Y}}\Ext^\ast_\msf{u}(V,V)^\sim,
\end{equation}
where $(-)^\sim$ denotes the sheaf on $\msf{Y}$ associated to a given graded $\Ext^\ast_\msf{u}(k,k)$-module.
\par

Suppose now that we have smoothly integrable $\msf{u}$ with associated deformation $U$ over $Z$.  The distinguished point $\epsilon: Z\to k$ at which we have $k\ot_Z U\cong \msf{u}$ must be the counit restricted to $Z$ in this case.
\par

Now, any choice of section $\tau:m_Z/m_Z^2\to m_Z$ specifies, for each point $c$ in the $m_Z/m_Z^2-\{0\}$, a corresponding ``noncommutative hypersurface" $U_c=U/(\tau(c))$.  As the algebra $U_c$ depends only on $c$ up to a scaling, we may instead consider the $U_c$ as associated to points in the projectivization $c\in \mbb{P}(m_Z/m_Z^2)$.  In what follows we say that $M$, a finitely generated module for $U_c$, is {\it perfect} if it has finite projective dimension. For $V$ a $\msf{u}$-representation we are interested in whether $V$ is, or is not, perfect when restricted to a given hypersurface $U_c$.
\par

In addition to finiteness of cohomology, our deformation situation also provides us with a map of varieties
\begin{equation}\label{eq:151}
\kappa:\msf{Y}=\Proj(\Ext^\ast_\msf{u}(k,k)\big)_{\rm red}\to \mbb{P}(m_Z/m_Z^2).
\end{equation}
The map $\kappa$ always has finite fibers, and in a number of important examples~\cite{friedlanderparshall86II,ginzburgkumar93,friedlandernegron18}, is a closed embedding.  The following theorem is a noncommutative variant of a result of Avramov and Buchweitz~\cite{avramovbuchweitz00}.

\begin{theoremA}[\ref{cor:hopf_hyp}]
Suppose that $\msf{u}$ is smoothly integrable, with corresponding deformation $U \to \msf{u}$ and parametrizing subalgebra $Z$.  Then for any $V$ in $\rep(\msf{u})$, and any choice of section $\tau: m_Z/m^2_Z\to m_Z$, we have
\[
\kappa\big(\supp_\msf{Y}(V)\big)=\{c\in \mbb{P}(m_Z/m_Z^2):V\text{ \rm is not perfect when restricted to }U_c\},
\]
where $\kappa$ is as in \eqref{eq:151}.  In particular, when $\kappa$ is a closed embedding there is an equality
\begin{equation}\label{eq:hyp0}
\supp_\msf{Y}(V)=\{c\in \mbb{P}(m_Z/m_Z^2):V\text{ \rm is not perfect when restricted to }U_c\}.
\end{equation}
\end{theoremA}

Implicit in the statement above is the fact that perfection of a given $\msf{u}$-representation over $U_c$, at some $c\in \mbb{P}(m_Z/m_Z^2)$, depends only on the point $c$, not on the choice of section $\tau: m_Z/m_Z^2 \to m_Z$.  This fact is proved independently in Corollary~\ref{cor:reductions}.
\par

We use the hypersurface expression~\eqref{eq:hyp0} to provide a stronger analysis of support for certain ``solvable" Hopf algebras.  As discussed above, we are particularly interested in the so-called {\it tensor product property} for support,
\begin{equation}\label{eq:tpp}
\supp_{\msf{Y}}(V\ot W)\overset{\rm question}=\supp_\msf{Y}(V)\cap\supp_\msf{Y}(W).
\end{equation}
This relation suggests, in the cases in which it applies, that cohomology actually ``understands" something about the tensor structure on the category $\rep(\msf{u})$.
\par

It is known, via work of Witherspoon and coauthors, that for Hopf algebras which are sufficiently noncocommutative, so that $\rep(\msf{u})$ is sufficiently non-symmetric, the tensor product property will {\it not} hold~\cite{bensonwitherspoon14,plavnikwitherspoon18}.  In the present work we consider certain centralizing hypotheses on objects, and (co)normality hypotheses on the integration $U\to \msf{u}$, when studying support for tensor products of objects (see Section~\ref{sect:Z_supp}).  We discuss some of our examples below.

\subsection{Applications}
\label{sect:PiqPi}

We first consider rings of regular functions $\O(\mcl{G})$ on finite group schemes in finite characteristic, where we have $\rep(\O(\mcl{G}))=\Coh(\mcl{G})$.  

\begin{theoremA}[\ref{thm:Pio}]
Let $\mcl{G}$ be a connected finite group scheme, in finite characteristic.  Then cohomological support for $\Coh(\mcl{G})$ satisfies the tensor product property
\[
\supp_\msf{Y}(V\ot W)= \supp_\msf{Y}(V) \cap \supp_\msf{Y}(W).
\]
\end{theoremA}

There are some referential points to take into account here, as we explain in the preamble to Section~\ref{sect:O} and Remark~\ref{rem:CI}.
\par

It is clear, from work of Plavnik and Witherspoon~\cite{plavnikwitherspoon18} (see also Example~\ref{ex:no_tpp}), that the precise tensor product property \eqref{eq:tpp} does not hold for coherent sheaves on general non-connected $\mcl{G}$.  However, we provide a version of the tensor product property for arbitrary finite group schemes by considering certain centralizing structures on objects in $\Coh(\mcl{G})$ (cf.\ \ref{sect:Z_supp}).

\begin{theoremA}[\ref{cor:Pi}]
Let $\mcl{G}$ be an arbitrary finite group scheme with $\pi=\pi_0(\mcl{G})$ its subgroup of connected components.  For $W$ in $\Coh(\mcl{G})$, and $V$ in the Drinfeld centralizer $Z^{\Coh(\pi)}(\Coh(\mcl{G}))$ of $\Coh(\pi)$, we have
\[
\supp_\msf{Y}(F(V)\ot W)=\supp_\msf{Y}(F(V))\cap \supp_\msf{Y}(W).
\]
\end{theoremA}

Here $Z^{\Coh(\pi)}(\Coh(\mcl{G}))$ is the category of objects in $\Coh(\mcl{G})$ which centralize the fusion subcategory $\Coh(\pi)\subset \Coh(\mcl{G})$ of semisimple objects.  More specifically, $Z^{\Coh(\pi)}(\Coh(\mcl{G}))$ is the category of objects which admit a \emph{centralizing structure} against this distinguished subcategory, and $F:Z^{\Coh(\pi)}(\Coh(\mcl{G}))\to \Coh(\mcl{G})$ is the functor forgetting the specific choice of centralizing structure (see Definition~\ref{def:ZC}).
\par

In a non-classical setting, over $\mbb{C}$, one can consider a truncated skew polynomial ring 
\[
\msf{a}^+_q=\msf{a}^+_q(P)=\mbb{C}\langle x_1,\dots, x_n\rangle/(x_ix_j-q_{ij}x_jx_i,\ x_i^l),
\] 
which is often referred to as a quantum complete intersection, or quantum linear space, in the literature.  Here $q$ is a root of unity of odd order $l$ and $P=[a_{ij}]$ is a skew-symmetric matrix specifying the parameters $q_{ij}=q^{a_{ij}}$.  The algebra  $\msf{a}^+_q$ has the natural structure of a Hopf algebra in a certain braided fusion category $\rep(\Lambda)$, where $\Lambda$ is the group ring of a finite abelian group.  One then obtains, via a standard bosonization procedure, a usual Hopf algebra $\msf{a}_q:=\msf{a}^+_q\rtimes \Lambda$. We refer to the Hopf algebra $\msf{a}_q$ as a {\it bosonized quantum complete intersection}.  

\begin{theoremA}[\ref{thm:fun1}]
Consider $\msf{a}_q=\msf{a}_q(P)$ a bosonized quantum complete intersection, with generic parameters $q$ and $P$.  Cohomological support for $\rep(\msf{a}_q)$ satisfies the tensor product property
\[
\supp_\msf{Y}(V\ot W)=\supp_\msf{Y}(V)\cap \supp_\msf{Y}(W).
\]
\end{theoremA}

By ``generic" parameters we mean precisely that the determinant of the matrix $P=[a_{ij}]$ is a unit mod $l$.  We prove in Theorem \ref{thm:qci} that, at \emph{arbitrary} $q$ and $P$, representation of $\msf{a}_q$ still satisfy a version of the tensor product property where one again employs centralizers, exactly as in the case of sheaves $\Coh(\mcl{G})$ on a non-connected group scheme $\mcl{G}$. 
\par

At Theorem~\ref{thm:ures}, we reprove the tensor product property for $\rep(u^{\rm res}(\mfk{n}))$ via hypersurfaces, where $\mfk{n}$ is a nilpotent restricted Lie algebra in large characteristic.  This is a fundamental result of Friedlander and Parshall~\cite[Theorem 2.7]{friedlanderparshall87}, which precedes the $\pi$-point approach to support for finite group schemes~\cite{friedlanderparshall86,suslinfriedlanderbendel97b,friedlanderpevtsova07}.  We consider also Drinfeld doubles $D(B_{(1)})$ for $B_{(1)}$ the first Frobenius kernel of a Borel $B\subset \mbb{G}$ in an almost-simple algebraic group $\mbb{G}$.

\begin{theoremA}[\ref{thm:D_unip}]
Consider $\mbb{G}$ an almost-simple algebraic group over $\overline{\mbb{F}}_p$, and let $B$ be a Borel subgroup in $\mbb{G}$.  If $p>\dim(B)+1$, then cohomological support for the Drinfeld double $D(B_{(1)})$ satisfies the tensor product property.
\end{theoremA}

We consider, finally, the small quantum Borel $u_q(\mfk{b})$ in a quantum group $u_q(\mfk{g})$, over $\mbb{C}$.  This algebra, and its full counterpart $u_q(\mfk{g})$, are the most important characteristic $0$ examples for which cohomological support is not understood at this point.  We address the quantum Borel in type $A$.

\begin{theoremA}[\ref{thm:borel_A}]
Consider $u_q(\mfk{b})$ the quantum Borel in type $A_n$, at $q$ of odd order $>h$.  Cohomological support for $u_q(\mfk{b})$ satisfies the tensor product property.
\end{theoremA}

In the above statement $n$ is arbitrary and $h$ denotes the Coxeter number for $A_n$, which is just $n+1$.  We observe in Corollary~\ref{cor:weak_uqb} that, for the quantum Borel $u_q(\mfk{b})$ in \emph{any} Dynkin type, there is at least a containment
\[
\supp_\msf{Y}(V\ot W)\subset\left(\supp_\msf{Y}(V)\cap\supp_\msf{Y}(W)\right).
\]
We conjecture, at Conjecture~\ref{conj:borel} below, that this containment is an equality in all Dynkin types.

\subsection{Snashall-Solberg support for associative algebras}

There is a theory of support for finite-dimensional (non-Hopf) algebras $R$, introduced by Snashall and Solberg~\cite{snashallsolberg04,erdmannetal04}, which utilizes the map from Hochschild cohomology to the center of the derived category $-\ot^{\rm L}_Rid:H\!H^\ast(R)\to Z(D^b(R))$.  The materials of Sections~\ref{sect:defos_n_fin} and~\ref{sect:supports} are presented in the general setting of a finite-dimensional algebra $R$ with a Noetherian deformation $Q\to R$ of finite global dimension.  We observe at Theorem~\ref{thm:cohom_hyp} that the Snashall-Solberg type support in this case is identified with the corresponding support defined via hypersurfaces.  That is to say, we show that the proper analog of equation~\eqref{eq:hyp0} holds in this context.

\subsection{Related works}

In~\cite{boekujawanakano} Boe, Kujawa, and Nakano propose an extremal version of the tensor product property for (possibly infinite-dimensional) modules over $u_q(\mfk{b})$, at odd order parameter with $\ord(q)>h$.  Specifically, the authors suggest that for modules $M$ and $N$ over $u_q(\mfk{b})$, the support of the product $M\ot N$ vanishes if and only if the supports for $M$ and $N$ have trivial intersection.  The approach of~\cite{boekujawanakano} involves reductions to an associated graded algebra $\operatorname{gr}u_q(\mfk{g})$ with respect to a root vector filtration of De Concini and Kac~\cite{deconcinikac91}.
\par

In~\cite{bensonerdmannholloway07} and~\cite{pevtsovawitherspoon09,pevtsovawitherspoon15}, Benson-Erdmann-Holloway and Pevtsova-Witherspoon, respectively, consider support for the (bosonized) quantum complete intersections $\msf{a}_q$ with constant parameter $q_{ij}=q_{12}$, for all $i<j$.  (To be clear, in~\cite{bensonerdmannholloway07} the tensor structure is not considered, and in~\cite{pevtsovawitherspoon09,pevtsovawitherspoon15} the skewing parameter is taken to be of constant value $1$.)  In~\cite{bensonerdmannholloway07} certain rank varieties are also constructed for $\msf{a}_q$ with arbitrary parameters by reducing, via a sequence of $\mbb{C}$-linear functors, i.e.\ non-tensor functors, to the constant parameter case~\cite[Definition 4.4]{bensonerdmannholloway07}.
\par

One can view the materials of this paper as an approach to support theory which is an ``inversion" of the canonical $\pi$-point approach to support for finite group schemes~\cite{suslinfriedlanderbendel97b,friedlanderpevtsova07,drupieskikujawa,bensoniyengarhenning}.  A strictly categorical reinterpretation of the $\pi$-point approach to support can be found in work of Balmer, Krause, and Stevenson~\cite{balmerkrausestevenson19,balmer20}.

\subsection{A basic outline}

The paper essentially has two interlocking halves.  In Sections~\ref{sect:defos_basics}--\ref{sect:hopf_supp} we show that the existence of a smooth integration $U\to \msf{u}$ implies strong finiteness conditions on cohomology (Theorem~\ref{thm:fg}).  We then develop the fundamentals regarding hypersurface support.  Sections \ref{sect:bg_examples}--\ref{sect:last} are dedicated to examples and applications.  Addressing the tensor product property for these examples does require the cultivation of certain additional tools, and in particular requires an analysis of derived categories for noncommutative local hypersurface rings (see Lemma \ref{lem:gen_lem2}).

\subsection{Acknowledgements}

Thanks to Nicholas Andruskiewitsch, Srikanth Iyengar, Eric Friedlander, Ellen Kirkman, Dan Nakano, Dan Rogalski, Mark Walker, Sarah Witherspoon and James Zhang for helpful conversation.  This work began at the conference Stable Cohomology: Foundations and Applications, May 2018 at Snowbird Utah, and continued during a visit of the first author to the University of Washington.  We thank the organizers of the Snowbird conference, and the staff of the Cliff Lodge, for providing a hospitable and intellectually engaging environment.  The first author thanks the University of Washington mathematics department for their hospitality.
\par

This material is based upon work supported by the National Science Foundation under Grant No.\ DMS-1440140, while the second author was in residence at the Mathematical Sciences Research Institute in Berkeley, California, during the Spring 2020 semester. We thank the MSRI staff for their exceptional efforts during the latter half of that difficult semester when the second author was also in virtual residence.  The second author was supported by the NSF grants DMS-1501146 and DMS-1901854.

\tableofcontents

\section{Examples of integrable Hopf algebras}
\label{sect:examples}

Throughout $k$ is an algebraically closed field of varying characteristic, and all algebras are algebras over $k$. We repeat the definition from the introduction before giving a long list of examples of smoothly integrable Hopf algebras. 
\par

Recall that a deformation of an algebra $\msf{u}$, parametrized by an affine scheme $\Spec(Z)$ with choice of distinguished closed point $x\in \Spec(Z)$, is the information of a flat $Z$-algebra $U$ equipped with an algebra map $U\to \msf{u}$ which reduces to an isomorphism $k\ot_ZU\cong \msf{u}$.  Here we reduce along the given point $x:Z\to k$.  Such a deformation can be represented via the corresponding sequence of algebra maps $Z\to U\to \msf{u}$, which we refer to informally as a deformation sequence.

\begin{definition}\label{def:integrable}
A finite-dimensional Hopf algebra $\msf{u}$ is said to be {\it smoothly integrable}, or just {\it integrable}, if $\msf{u}$ admits a deformation $U\to \msf{u}$ parametrized by a smooth central subalgebra $Z\subset U$ such that
\begin{enumerate}
\item[(a)] $U$ is a Noetherian Hopf algebra of finite global dimension, and $U\to \msf{u}$ is a map of Hopf algebras.
\item[(b)] $Z$ is a coideal subalgebra in $U$.
\end{enumerate}
Under these conditions, $U$ is called a {\it smooth integration} of $\msf{u}$ parametrized by $Z$.
\end{definition}

We suppose specifically that $Z$ is a \emph{right} coideal subalgebra in $U$, so that the comultiplication on $U$ restricts to a coaction $\Delta:Z\to Z\ot U$.

\begin{remark}
Our choice of right versus left is irrelevant.  If $U\to \msf{u}$ is a smooth integration with parametrizing algebra $Z$ a right coideal subalgebra, then $U$ is also flat over the left coideal subalgebra $S(Z)$ and $k\ot_{S(Z)}U=\msf{u}$.
\end{remark}

Implicit in Definition \ref{def:integrable} is the assumption that $Z$ is of finite type over $k$.  Although all of our examples arrive in such a finite type flavor, it is convenient to complete and work with a formal version of integrability.  We elaborate on the formal setting in Section \ref{sect:formal} below, after giving our examples.
\par

We note that for an integration $f:U\to \msf{u}$ of a finite-dimensional Hopf algebra $\msf{u}$, as in Definition~\ref{def:integrable}, $Z$ is contained in the $\msf{u}$-coinvariants in $U$ under the left coaction $(f\ot 1)\Delta:U\to \msf{u}\ot U$.  Indeed, $Z$ must be {\it equal} to the $\msf{u}$-coinvariants $Z=U^{\msf{u}}$ in this case~\cite[Theorem 1]{takeuchi79}.  So the parametrization space $\Spec(Z)$ is determined uniquely by the Hopf surjection $U\to \msf{u}$.

\subsection{Examples of smooth integration, and an anti-example}\ 

\begin{example}[Finite abelian groups]
Any finite abelian group $A$ admits a surjective group map from a free abelian group of finite rank, $p:\mbb{Z}^r\to A$. Let $A^\prime$ be the kernel of $p$.  Note that $A^\prime$ is a free abelian group of finite rank as well.  We take group rings $kA^\prime\to k(\mbb{Z}^r)\to kA$, over an arbitrary field, to obtain a smooth integration of $kA$.
\end{example}

\begin{example}[Rings of regular functions on a finite group scheme]\label{ex:1}
Take $k=\overline{\mbb{F}}_p$, and let $\mbb{G}$ be a smooth connected algebraic group. Denote by  $\mbb{G}^{(r)}$ the $r^{\rm th}$ Frobenius twist of $\mbb G$ and let $\mcl{G} = \mbb{G}_{(r)}$ be the $r^{\rm th}$ Frobenius kernel of $\mbb G$, that is, the kernel of the $r^{\rm th}$ iteration of the Frobenius map $F^r: \mbb G \to \mbb{G}^{(r)}$. The algebra of functions $\O(\mcl{G})$ is a complete intersection, and in this instance we get a deformation sequence 
\[
\O(\mbb{G}^{(r)})\to \O(\mbb{G})\to \O(\mcl{G}),
\]  
so that $\O(\mcl{G})$ is smoothly integrated by $\O(\mbb{G})$.
\par

Frobenius kernels are examples of finite connected group schemes. More generally, for $\mcl{G}$ {\it any} finite group scheme, we choose an embedding $\mcl{G}\to \mcl{H}$ into a smooth connected algebraic group.  Any faithful representation $V$ for $\mcl{G}$ provides an embedding into $\mcl{H}=\operatorname{GL}(V)$, for example.  Then we have the deformation sequence $\O(\mcl{H}/\mcl{G})\to \O(\mcl{H})\to \O(\mcl{G})$, and find that $\O(\mcl{H})$ integrates $\O(\mcl{G})$.  (Recall that $\mcl{H}/\mcl{G}$ is smooth~\cite[Corollary 5.26]{milne17}.)  When $\mcl{G}$ is not normal in $\mcl{H}$, $\O(\mcl{H}/\mcl{G})$ is not a Hopf subalgebra in $\O(\mcl{H})$, but a $\O(\mcl{H})$-coideal subalgebra.  Indeed, functions on $\mcl{H}/\mcl{G}$ are, essentially by definition, the $\O(\mcl{G})$-coinvariants in $\O(\mcl{H})$.
\end{example}

\begin{example}[Restricted enveloping algebras]\label{ex:2}
Take $k=\overline{\mbb{F}}_p$ and $\mfk{g}$ a restricted Lie algebra.  We have the surjection $U(\mfk{g})\to u^{\rm res}(\mfk{g})$ onto the restricted enveloping algebra, and consider the Zassenhaus subalgebra, or $p$-center, $Z_0(\mfk{g})$ of $U(\mfk{g})$.  This is the central subalgebra generated by the differences $x^p-x^{[p]}$ at arbitrary $x\in \mfk{g}$~\cite{zassenhaus54,jantzen98}.  The algebra $Z_0(\mfk{g})$ is a Hopf subalgebra in $U(\mfk{g})$ which is isomorphic to a polynomial ring, and $k\ot_{Z_0(\mfk{g})}U(\mfk{g})=u^{\rm res}(\mfk{g})$, so that $u^{\rm res}(\mfk{g})$ is seen to be smoothly integrated by the standard universal enveloping algebra.
\end{example}

\begin{example}[Height $1$ doubles]\label{ex:3}
Consider again $\mcl{H}$ a smooth connected algebraic group over $\overline{\mbb{F}}_p$, and $\mcl{G}\subset \mcl{H}$ a normally embedded, finite, connected subgroup scheme of height 1. Recall that a connected finite group scheme $\mcl{G}$ is of height 1 if  the $p^{\rm th}$ power of any element in the augmentation ideal of $ \O(\mcl{G})$ vanishes. Equivalently, $\mcl{G}$ is a subgroup scheme of $\mcl{H}_{(1)}$, the first Frobenius kernel of $\mcl H$ (see \cite{jantzen03}).  Let $\mfk{g}$ denote the restricted Lie algebra for $\mcl{G}$. We denote by $k\mcl{G}$ the {\it group algebra} of $\mcl{G}$, which is the linear dual of the algebra of regular functions $\O(\mcl{G})$.  Since $\mcl{G}$ is of height $1$, there is an identification $k\mcl{G}=u^{\rm res}(\mfk{g})$ (\cite[I.9.6]{jantzen03}).  Then $k\mcl{G}$ acts on $\O(\mcl{G})$ via the adjoint representation and we consider the Drinfeld double 
\[
D(\mcl{G})=\O(\mcl{G})\rtimes_{\rm ad}k \mcl{G}.
\]  
This algebra is integrated by the smash product $\O(\mcl{H})\rtimes_{\rm ad} U(\mfk{g})$, where $U(\mfk{g})$ acts on $\O(\mcl{H})$ via the projection $U(\mfk{g})\to k\mcl{G}$ (using the isomorphism $k\mcl{G} \cong u^{\rm res}(\mfk{g})$) and adjoint action of $\mcl{G}$ on $\mcl{H}$.  The parametrization space in this case is provided by the product $\O(\mcl{H}/\mcl{G})\ot Z_0(\mfk{g})$.  The relative double $D(\mcl{H}_1,\mcl{G}):=\O(\mcl{H}_1)\rtimes_{\rm ad}k\mcl{G}$ is similarly integrated by $\O(\mcl{H})\rtimes_{\rm ad}U(\mfk{g})$.
\end{example}

\begin{example}[Small quantum groups]\label{ex:4}
Consider the De Concini-Kac quantum enveloping algebra $U_q^{DK}(\mfk{g})$ at $q$ a root of unity of odd order $l$.  We have the standard sequence 
\[
Z_0\to U_q^{DK}(\mfk{g})\to u_q(\mfk{g}),
\] where $Z_0$ is the subalgebra generated by the $l$-th powers $E_\gamma^{l_\gamma}$, $F^{l_\gamma}_\gamma$, $K_\alpha^{l_\alpha}$~\cite[Section 3]{deconcinikac91},~\cite[Proposition 5.6]{deconcinikacporcesi92}.  By filtering by a normal ordering on the root vectors and considering the associated graded algebra (cf.~\cite{deconcinikac91}), we find that $U_q^{DK}(\mfk{g})$ is of finite global dimension~\cite[Corollary 3.1.2]{chemla04}.  Hence $u_q(\mfk{g})$ is smoothly integrated by $U^{DK}_q(\mfk{g})$.  Similarly, the quantum Borel $u_q(\mfk{b})$ is integrated by $U^{DK}_q(\mfk{b})$. For the quantum Borel, however, we consider the De Concini-Kac algebra $U^{DK}_q(\mfk{b})$ with the same (torsion) group of grouplikes as its small counterpart $u_q(\mfk{b})$.
\end{example}

In each of the above examples, faithful flatness of the extension $Z\to U$ is well-known.  However, it is also the case that any extension $Z\to U$ of a commutative coideal subalgebra $Z$ whatsoever is faithfully flat~\cite[Proposition 3.12]{arkhipovgaitsgory03}.  (There is a problem with the proof given in~\cite{arkhipovgaitsgory03}, but the result is nonetheless correct~\cite{gaitsgory}.)

\begin{example}[Anti-example]
The group ring $k(\operatorname{GL}_n)_{(2)}$ of the second Frobenius kernel in $\operatorname{GL}_n$, over $\overline{\mbb{F}}_p$, should not be integrable.  This is due to the presence of divided powers in the algebra.
\end{example}

\subsection{Formality and deformations}
\label{sect:formal}

In considering deformations of Hopf algebras, it is often advantageous to suppose that the parametrizing subalgebra $Z$ is not smooth, but \emph{formally} smooth.  

\begin{definition}\label{def:integrableformal}
A finite-dimensional Hopf algebra $\msf{u}$ is said to be {\it formally smoothly integrable} if $\msf{u}$ admits a deformation $\msf{U}\to \msf{u}$ parametrized by a central subalgebra $Z\subset \msf{U}$ such that
\begin{enumerate}
\item[(z)] $Z$ is isomorphic to a power series algebra $k\b{y_1,\dots,y_n}$.
\item[(a)] $\msf{U}$ is a Noetherian Hopf algebra of finite global dimension, and $\msf{U}\to \msf{u}$ is a map of Hopf algebras.
\item[(b)] $Z$ is a (right) coideal subalgebra in $\msf{U}$.
\end{enumerate}
In this case we call $\msf{U}$ a \emph{formally smooth integration} of $\msf{u}$.
\end{definition}

Abstractly, a commutative algebra $Z$ is isomorphic to such a power series provided $Z$ is formally smooth and complete local, with finite-dimensional cotangent space $m_Z/m_Z^2$ at the unique maximal ideal $m_Z$.  Such a formally smooth integration $\msf{U}$, which is necessarily a finite $Z$-module, is complete with respect to the $m_Z$-adic topology.  We therefore consider $\msf{U}$ as a complete linear topological algebra.  By a ``Hopf algebra" in this case we mean a Hopf algebra in the symmetric category of complete, linear topological, vector spaces.  This simply means that the comultiplication on $\msf{U}$ is a map to the completed tensor product $\Delta:\msf{U}\to \msf{U}\hat{\ot}\msf{U}$.
\par

To distinguish between the finite type and formal setting we generally employ a serif font $\msf{U}$ when we intend for our deformation to be complete.  We note that the formal situation is preferable simply because the parametrizing subalgebra $Z$ is local in this case.  As we explain below, any finite type integration $U\to \msf{u}$ determines a formal integration $\msf{U}\to \msf{u}$.  Unless confusion will arise, we drop the modifier ``formal" in our presentation and speak only of integrable Hopf algebras and (implied, formally smooth) integrations $\msf{U}\to \msf{u}$.
\par

Now, given a deformation $U\to \msf{u}$ parameterized by a smooth, finite type, $k$-algebra $Z$, one produces a corresponding formal deformation by completing at the distinguished maximal ideal $m_Z$,
\[
\msf{U}\to \msf{u},\ \ \msf{U}:=\hat{Z}\ot_Z U.
\]
Since $m_ZU$ is the kernel of the Hopf map $U\to \msf{u}$ it is a Hopf ideal in $U$, and the completion $\msf{U}=\varprojlim_n U/m^n_ZU$ is a complete, linear topological Hopf algebra.

\begin{lemma}
Suppose that $\msf{u}$ is (finite-dimensional and) integrable, $Z\to U\to \msf{u}$ is a smooth integration with $Z$ of finite type over $k$, and that $U$ is Noetherian and of finite global dimension.  Then the completion $\hat{Z}\ot_Z U$ is also Noetherian and of finite global dimension.
\end{lemma}

\begin{proof}
Take $\msf{U}=\hat{Z}\ot_ZU$ the completion.  Noetherianity of $\msf{U}$ follows from the fact that $\msf{U}$ is finite over the Noetherian algebra $\hat{Z}$.  We have that $\msf{U}$ is semi-local, and all of the simples are restricted from simples over $\msf{u}$.  So $\msf{U}/\operatorname{Jac}(\msf{U})=\msf{u}/\operatorname{Jac}(\msf{u})$ is a finite sum of finite-dimensional modules which are annihilated by the augmentation ideal of $\hat{Z}$.  Take $\Lambda=\msf{U}/\operatorname{Jac}(\msf{U})$ the sum of the simples.
\par

By applying the faithful and exact functor $\hat{Z}\ot_Z-$ to any flat resolution of $\Lambda$ over $U$ we obtain a resolution over $\msf{U}$, and so find that $\operatorname{Tor}^\msf{U}_\ast(\Lambda,-)=\operatorname{Tor}^U_\ast(\Lambda,-)$.  By our assumption that $U$ is of finite global dimension, we have $\operatorname{Tor}^U_{>d}(\Lambda,-)=0$ for $d=\rm{gldim}(U)$, and so observe $\operatorname{Tor}^\msf{U}_{>d}(\Lambda,-)=0$.
\par

By considering minimal resolutions of modules over $\msf{U}$, we see that the projective dimension of any finitely generated $\msf{U}$-module $M$ is the minimal $d(M)$ so that $\operatorname{Tor}^\msf{U}_{>d(M)}(\Lambda,M)=0$.  So all finitely generated $\msf{U}$-modules are of projective dimension $\leq d$.  It follows that $\msf{U}$ is of global dimension $\leq d$, and in particular of finite global dimension~\cite[Theorem 4.1.2]{weibel95}.
\end{proof}


\section{Basics for deformations of associative algebras}
\label{sect:defos_basics}

We consider for the moment deformations of associative algebras, without any assumed Hopf structure.  Below we explain how such deformations provide specific approaches to cohomology which leverage the deforming algebra as a realization of its corresponding classes in Hochschild cohomology.  After providing a relatively general analysis of cohomology, in the present section and Section~\ref{sect:supports}, we give the necessary implications for Hopf algebras in Section~\ref{sect:hopf_supp}.
\par

The materials of the present section are basically a review of work of Bezrukavnikov and Ginzburg~\cite{bezrukavnikovginzburg07}, and provide the technical foundations of our study.  We employ standard constructions for dg modules, which can be found throughout the literature (e.g.~\cite[Chapter III]{krizmay95} or~\cite{drinfeld04}).

\subsection{The setup}
\label{sect:setup}

In Sections~\ref{sect:defos_basics}--\ref{sect:hopf_supp} we work with a deformation sequence $Z\to Q\to R$ satisfying the following:
\begin{itemize}
\item $Z$ is complete local, formally smooth, and essentially of finite type, i.e.\ isomorphic to $k\b{y_1,\dots,y_n}$ in coordinates.
\item $Q$ is finite and flat over $Z$.
\item The reduction $k\ot_Z Q\cong R$ is finite-dimensional.
\end{itemize}

The third point here is obviously a consequence of the second.  Note that in this situation $Q$ is also Noetherian, since it is finite over the Noetherian subalgebra $Z$.  We always let $1\in \Spf(Z)$ denote the distinguished (and unique) closed point in $\Spf(Z)$ at which the fiber of $Q$ is identified with $R$.  Here $Q$ and $R$ are not assumed to be Hopf algebras.

\begin{remark}
One can replace the finiteness conditions in the above setup with the assumption that $Q$ is finite over its center.  Specifically, one should consider a deformation $Q$ which is finite over a central subalgebra $T$ such that $T$ contains $Z$, is complete local, and is essentially of finite type over $k$ (cf.\ Remark \ref{rem:1586}).  Of course, in this more general setting one must augment the results of Sections~\ref{sect:defos_n_fin}--\ref{sect:hopf_supp} in the expected manner.  For example, in the statement of Theorem~\ref{thm:fin2} one should consider instead finiteness over $T\ot B_Z$ and $T\ot A_Z$.
\end{remark}

\subsection{Dg algebras, dg modules and semi-projective resolutions}

Consider a dg algebra $K$ over $T$, where $T$ is a commutative algebra.  (Usually for us $T$ will be $Z$, the parametrizing subalgebra of \ref{sect:setup}.)  We consider $K$-dgmod the category of left dg modules over $K$, and let $D(K)$ denote the corresponding derived category, $D(K)=(K\text{-dgmod})[{\rm quis}^{-1}]$.  When $T$ is Noetherian, which is always the case for us, we take
\[
D_{coh}(K)=\left\{\begin{array}{c}\text{The full subcategory of dg modules in $D(K)$}\\ \text{with finitely generated cohomology over $T$}
\end{array}\right\}
\]
When $T=k$ we write $D_{fin}(K)$ for $D_{coh}(K)$, the category of dg modules with finite-dimensional cohomology.  Note that when $K$ is a finite-dimensional algebra $D_{fin}(K)$ is just the bounded derived category of finite $K$-modules.
\par
 
We say that a dg $K$-module $M$ is \emph{semi-projective} if it has a dg filtration 
\[
0 \subset F_0M \subseteq F_1M \subseteq \ldots \subseteq M
\] 
such that each $F_nM/F_{n-1}M$ is a direct summand of a direct sum of shifts of $K$.  (See for example \cite[\S3.1]{keller94}, where such modules are referred to as having \emph{property ({\rm P})}.)  We call a semi-projective dg module \emph{semi-free} if each subquotient $F_nM/F_{n-1}M$ is exactly a sum of shifts of $K$.  For a dg module $N$, we call a quasi-isomorphism $M \overset{\sim}\to N$ from a semi-projective dg module $M$ a {\it semi-projective resolution} of $N$.
As shown, for example, in \cite[Theorem 3.1]{keller94} or \cite[Lemma 20.4]{stacks}, such semi-projective resolutions always exist.  Following the usual procedures, we derive tensor products and Hom functors by employing semi-projective resolutions.
\par

For $M$ a dg module (over $k$) we let $M^\ast$ denote the corresponding dual dg module $M^\ast=\Hom_k(M,k)$.  Here $\Hom_k$ is the standard Hom complex, so that the underlying graded space for $M^\ast$ is the sum $\oplus_{n\in \mbb{Z}}(M^{-n})^\ast$.

\subsection{Deformations, Koszul resolutions, and the Koszul dual pair $(A_Z,B_Z)$}
\label{sect:koszul}

Consider $Z\to Q\to R$ a sequence as in \ref{sect:setup}.  A choice of coordinates for $Z$ is equivalent to a choice of section $\tau:m_Z/m_Z^2\to m_Z$ from the cotangent space at $1$.  From $\tau$ we can explicitly define the Koszul complex $K_Z$, which is a dg algebra resolving $k$ over $Z$.  Specifically, we take $K_Z$ to be the dg $Z$-algebra $\wedge^\ast_Z\Omega_Z$ with differential given on the generators $\Omega_Z$ by $d_{y_i}\mapsto y_i$.  Here the $\{y_i\}_{i=1}^n$ is any choice of basis for the image of $\tau$ in $m_Z$.  The quasi-isomorphism $K_Z\overset{\sim}\to k$ is given by the augmentation on $Z$ in degree $0$.

\begin{remark}
Different choices of sections $\tau$ result in resolutions which are isomorphic, as dg $Z$-algebras, but not literally equal.  However, one can check directly that the definition of $K_Z$ is completely independent the choice of basis $\{y_i\}_{i=1}^n$.  We employ the basis $\{y_i\}_i$ only for clarity of presentation.
\end{remark}

We change base along the extension $Z\to Q$ to obtain, from $K_Z$, a complex
\[
K_Q:=Q\ot_{Z}K_Z
\]
which provides a resolution of $R$ by a dg algebra which is finite and flat over $Z$: 
\[
K_Q\overset{\sim}\to Q\ot_Z k=R.
\]
We refer to $K_Q$ as the \emph{Koszul resolution} of $R$ associated to the deformation $Q\to R$.

\begin{lemma}
The resolution $K_Q$ is finite and free over $Q$.
\end{lemma}

\begin{proof}
This follows from the fact that $K_Z$ is finite and free over $Z$.
\end{proof}

Given a deformation $Q\to R$ as above, we make repeated use of the (graded) commutative dg algebras
\begin{equation}\label{eq:AZ}
B_Z:=\Sym(\Sigma (m_Z/m_Z^2))\ \ \text{and}\ \ A_Z:=\Sym(\Sigma^{-2}(m_Z/m_Z^2)^\ast).
\end{equation}
Since we are in the dg context $B_Z$ is, as an associative algebra, actually an {\it exterior} algebra.  The two algebras $B_Z$ and $A_Z$ are generated by the shifted cotangent space (in degree -1) and tangent space (in degree 2) of $\Spf(Z)$ at $1$, respectively, and can therefore be viewed as functions on corresponding linear dg schemes.
\par

An important point about these two algebras is that they are Koszul dual, in the usual dg sense. In particular,  $\RHom_{B_Z}(k,k) \cong A_Z$ and $\RHom_{A_Z}(k,k) \cong B_Z$ 
(see, for example, \ \cite{goreskykottwitzmacpherson98} \cite[Section 3.3]{arkhipovbezrukavnikovginzburg04}).

\subsection{Koszul resolutions and a map to the center of $D_{fin}(R)$ \cite{bezrukavnikovginzburg07}}
\label{sect:w/e}

We fix a deformation $Q\to R$ as in Section~\ref{sect:setup}.  Since the Koszul resolution $K_Q\overset{\sim}\to R$ is a quasi-isomorphism of dg algebras, reduction and restriction provide mutually inverse equivalences of triangulated categories
\begin{equation}\label{eq:equiv}
k\ot_{K_Z}^{\rm L}-:D(K_Q)\overset{\sim}\longrightarrow D(R),\ \ \res: D(R)\overset{\sim}\longrightarrow D(K_Q).
\end{equation}
These equivalences restrict to equivalences on the full subcategories of dg modules with coherent (finitely generated) cohomology over $Z$ and $k$ respectively,
\[
k\ot_{K_Z}^{\rm L}-:D_{coh}(K_Q)\overset{\sim}\longrightarrow D_{fin}(R),\ \ \res: D_{fin}(R)\overset{\sim}\longrightarrow D_{coh}(K_Q).
\]
Note that $D_{coh}(K_Q)$ and $D_{fin}(R)$ can be identified with the respective derived categories of finitely generated dg modules.

Consider now dg modules $M$ and $N$ over $K_Q$, and the algebra of extensions
\[
\Ext^\ast_{K_Q}(M,N)=\oplus_{i\in\mbb{Z}}\Hom^i_{D(K_Q)}(M,\Sigma^i N).
\]
Consider also the Koszul dual dg algebras $B_Z$ and $A_Z$ of~\eqref{eq:AZ}.
\par

Let us explicitly choose coordinates $\tau:m_Z/m_Z^2\to m_Z$ here, and consider the Koszul complex $K_Z$ defined precisely as in Section \ref{sect:koszul}.  We enumerate a basis $\{y_i\}_{i=1}^n$ for the image of the section $\tau$, and have the corresponding basis $\{\bar{y}_i\}_{i=1}^n$ of $m_Z/m_Z^2$.  Recall that the algebra $B_Z$ is generated by the cotangent space $T_1^\ast\Spf(Z)=m_Z/m_Z^2$, so that the basis $\{\bar{y}_i\}_i$ provides a set of generators for this algebra as well.
\par

Consider the product $K_Z\ot_Z K_Q$, which is a dg algebra over $Z$.  As explained in~\cite[Lemma 2.4.2]{bezrukavnikovginzburg07}, one has a flat extension of dg algebras $j_\tau:B_Z\to K_Z\ot_Z K_Q$ defined as follows: by definition $K_Z= \wedge^\ast_Z\Omega_Z$, with differential $d_{y_i}\mapsto y_i$, and we take
\begin{equation}\label{eq:BZKZ}
j_{\tau}:B_Z\to K_Z\ot_Z K_Q,\ \ j_\tau(\bar{y}_i)=d_{y_i}\ot 1-1\ot d_{y_i}.
\end{equation}
\par

We may view $K_Z\ot_Z K_Q=K_Q\ot_{Q} K_Q$ as a $K_Q$-bimodule, with commuting $B_Z$-action provided by the aforementioned algebra map.  Flatness of $K_Z\ot_Z K_Q$ over $B_Z$ implies that the induction functor $(K_Z\ot_Z K_Q)\ot_{B_Z}-$ from dg $B_Z$-modules to dg $K_Q$-bimodules is exact, and hence is derived simply as
\begin{equation}
\label{eq:hoch}
(K_Z\ot_Z K_Q)\ot^{\rm L}_{B_Z}-=(K_Z\ot_Z K_Q)\ot_{B_Z}-:D(B_Z)\to D(K_Q\ot K_Q^{op})
\end{equation} 
One calculates directly $(K_Z\ot_Z K_Q)\ot_{B_Z}k = Q \ot_Z (K_Z \ot K_Z) \ot_{B_Z} k  = Q \ot_Z K_Z  = K_Q$  (where the middle equality is by \cite[Lemma 2.4.2]{bezrukavnikovginzburg07}) to see that the above functor induces a map to Hochschild cohomology
\[
i_\tau:A_Z=\Ext_{B_Z}^\ast(k,k)\to H\!H^\ast(K_Q).
\]

\begin{remark} 
One is free to view the functor $(K_Z\ot_Z K_Q)\ot_{B_Z}-$ as having image in the category of dg $K_Z\ot_Z K_Q^{op}$-modules, or in the category of dg $K_Q$-bimodules.  This depends on whether one views $K_Z\ot_Z K_Q$ as a $Z$-central $(K_Z,K_Q)$-bimodule or a $K_Q$-bimodule. One advantage of the $K_Q$-bimodule perspective is that we can then use the quasi-isomorphism $K_Q \to R$ to pass from 
$D(K_Q\ot K_Q^{op})$ to $D(R \ot R^{op})$ and, hence, to identify the Hochschild cohomologies $H\!H^\ast(K_Q)$ and $H\!H^\ast(R)$ as we do below. 
\end{remark}

Recall that the center $Z(D_{coh}(K_Q))$ of the derived category $D_{coh}(K_Q)$ is the graded algebra of natural transformations from the identity functor on $D_{coh}(K_Q)$ to its shifts
\[
Z(D_{coh}(K_Q))=\oplus_{n\in\mbb{Z}} \Hom_{\rm Fun}(id,\Sigma^n),
\]
and recall also the map from Hochschild cohomology to the center.  This map takes a class in $H\!H^\ast(K_Q)$, which is some map $f:K_Q\to \Sigma^n K_Q$ in the derived category of bimodules, to the transformation $f\ot_{K_Q}-:id\to \Sigma^n$.  So from  $i_\tau$ we obtain a map to the center $\iota_\tau:A_Z\to Z(D_{coh}(K_Q))$, i.e.\ an action of $A_Z$ on the derived category of $K_Q$-modules.
\par

We elaborate on this final point.  From the maps $Z(D_{coh}(K_Q))\to \Ext^\ast_{K_Q}(M,M)$, $\xi\mapsto \xi_M$, we obtain a natural action of $A_Z$ on arbitrary objects $M$ in the derived category via restricting along $\iota_\tau$.  Through these actions on objects we obtain, in principle, \emph{two} actions of $A_Z$ on each graded extension group $\Ext^\ast_{K_Q}(M,N)$, one through $M$ and one through $N$.  Naturality tells us that these two actions agree, and in the case $M=N$ we find that the corresponding algebra map $A_Z\to \Ext^\ast_{K_Q}(M,M)$ has central image.

\begin{theorem}[\cite{bezrukavnikovginzburg07}]
\label{thm:bg}
The algebra map $i_\tau:A_Z\to H\!H^\ast(K_Q)$ is independent of the choice of section $\tau:m_Z/m_Z^2\to m_Z$.  In particular, the restriction to the generators $i_\tau:(A_Z)^2=(m_Z/m_Z^2)^\ast\to H\!H^2(K_Q)\cong H\!H^2(R)$ is the usual deformation map of~\cite{gerstenhaber64,bezrukavnikovginzburg07,friedlandernegron18}.
\end{theorem}

Since the $i_\tau$ are independent of the choice of section $\tau$, we can speak of {\it the} algebra map $A_Z\to H\!H^\ast(R)$ determined by the deformation $Q\to R$, parametrized by $\Spf(Z)$, and \emph{the} corresponding action of $A_Z$ on $D_{coh}(K_Q)\cong D_{fin}(R)$.

\begin{definition}\label{def:456}
We fix $i_Q:A_Z\to H\!H^\ast(R)$ the algebra map associated to a given deformation sequence $Z\to Q\to R$, and view all extensions $\Ext^\ast_R(M,N)$ as $A_Z$-modules via $i_Q$.
\end{definition}


\section{Deformations and finiteness conditions on cohomology}
\label{sect:defos_n_fin}

For this section we fix a deformation $Q\to R$ as in~\ref{sect:setup}.  We provide a general finite generation result at Theorem~\ref{thm:fin2} (see also Lemma~\ref{lem:twtt}).  We then apply Theorem~\ref{thm:fin2} to recover, in a uniform manner, finite generation results of~\cite{friedlanderparshall86II,ginzburgkumar93,friedlandernegron18} in Subsection~\ref{sect:hopf_fg}.  In particular, we find that any integrable Hopf algebra $\msf{u}$ has finitely generated cohomology.

The key observation of this section is that $\RHom_R(V,W)$ with its $A_Z$-action as described in \ref{thm:bg} and $
\RHom_Q(V,W)$ with a natural $B_Z$-action are related by a dg version of Koszul duality for $A_Z$ and $B_Z$, see Lemma~\ref{lem:twtt} as the main technical result.
Since Koszul duality preserves finiteness conditions (see \cite{avramovetal10,avramovetal10corrigendum}),  it allows us to translate between finite generation of $\Ext^\ast_R(V,W)$ as $A_Z$-module and finite generation of $\Ext^\ast_Q(V,W)$ as $B_Z$-module, as done in Theorem~\ref{thm:fin2}. Finally, for our applications to integrable Hopf algebras, one observes that finite generation of $\Ext^\ast_Q(V,W)$ as a $B_Z$-module is equivalent to finite-dimensionality which, in turn, follows from the assumption that $Q$ has final global dimension.

\subsection{Technicalities}
\label{sect:technology}

Consider dual bases $\{y^i\}_i$ and $\{y_i\}_i$ for the tangent space and cotangent space of $\operatorname{Spf}(Z)$, and consider $A_Z\ot B_Z$ as a dg algebra with vanishing differential.  We define the functor
\begin{equation}\label{eq:t}
\begin{array}{l}
A_Z\ot^t-:B_Z\text{-dgmod}\to A_Z\ot B_Z\text{-dgmod},\\
\hspace{5mm}A_Z\ot^tM:=(A_Z\ot M,1\ot d_M+\sum_i y^i\ot y_i).
\end{array}
\end{equation}
The functor $A_Z\ot^t-$ respects quasi-isomorphisms on the category $B_Z\text{-dgmod}^+$ of bounded below dg $B$-modules, as it is identified with the functor $\Hom_{B_Z}(F,-)$, where $F=(B_Z\ot(A_Z)^\ast,d_F)$ is the Koszul resolution of $k$ over $B_Z$.  So we have the derivation $A_Z\ot^t-={\rm L}(A_Z\ot^t-)$ on $D^+(B_Z)$ (see, for example, \cite[Section 4]{negron17}, \cite{avramovetal10}).   

Recall that a dg module $N$ over a dg algebra $A$ is called $K$-flat if the functor $-\ot_AN$ preserves acyclic objects \cite{spaltenstein88}.  For $K$-flat $N$ the derived and underived tensor products can be identified $-\ot^{\rm L}_AN=-\ot_AN$.

\begin{lemma}\label{lem:529}
For any dg $B_Z$-module $M$, the twisted product $A_Z\ot^t M$ is $K$-flat over $A_Z$.
\end{lemma}

\begin{proof}
Consider an acyclic right dg $A_Z$-module $L$.  We have
\[
L\ot_{A_Z}(A_Z\ot^t M)=L\ot^tM,
\]
where the latter object is the linear product $L\ot M$ equipped with the differential $d_L\ot 1+1\ot d_M+\sum_i y_i\ot y^i$.  This complex has a bounded filtration provided by the $m_{B_Z}$-adic filtration on $M$, and the associated graded complex is the linear product
\[
\operatorname{gr}(L\ot^t M)=L\ot (\operatorname{gr}M).
\]
As any complex of $k$-vector spaces is $K$-flat over $k$,
we see that the associated graded complex for $L\ot^tM$ is acyclic.  By considering the spectral sequence associated to this filtration, it follows that $L\ot^tM$ is in fact acyclic.
\end{proof}

We recall also, from the materials of Section~\ref{sect:w/e}, that $\Ext_R\cong \Ext_{K_Q}$ naturally takes values in the category of dg $A_Z$-modules.  Furthermore, by resolving $R$-modules by $K_Q$-modules $M\to V$ which are projective over $Q$, we see that the functor $\Ext_Q$ from $R$-modules naturally takes values in the category of dg $B_Z$-modules.

\subsection{Finite generation results}

As we argue below, one can lift the natural actions of $A_Z$ on $\Ext^\ast_{K_Q}(M,N)$ described in Section~\ref{sect:w/e} the dg level, i.e.\ to an action of $A_Z$ on $\RHom_{K_Q}(M,N)$.  The action of $B_Z$ on $\Ext^\ast_{Q}(M,N)$ also lifts to the dg level, at least when $N$ is replaced with an (quasi-isomorphic) dg $R$-module.  After translating from $K_Q$ to $R$, we establish the following.

\begin{lemma}\label{lem:twtt} 
\begin{enumerate} 
\item Let $V$ and $W$ be in $D_{fin}(R)$. Then $\RHom_R(V, W)$ has a natural structure of a dg $A_Z \ot B_Z$-module which lifts the action of $A_Z$ on $\Ext^\ast_R(V, W)$ of Definition~\ref{def:456}. 

\item There is a natural isomorphism
\[
\RHom_R\cong A_Z\ot^t\RHom_Q
\]
of functors from $D_{fin}(R)^{op}\times D_{fin}(R)$ to $D(A_Z\ot B_Z)$, where $\ot^t$ is as defined in \eqref{eq:t}. 
\end{enumerate} 
\end{lemma}

The proof of Lemma~\ref{lem:twtt} is given in Section~\ref{sect:twtt}.  In Section~\ref{sect:fin2} we prove the following result.

\begin{theorem}[{cf.~\cite[Theorem 4.2]{avramovgasharovpeeva97}}]\label{thm:fin2}
Consider a deformation sequence $Z\to Q\to R$ as in Section~\ref{sect:setup}.  For finite $R$-modules $V$ and $W$ the following are equivalent:
\begin{enumerate}
\item[(a)] $\Ext^\ast_Q(V,W)$ is finite over $B_Z$.
\item[(b)] $\Ext^\ast_R(V,W)$ is finite over $A_Z$.
\end{enumerate}
\end{theorem}

Note that finiteness of $\Ext^\ast_Q(V,W)$ over $B_Z$ is equivalent to boundedness of $\Ext^\ast_Q(V,W)$, since $B_Z$ is concentrated in finitely many degrees.  Theorem~\ref{thm:fin2} follows by a basic understanding of the category of dg modules over $A_Z$ (Lemma~\ref{lem:315}) and a direct application of Lemma~\ref{lem:twtt}.
\par

We apply Theorem~\ref{thm:fin2} to recover a number of essential finite generation results for Hopf cohomology, in Section~\ref{sect:hopf_fg}.

\subsection{Interpretation via noncommutative complete intersections}

One can see Theorem~\ref{thm:fin2} as a noncommutative variant of results of Gulliksen and Avramov-Gasharov-Peeva~\cite{gulliksen74,avramovgasharovpeeva97} which concern the cohomology of local complete intersections (see also~\cite{kirkmankuzzhang15}).  From the complete intersection perspective~\cite{gulliksen74,eisenbud80}, $A_Z$ is the ``algebra of cohomological operators" deduced from the ``noncommutative complete intersection" $Q\to R$ (cf.~\cite{kirkmankuzzhang15}).

\subsection{Proof of Lemma~\ref{lem:twtt}}
\label{sect:twtt}

\begin{lemma}\label{lem:511}
If $M$ is a dg $K_Q$-module which is semi-projective over $Q$, then $(K_Z\ot_Z K_Z)\ot_{K_Z}M$ is semi-projective over $K_Q$.
\end{lemma}

To be clear, $K_Q$ acts on $(K_Z\ot_Z K_Z)\ot_{K_Z}M$ via the $K_Z$ action on the left factor of $K_Z\ot_Z K_Z$ and the $Q$-action on $M$.

\begin{proof}
We have by definition a filtration $M=\cup_{i\geq 0}F_iM$ by dg $Q$-modules so that each quotient $F_iM/F_{i-1}M$ is a summand of a free dg module over $Q$ (recall that we consider $Q$ as a dg algebra concentrated in degree $0$).  Since $K_Z$ is a bounded complex of flat $Z$-modules, the functor 
\[
(K_Z\ot_Z K_Z)\ot_{K_Z}-=K_Z\ot_Z-: D(K_Q) \to D(K_Q)
\] 
is exact. Hence, 
\[
K_Z\ot_Z(F_iM/F_{i-1}M)\cong \frac{K_Z \ot_Z F_iM}{K_Z \ot_Z F_{i-1}M}. 
\] 
Since each $F_iM/F_{i-1}M$ is a summand of a free dg $Q$-module, it follows that  $K_Z\ot_Z(F_iM/F_{i-1}M)$ is a summand of a free dg $K_Q$-module (that is, direct sum of shifts of $K_Q$). We consider the filtration $\cup_{i\geq 0} K_Z\ot_Z F_iM$ on $(K_Z\ot_Z K_Z)\ot_{K_Z} M$ to observe a semi-projective structure of this dg module over $K_Q$. 
\end{proof}

We can now prove our essential lemma.

\begin{proof}[Proof/Construction of Lemma~\ref{lem:twtt}] Let $V$ and $W$ be finite complexes of $R$-modules which we consider as dg $K_Q$-modules via the restriction functor \eqref{eq:equiv}. We first construct an action of $A_Z\ot B_Z$ on $\RHom_{K_Q}(V,W)$ which lifts the action of $A_Z$ on $\Ext_R(V,W) \cong \Ext_{K_Q}(V,W)$ of Definition~\ref{def:456} to the dg level. 

Choose a resolution $M\to V$ by a dg $K_Q$-module which is bounded above and projective over $Q$ in each degree.  (We may assume the resolution $M$ is bounded when $Q$ is of finite global dimension.)  Let $F=B_Z\ot^t A_Z^\ast$ be the Koszul resolution of $k$ over $B_Z$, and consider the induction
\[
F^{\rm ind}=(K_Z\ot_Z K_Z)\ot_{B_Z}F
\]
along the dg map~\eqref{eq:BZKZ}.  Since $(K_Z\ot_Z K_Z)\ot_{B_Z}k \cong K_Z$  (see \cite[Lemma 2.4.2]{bezrukavnikovginzburg07}), and $F$ is the Koszul resolution of $k$ over $B_Z$, we conclude that 
\[F^{\rm ind} \overset{\sim}\to K_Z\] 
is a semi-free resolution of $K_Z$ over $K_Z\ot_Z K_Z$. Therefore,  $F^{\rm ind}\ot_{K_Z}M$ is a semi-projective resolution of $V$ over $K_Q$, by Lemma~\ref{lem:511}.  We have specifically the quasi-isomorphism
\[
F^{\rm ind}\ot_{K_Z} M\overset{\sim}\to K_Z\ot_{K_Z} M=M\overset{\sim}\to V.
\]
\par

Since $F = B_Z\ot^t A_Z^\ast$ is a dg $(B_Z,A_Z)$-bimodule, the induction $F^{\rm ind}$ admits commuting actions of $K_Z\ot_Z K_Z$ and $A_Z$, and the product $F^{\rm ind}\ot_{K_Z}M$ has commuting actions of $K_Z$ and $A_Z$.  Specifically, we employ the $K_Z$ action on $M$, and the $A_Z$-action on $F^{\rm ind}$ to obtain the action of $A_Z\ot K_Z$ on $F^{\rm ind}\ot_{K_Z} M$.   The action of $A_Z\ot K_Z$ on $F^{\rm ind}\ot_{K_Z}M$ induces a natural action of $A_Z\ot B_Z=k\ot_Z(A_Z\ot K_Z)$ on the complex
\[
\RHom_{K_Q}(V,W)=\Hom_{K_Q}(F^{\rm ind}\ot_{K_Z}M,W).
\]
A direct comparison verifies that this dg-level action lifts the action of Definition~\ref{def:456}.  
\par

We have the sequence of natural isomorphisms
\[
\begin{array}{ll}
\RHom_{K_Q}(V,W)=\Hom_{K_Q}(F^{\rm ind}\ot_{K_Z}M,W)\\
\hspace{5mm}=\Hom_{K_Z\ot_ZK_Z}(F^{\rm ind},\Hom_Q(M,W)) &\text{(adjunction)}\\
\hspace{5mm}=\Hom_{K_Z\ot_ZK_Z}\big((K_Z\ot_ZK_Z)\ot_{B_Z}F,\Hom_Q(M,W)\big)\\
\hspace{5mm}=\Hom_{B_Z}(F,\Hom_Q(M,W)). & \text{(adjunction)}\\
\hspace{5mm}=A_Z\ot^t\Hom_Q(M,W).
\end{array}
\]
The final expression is $A_Z\ot^t\RHom_Q(V,W)$, and we obtain the claimed result.
\end{proof}

\subsection{Proof of Theorem~\ref{thm:fin2}}
\label{sect:fin2}

We first give a basic lemma on dg modules over regular dg algebras.

\begin{lemma}\label{lem:315}
Consider a non-negatively graded dg algebra $S$ with vanishing differential.  Suppose also that $S$ is Noetherian and has finite global dimension, as an associative algebra.  For any dg $S$-module $M$ the following are equivalent
\begin{enumerate}
\item The cohomology $H^\ast(M)$ is finitely generated over $S$.
\item $M$ admits a resolution $M'\to M$ by a finitely generated semi-projective dg $S$-module $M'$.
\end{enumerate}
\end{lemma}

\begin{proof}
The existence of such a resolution implies coherence of cohomology, since $S$ is Noetherian.  Suppose now that the cohomology $H^\ast(M)$ is finitely generated.  We prove the existence of the appropriate resolution $M$ by induction on the projective dimension of $H^\ast(M)$, the projective dimension $0$ case being clear.  Suppose now that $H^\ast(M)$ has projective dimension $r$, and that the desired result holds for bounded below dg modules with cohomology of projective dimension $<r$.  Take $P\to H^\ast(M)$ a surjection from a finitely generated (graded) projective $S$-module, and $\phi:P\to Z(M)\subset M$ any lift thereof.  Let $C$ denote the mapping cone of this lift.
\par

By a spectral sequence argument one sees that
\[
H^\ast(C)\cong H^\ast(\mrm{cone}(P\to H^\ast(M)))=\Omega^1 H^\ast(M),
\]
and hence that $H^\ast(C)$ is finite over $S$ and of projective dimension $r-1$.  Let $\psi:M''\to \Sigma^{-1}C$ be a semi-projective resolution by a finite dg $S$-module.  Then the mapping cone $M'=\mrm{cone}(M''\to P)$ of the composite $M''\to \Sigma^{-1}C\to P$, along with the morphism of dg modules $M'\to M$ induced by the graded $S$-module maps $\bar{\psi}:M''\to M$ and $\phi:P\to M$, provides the desired resolution.
\end{proof}

 \begin{remark} The lemma essentially identifies the categories $D_{coh}(S)$ and $\thick_S(S)$ (as defined in, for example, \cite{avramovetal10}) allowing us to use the conditions interchangeably. 
 \end{remark}

Recall that the symmetric algebra $A_Z$ is of (finite) global dimension $\dim(m_Z/m_Z^2)$.  So the above result applies to dg modules over $A_Z$.  We now offer a proof of our theorem.

\begin{proof}[Proof of Theorem~\ref{thm:fin2}]
Suppose that $\RHom_Q(V,W)$ has finitely generated cohomology over $B_Z$.  Then $\RHom_Q(V,W)$ admits a quasi-isomorphism $N\overset{\sim}\to \RHom_Q(V,W)$ from a bounded dg $B_Z$-module.  We have then
\[
\RHom_R(V,W)=A_Z\ot^t\RHom_Q(V,W)\overset{\sim}\leftarrow A_Z\ot^t N.
\]
The dg module on the right is coherent over $A_Z$, and thus has coherent cohomology.
\par

Suppose now that $\Ext^\ast_R(V,W)$ is finite over $A_Z$.  Then $\RHom^\ast_R(V,W)$ admits a resolution $M\to \RHom_R(V,W)$ by a semi-projective, coherent, dg module over $A_Z$, by Lemma~\ref{lem:315}.  It follows that the cohomology of the fiber
\[
k\ot^{\rm L}_{A_Z}\RHom_R(V,W)\cong k\ot_{A_Z}M
\]
is finite-dimensional.  We have an identification
\[
\begin{array}{rll}
k\ot^{\rm L}_{A_Z}\RHom_R(V,W)&=k\ot^{\rm L}_{A_Z}(A_Z\ot^t\RHom_Q(V,W)) & (\text{Lemma \ref{lem:twtt}})\\
&=k\ot_{A_Z}(A_Z\ot^t\RHom_Q(V,W)) & (\text{Lemma \ref{lem:529}})\\
&=\RHom_Q(V,W).
\end{array}
\]
Thus
\[
\Ext^\ast_Q(V,W)\cong H^\ast(\RHom_Q(V,W))\cong H^\ast(k\ot^{\rm L}_{A_Z} \RHom_R(V,W))
\]
is finite-dimensional, and subsequently finite over $B_Z$.
\end{proof}

As a corollary to Theorem~\ref{thm:fin2} we have the following

\begin{corollary}\label{cor:finH}
Suppose $Q$ is of finite global dimension, and that $Z\to Q\to R$ is a deformation sequence as in Section~\ref{sect:setup}.  Then for any finite $R$-modules $V$ and $W$, $\Ext^\ast_R(V,W)$ is a finite $A_Z$-module.
\end{corollary}

\subsection{Implications for Hopf cohomology}
\label{sect:hopf_fg}

Suppose that $\msf{u}$ is a finite-dimensional Hopf algebra which admits a smooth integration $\msf{U}\to \msf{u}$, as in the examples of Section~\ref{sect:examples}.  We have the algebra map
\begin{equation}\label{eq:627}
A_Z\overset{i_\msf{U}}\longrightarrow H\!H^\ast(\msf{u})\overset{-\ot_\msf{u}^{\rm L} k}\longrightarrow \Ext^\ast_\msf{u}(k,k).
\end{equation}

\begin{theorem}\label{thm:fg}
In the above situation,~\eqref{eq:627} is a finite algebra map, so that $\Ext_\msf{u}^\ast(k,k)$ is a finitely generated algebra.  Furthermore, for finite-dimensional $\msf{u}$-modules $V$ and $W$, $\Ext^\ast_\msf{u}(V,W)$ is a finite $\Ext_\msf{u}^\ast(k,k)$-module via the tensor action $W\ot-:\Ext^\ast_\msf{u}(k,k)\to \Ext^\ast_\msf{u}(W,W)$.
\end{theorem}

\begin{proof}
Corollary~\ref{cor:finH} says directly that $\Ext^\ast_\msf{u}(k,k)$ is finite over $A_Z$ in this case.  For modules, we have $\Ext^\ast_\msf{u}(V,W)=\Ext^\ast_\msf{u}(W^\ast\ot V,k)$, as an $\Ext^\ast_\msf{u}(k,k)$-module.  So it suffices to assume $W=k$.  But in this case the actions of $A_Z$ on $\Ext^\ast_\msf{u}(V,k)$ factors through $\Ext^\ast_\msf{u}(k,k)$.  So finiteness over $A_Z$, which holds by Theorem~\ref{thm:fin2}, implies finiteness over $\Ext^\ast_{\msf{u}}(k,k)$.
\end{proof}

At this point we recover a number of finite generation results, uniformly, via application of Theorem~\ref{thm:fg}.

\begin{corollary}\label{cor:fg_revisit}
The following Hopf algebras have finitely generated cohomology, in the strong sense described in the introduction:
\begin{enumerate}
\item The quantum group $u_q(\mfk{g})$~\cite{ginzburgkumar93,bnpp14} {\rm ($\operatorname{char}(k)=0$)}.
\item The quantum Borel $u_q(\mfk{b})$~\cite{ginzburgkumar93,bnpp14} {\rm ($\operatorname{char}(k)=0$)}.
\item Restricted enveloping algebras $u^{\rm res}(\mfk{g})$~\cite{friedlanderparshall86II} {\rm ($\operatorname{char}(k)<\infty$)}.
\item Algebras of functions $\O(\mcl{G})$ on arbitrary finite group schemes~\cite{waterhouse} {\rm ($\operatorname{char}(k)<\infty$)}.
\item Drinfeld doubles $D(\mcl{G})$, and relative doubles $D(\mcl{H}_r,\mcl{G})$, for $\mcl{G}$ a height $1$ connected group scheme which embeds normally in a smooth algebraic group $\mcl{H}$~\cite{friedlandernegron18} {\rm ($\operatorname{char}(k)<\infty$)}.
\end{enumerate}
\end{corollary}

Our finite generation results for quantum groups are stronger than those of~\cite{ginzburgkumar93}, as we allow $q$ to be of small (odd) order.  (Finite generation at arbitrary $q$ was already known, however, and can be found in the text~\cite{bnpp14}.)  Our results for doubles are also slightly stronger than those of \cite{friedlandernegron18}, in the height $1$ case, although \cite{friedlandernegron18} also addresses doubles of group schemes of height $>1$.

\section{A meditation on hypersurfaces}

We explain the particular implications of the results of Section~\ref{sect:defos_n_fin} in the cases in which $Q$ is (homologically) smooth, or of finite global dimension, and subsequently when $Q$ is a ``noncommutative hypersurface" in such a smooth algebra.
\par

In this section, and all sections that follow, by a perfect module over a Noetherian algebra we simply mean a finitely generated module which is of finite projective dimension.  This is equivalent to perfection of such a module in the corresponding derived category, and we adopt the derived language as it is still meaningful in a dg context.

\subsection{Smooth deformations, finiteness, and naturality}

Fix $Z\to Q\to R$ a deformation sequence as in Section~\ref{sect:setup} with $Q$ \emph{of finite global dimension}.  Corollary~\ref{cor:finH} tells us that each $\Ext^\ast_R(M,N)$ is a finite module over $A_Z$ in this case.  So we may view extensions as a bifunctor
\[
\Ext^\ast_R:(R\text{-mod}_{fg})^{op}\times (R\text{-mod}_{fg})\to A_Z\text{-dgmod}_{fg}.
\]
\par

Let $f\in m_Z$ be any function with nonvanishing reduction $\bar{f}\in m_Z/m_Z^2$.  We then have the diagram of deformation sequences
\[
\xymatrix{
Z\ar[r]\ar[d]^{\pi_Z} & Q\ar[r]\ar[d]^{\pi_Q} & R\ar[d]^=\\
Z/(f)\ar[r] & Q/(f)\ar[r] & R,
}
\]
where the $\pi_\#$ are the obvious projections.  The projection $Z\to Z/(f)$ also induces an inclusion of tangent spaces $(m_{(Z/f)}/m_{(Z/f)}^2)^\ast\to (m_Z/m_Z^2)^\ast$ and subsequent algebra inclusion
\[
A_{(Z/f)}\to A_Z.
\]
We note that the image of this inclusion depends only on the {\it class} of $f$ in the cotangent space, and not on the choice of $f$ itself.  Consider now the possibly non-commuting diagram
\begin{equation}\label{eq:310}
\xymatrix{
A_{(Z/f)}\ar[d]_{\rm incl}\ar[rr]^{i_{Q/f}} && H\!H^\ast(R)\ar[d]^=\\
A_Z\ar[rr]^{i_Q} && H\!H^\ast(R).
}
\end{equation}

\begin{lemma}\label{lem:diagram_commutes}
The diagram~\eqref{eq:310} does in fact commute.
\end{lemma}

We view the above diagram as a type of naturality property for our setup~\ref{sect:setup}.

\begin{proof}
Consider any deformation situation $Z'\to Q'\to R'$.  The image of a tangent vector $\xi\in (m/m^2)^\ast$ in Hochschild cohomology is the class of the deformation
\[
Q'\ot_{Z'}k[\epsilon_\xi],
\]
where the $k[\epsilon_\xi]$ is $k[\epsilon]=k[\epsilon]/(\epsilon^2)$ equipped with the algebra map $Z'\to k[\epsilon]$ which annihilates $m^2$ and is defined on the quotient $Z/m^2\to k[\epsilon]$ by the function $\epsilon\cdot\xi:m/m^2\to k\epsilon$.  Take $Q_f=Q/(f)$ and $Z_f=Z/(f)$.  Using this description, one finds that for $\xi\in (m_{Z_f}/m_{Z_f})^\ast\subset (m_Z/m_Z^2)^\ast$, i.e.\ for any function $\xi:m_Z/m^2_Z\to k$ which vanishes in $\bar{f}$, the projection $Q\to Q_f$ induces an isomorphism of deformations
\[
Q\ot_{Z}k[\epsilon_{\xi}]\cong Q_f\ot_{Z_f}k[\epsilon_\xi].
\]
That is to say, the diagram~\eqref{eq:310} commutes when restricted to the generators, and thus commutes on all of $A_{(Z/f)}$.
\end{proof}

\subsection{Perfection of modules over hypersurfaces}

Commutativity of~\eqref{eq:310} implies the following.

\begin{corollary}\label{cor:606}
Consider $f,g\in m_Z$ elements with equivalent, nonzero, reductions $\bar{f}=\bar{g}$ in $m_Z/m_Z^2$.  Then for finite $R$-modules $V$ and $W$,
\[
\Ext^{\ast}_{Q/(f)}(V,W)\text{ vanishes in high degree }\Leftrightarrow\ \Ext^{\ast}_{Q/(g)}(V,W)\text{ vanishes in high degree}.
\]
\end{corollary}

\begin{proof}
By considering $Q$ in Theorem~\ref{thm:fin2} to be either of the deformations $Q/(f)$ or $Q/(g)$ of $R$, the result is a consequence of Theorem~\ref{thm:fin2} and Lemma~\ref{lem:diagram_commutes}.
\end{proof}

Locality of $Z$ tells us that all of the finitely generated simples over $Q$ are restricted from simple $R$-modules along the quotient $Q\to R$.  We take $\Lambda=R/\operatorname{Jac}(R)=Q/\operatorname{Jac}(Q)$ to find, for any finitely generated module $M$ over $Q$,
\[
M\text{ is perfect over }Q\ \Leftrightarrow\ \Ext^{\gg0}_Q(M,\Lambda)=0.
\]
By taking $W=\Lambda$ in Corollary~\ref{cor:606} we therefore obtain

\begin{corollary}\label{cor:reductions}
Consider $f,g\in m_Z$ elements with equivalent, nonzero, reductions $\bar{f}=\bar{g}$ in $m_Z/m_Z^2$.  Then a finite $R$-module $V$ is perfect over $Q/(f)$ if and only if $V$ is perfect over $Q/(g)$.
\end{corollary}

\section{Cohomological supports and hypersurfaces}
\label{sect:supports}

We show that the cohomological support for $R$, defined via the action of $A_Z$ on extension $\Ext^\ast_R(V,W)$, can be identified with a certain hypersurface support.  The hypersurface support for a module $V$ is defined by considering (non-)perfection of $V$ over hypersurfaces $Q/(f)$, for varying $f\in m_Z$.  (See Definition~\ref{def:hyp_supp} and Corollary~\ref{cor:hopf_hyp} below.)  We then discuss consequences for support in a Hopf theoretic context in Section \ref{sect:hopf_supp}.  Throughout $Z\to Q\to R$ is a deformation sequence as in~\ref{sect:setup}.

\begin{remark}
The results of this section are noncommutative analogs of results of Avramov-Buchweitz~\cite{avramovbuchweitz00}.
\end{remark}

\subsection{Graded Nakayama}

The following graded version of Nakayama's lemma is standard.

\begin{lemma}\label{lem:Z_Nak}
Let $S$ be a $\mbb{Z}_{\geq 0}$-graded algebra, and $N$ be a bounded below graded $S$-module.  Then $N$ is $0$ if and only if $S_0\ot_S N=0$.
\end{lemma}

Note that we have imposed no finiteness conditions on $N$.

\begin{proof}
Suppose that $N$ is nonzero and that the fiber $S_0\ot_S N$ vanishes.  We have that the minimal degree $N_{\rm min}$ survives in the fiber $N_{\rm min}\subset S_0\ot_S N$.  Hence vanishing of the fiber implies $N_{\rm min}=0$, which is nonsense.
\end{proof}

\begin{corollary}\label{cor:Z_Nak}
For $S$ and $N$ as in Lemma~\ref{lem:Z_Nak}, $N$ is generated by any graded $S_0$-submodule $N_\omega\subset N$ which surjects onto the fiber $S_0\ot_S N$.  In particular, if the fiber $S_0\ot_S N$ is finite over $S_0$ then $N$ is finite over $S$.
\end{corollary}

\begin{proof}
Take such a lift $N_\omega\subset N$ and consider the graded $S$-module map $S\ot_{S_0}N_\omega\to N$ induced by the action of $S$ on $N$.  Let $M$ denote the cokernel, so that we have an exact sequence $S\ot_{S_0}N_\omega\to N\to M\to 0$.  By right exactness of $S_0\ot_S-$, we obtain an exact sequence $N_\omega\to S_0\ot_S N\to S_0\ot_S M\to 0$, from which we conclude that $S_0\ot_S M$ vanishes, and hence that $M$ vanishes.  So the result holds.
\end{proof}

\subsection{Supports!}
\label{sect:nchyp_supp}

In anticipation of the Hopf case, we now impose the following conditions on our deformation $Q\to R$:
\begin{enumerate}
\item[(A)] $Q$ is of finite global dimension and augmented.
\item[(B)]$Q\text{-mod}_{fg}$ comes equipped with a $Z$-linear endofunctor $\sigma$ such that, for any finite $R$-module $V$ and $f\in m_Z$, $V$ is perfect over $Q/(f)$ if and only if $\Ext^{\gg 0}_{Q/(f)}(\sigma V,k)=0$.
\end{enumerate}

Note that $Z$-linearity of $\sigma$ implies that $\sigma$ restricts to an endofunctor on each full subcategory $Q/IQ\text{-mod}_{fg}$, for any ideal $I\subset Z$, as this category is simply the subcategory of $Q$-modules for which the action $Z\to \End_Q(M,M)$ vanishes on $I$.

\begin{example}[Hopf deformations]\label{ex:comod_def}
Suppose that $\msf{u}$ is a finite-dimensional Hopf algebra with an integration $\msf{U}\to \msf{u}$.  Then for the sum of the simples $\Lambda=\msf{u}/\operatorname{Jac}(\msf{u})$ we may take $\sigma=\Lambda\ot-:\msf{U}\text{-mod}_{fg}\to \msf{U}\text{-mod}_{fg}$.
\end{example}

In the given context, each $\Ext^\ast_R(V,W)$ is finite over $A_Z$, by Corollary~\ref{cor:finH}.  We fix
\[
\mbb{P}:=\mbb{P}(m_Z/m_Z^2)=\Proj(A_Z).
\]
(In a more geometric notation, one could write $\mbb{P}=\mbb{P}\big(T_1^\ast\Spf(Z)\big)$.)  For a graded module $N$ over $A_Z$ we let $N^\sim$ denote the associated sheaf on $\mbb{P}$.
\par

For a deformation $Q\to R$ as above we define two cohomological supports as follows.

\begin{definition}
For $V$ finite over $R$, define the cohomological support in $\mbb{P}$ as
\[
\supp_{\mbb{P}}(V):=\Supp_{\mbb{P}}\Ext^\ast_R(V,\Lambda)^\sim,
\]
where $\Lambda=R/\operatorname{Jac}(R)$ is the sum of the simples for $R$.  We defined the cohomological $\sigma$-support as
\[
\supp_{\mbb{P}}^\sigma(V):=\Supp_{\mbb{P}}\Ext^\ast_R(\sigma V,k)^\sim.
\]
\end{definition}

\begin{remark}
The reader should note that these supports $\supp_\mbb{P}$ and $\supp^\sigma_\mbb{P}$ dependent on the choice of deformation $Q$ for $R$.
\end{remark}

For a closed point $c\in \mbb{P}$ we let $Q_c$ denote any quotient $Q/(f_c)$ where $f_c\in m_Z$ is any lift of an element in the corresponding line $c\subset m_Z/m_Z^2-\{0\}$.  Recall that the algebra $A_{Z/(f_c)}$ and its corresponding map to Hochschild cohomology are independent of the choice of lift $f_c$, by Lemma \ref{lem:diagram_commutes}.  We take
\begin{equation}\label{eq:Ac}
A_c:=A_{(Z/f_cZ)},\ \ \text{for any lift $f_c$ of }c\in \mbb{P}(m_Z/m_Z^2).
\end{equation}
By Corollary~\ref{cor:reductions}, for any other choice of lift $g_c$, and $R$-module $V$, perfection of $V$ over $Q/(f_c)$ is equivalent to perfection over $Q/(g_c)$.  We may therefore speak unambiguously of perfection of $V$ over $Q_c$.

\begin{definition}\label{def:hyp_supp}
For $V$ a finite $R$-module, we define the {\it hypersurface support} as
\[
\supp^{hyp}_{\mbb{P}}(V):=\{c\in \mbb{P}(m_Z/m_Z^2):V\text{ is \emph{not} perfect over }Q_c\}^{-},
\]
where the final bar denotes the closure in $\mbb{P}$, and $c$ runs over all {\it closed} points in projective space.
\end{definition}

\subsection{Equating cohomological and hypersurface supports}

We maintain the assumptions of the previous subsection.  Closed points $c\in \mbb{P}(m_Z/m_Z^2)$ correspond to graded algebra surjections $\phi_c:A_Z\to k[t]$, where $t$ is given degree $2$, and two maps are taken to be equivalent if they differ by a graded automorphism of $k[t]$, i.e.\ a nonzero scaling of $t$.  Taking the fiber $k[t]\ot_{A_Z}-$ along any point $c$ is then identified with the reduction
\begin{equation}\label{eq:811}
k[t]\ot_{A_Z}-\cong k\ot_{A_c}-,
\end{equation}
where $A_c\subset A_Z$ is the subalgebra generated by $\ker(\phi_c|_{A^2_Z})$.  One sees directly that the generating subspace $\ker(\phi_c|_{A^2_Z})\subset (m_Z/m_Z^2)^\ast$ is the kernel of the associated map $c:(m_Z/m_Z^2)^\ast\to k$, or rather of the map associated to any choice of representative for $c\in \mbb{P}(m_Z/m_Z^2)$, so that $A_c$ is precisely the subalgebra of~\eqref{eq:Ac}.

\begin{lemma}\label{lem:whatevsbruh}
Consider a deformation $Q\to R$ as in Section \ref{sect:setup}.  Then for any closed point $c\in \mbb{P}(m_Z/m_Z^2)$, and any finite-dimensional $R$-modules $V$ and $W$, the following are equivalent:
\begin{enumerate}
\item[(a)] $\Ext^{\gg0}_{Q_c}(V,W)=0$.
\item[(b)] The base change $k[t]\ot_{A_Z}\Ext^\ast_R(V,W)$ along the map $\phi_c:A_Z\to k[t]$ is finite-dimensional.
\item[(c)] The base change $k[t,t^{-1}]\ot_{A_Z}\Ext^\ast_R(V,W)$ along the localized map $\phi_{c,loc}:A_Z\to k[t,t^{-1}]$ vanishes.
\end{enumerate}
\end{lemma}

\begin{proof}
Suppose (a) holds.  Then Theorem \ref{thm:fin2} implies that $\Ext^\ast_R(V,W)$ is a finitely generated $A_c$-module, and hence the fiber
\[
k[t]\ot_{A_Z}\Ext^\ast_R(V,W)\cong k\ot_{A_c}\Ext^\ast_R(V,W)
\]
is finite-dimensional.  So we see (a) implies (b).  Conversely, if the above fiber is finite-dimensional then $\Ext^\ast_R(V,W)$ is finitely generated over $A_c$, by Corollary \ref{cor:Z_Nak}.  Theorem \ref{thm:fin2} then tells us that $\Ext^{\gg 0}_{Q_c}(V,W)=0$, providing (a).
\par

For the equivalence between (b) and (c), we note that $\Ext^\ast_R(V,W)$ is finitely-generated over $A_Z$, by Theorem \ref{thm:fin2}, so that $k[t]\ot_{A_Z}\Ext^\ast_R(V,W)$ is a finitely-generated $k[t]$-module.  We recall that a finitely generated, graded, $k[t]$-module is $t$-torsion if and only if it is finite-dimensional.  So we see that the localization $k[t,t^{-1}]\ot_{A_Z}\Ext^\ast_R(V,W)$ vanishes if and only if $k[t]\ot_{A_Z}\Ext^\ast_R(V,W)$ is finite-dimensional.
\end{proof}

\begin{theorem}\label{thm:sigma_hyp}
Suppose $Q\to R$ is a deformation satisfying {\rm (A)} and {\rm (B)} from Section \ref{sect:nchyp_supp}.  Then for $V$ any finitely generated $R$-module we have an identification of supports
\[
\supp^{\sigma}_{\mbb{P}}(V)=\supp^{hyp}_{\mbb{P}}(V).
\]
\end{theorem}

\begin{proof}
It suffices to show that the closed points of $\supp_{\mbb{P}}^\sigma(V)$ are equal to the collection of closed points $c\in \mbb{P}$ at which $V$ is not perfect over $Q_c$, in which case taking the closure of this final set recovers $\supp_{\mbb{P}}(V)$.  Equivalently, it suffices to show that the complements of these collections of closed points agree.  Recalling our assumption (B), Theorem \ref{thm:fin2} applied to the deformation sequence $Z/(f_c) \to Q/(f_c) \to R$ implies
\[
c\notin\supp^{hyp}_\mbb{P}(V)\ \Leftrightarrow\ \Ext^{\gg0}_{Q_c}(\sigma V,k)=0.
\]
By Lemma \ref{lem:whatevsbruh} we therefore find that $c\notin\supp^{hyp}_\mbb{P}(V)$ if and only if the base change of $\Ext^\ast_R(\sigma V,k)$ along the localized map $\phi_c:A_Z\to k[t,t^{-1}]$ vanishes,
\[
k[t,t^{-1}]\ot_{A_Z}\Ext^\ast_{R}(\sigma V,k)=0.
\]
This gives $\mbb{P}-\supp^{hyp}_\mbb{P}(V)$ as the complement of the support of the coherent $\O_\mbb{P}$-module associated to $\Ext_R^\ast(\sigma V,k)$, and thus $\supp^{hyp}_{\mbb{P}}(V)=\supp^{\sigma}_{\mbb{P}}(V)$.
\end{proof}

As one sees from the proof, the closed points of $\supp^{\sigma}_\mbb{P}(V)$ agree \emph{precisely} with those (closed) points $c$ of $\mbb{P}$ at which $V$ is non-perfect over the hypersurface $Q_c$.  So we have as a corollary to the proof

\begin{corollary}
The closed points of $\supp^{\sigma}_\mbb{P}(V)$ in $\mbb{P}$ are exactly the set $\{c\in \mbb{P}:V\text{\rm is \emph{not} perfect over }Q_c\}$.  In particular, this set is already closed (in the topological subspace of $k$-points) in $\mbb{P}$.
\end{corollary}

The same arguments apply verbatim to the support $\supp_\mbb{P}(V)=\Supp_{\mbb{P}}\Ext_R(V,\Lambda)^\sim$ in $\mbb{P}$, so that we may identify this cohomological and hypersurface support.

\begin{theorem}[{cf.~\cite[Theorem 2.5]{avramovbuchweitz00}}]\label{thm:cohom_hyp}
For $V$ any finite $R$-module, there is an identification of supports $\supp_{\mbb{P}}(V)=\supp^{hyp}_{\mbb{P}}(V)$.
\end{theorem}

\begin{remark}
The cohomological support $\supp_\mbb{P}$ does not play a significant role in our study.  However, it is the appropriate object to consider if one is interested in Snashall and Solberg's support theory for associative algebras via Hochschild cohomology~\cite{snashallsolberg04}.  All un-superscripted supports $\supp_{?}$ which appear in the complement of this section are Hopfy cohomological supports, as defined in Section~\ref{sect:hopf_hypersurface} below.
\end{remark}

\section{Hopfy support and hypersurfaces}
\label{sect:hopf_supp}

We give a precise relation between hypersurface support, defined as in Definition \ref{def:hyp_supp}, and the usual cohomological support for integrable Hopf algebras.

\subsection{Choose a side}\label{sect:side}

In deciding, in~\ref{sect:setup}, that the parametrizing algebra $Z$ for an integration $\msf{U}\to \msf{u}$ is a right coideal subalgebra in $\msf{U}$, we have chosen that $\rep(\msf{u})$ acts on the left of the hypersurface categories $\msf{U}/(f)\text{-mod}_{fg}$ (see Section~\ref{sect:module_cats}).  So the endomorphism $\sigma:\msf{U}/(f)\text{-mod}_{fg}\to \msf{U}/(f)\text{-mod}_{fg}$ considered in the Hopf situation of Example~\ref{ex:comod_def} must be tensoring with the simples $\Lambda$ on the {\it left}, $\sigma=\Lambda\ot-$.
\par

To produce consistency with these choices, we consider Hopfy cohomological support according to the algebra maps
\begin{equation}\label{eq:881}
V\ot-:\Ext^\ast_\msf{u}(k,k)\to \Ext^\ast_\msf{u}(V,V)
\end{equation}
provided by the left action of $\rep(\msf{u})$ on itself.
\par

These particular right/left choices do not really matter.  What matters, however, is that one does not switch from a ``right handed support" to a ``left handed support" in a willy-nilly manner.  So, we stick to the handedness proposed in Section~\ref{sect:setup}, unless explicitly stated otherwise.

\subsection{Hypersurface support and Hopf algebras}
\label{sect:hopf_hypersurface}

We consider again $\msf{u}$ a finite-dimensional, integrable, Hopf algebra with chosen integration $\msf{U}\to \msf{u}$.  We observe the algebra map $A_Z\to\Ext^\ast_\msf{u}(k,k)$ of equation~\eqref{eq:627}.  Fix $\msf{Y}:=\Proj\left(\Ext^\ast_\msf{u}(k,k)\right)_{\rm red}$ and
\begin{equation}\label{eq:kappa}
\kappa:\msf{Y}\to \mbb{P}(m_Z/m_Z^2)
\end{equation}
the map of varieties dual to~\eqref{eq:627}.  When $Z\subset \msf{U}$ is specifically \emph{not} a Hopf subalgebra, we may consider the alternate parametrizing subalgebra $Z'=S(Z)$ and corresponding map $\kappa':\msf{Y}\to \mbb{P}(m_Z/m_Z^2)$.
\par

Recall the standard Hopfy cohomological support
\[
\supp_{\msf{Y}}(V):=\Supp_{\msf{Y}}\Ext_\msf{u}^\ast(V,V)^\sim,
\]
where we calculate the sheaf $\Ext_\msf{u}^\ast(V,V)^\sim$ according to the (graded) action of $\Ext_\msf{u}^\ast(k,k)$ on $\Ext_\msf{u}^\ast(V,V)$ provided by the algebra map \eqref{eq:881}.

\begin{theorem}\label{thm:hopf_hyp}
Consider $\msf{u}$ a Hopf algebra which admits an integration $\msf{U}\to \msf{u}$.  Then the Hopfy cohomological support and hypersurface support satisfy
\[
\kappa\big(\supp_{\msf{Y}}(V)\big)=\supp^{hyp}_\mbb{P}(V),
\]
where $\kappa:\msf{Y}\to \mbb{P}$ is as above.
\end{theorem}

\begin{proof}
For $\sigma=\Lambda\ot -$ we have $\kappa(\supp_{\msf{Y}}(V))=\supp^\sigma_\mbb{P}(V)$.  Specifically, the modules $\Ext^\ast_\msf{u}(V,V)$ and $\Ext^\ast_\msf{u}(\Lambda\ot V,k)$ have the same support over $\Ext^\ast_\msf{u}(k,k)$ \cite[\S 5.7]{benson91}.  By commutative algebra, when we restrict along the (finite) algebra map $A_Z\to \Ext^\ast_\msf{u}(k,k)$, the support of the given module over $A_Z$ is the image of the support in $\Spec(\Ext^\ast_\msf{u}(k,k))_{\rm red}$ along the map $i_\msf{U}^\ast$ to $\Spec(A_Z)=\mbb{A}(m_Z/m_Z^2)$.  These supports are conical subvarieties in the given $\mbb{G}_m$-equivariant varieties, and we projectivize to obtain the claimed identification of $\kappa(\supp_{\msf{Y}}(V))$ with $\supp^\sigma_\mbb{P}(V)$.  So the result follows by Theorems~\ref{thm:sigma_hyp} and~\ref{thm:cohom_hyp}.
\end{proof}

If we suppose, furthermore, that $\kappa:\msf{Y}\to \mbb{P}$ is a closed embedding then the Hopfy support can be seen as a support theory valued in closed subvarieties in $\mbb{P}$.

\begin{corollary}\label{cor:hopf_hyp}
If the map $\kappa:\msf{Y}\to \mbb{P}$ of~\eqref{eq:kappa} is a closed embedding, then
\[
\supp_{\msf{Y}}(V)=\supp^{hyp}_\mbb{P}(V).
\]
In particular, the hypersurface support vanishes on the open complement $\mbb{P}-\msf{Y}$.
\end{corollary}

\subsection{A weak tensor product property}
\label{sect:weak_tpp}

We consider again an integrable Hopf algebra $\msf{u}$ with chosen integration $\msf{U}\to \msf{u}$.  We consider the following additional conditions:
\begin{enumerate}
\item[(S1)] $Z$ is a Hopf subalgebra in $\msf{U}$.
\item[(S2)] $\msf{u}$ is local.
\item[(S3)] $\rep(\msf{u})$ is braided.
\end{enumerate}
Recall that a braiding on $\rep(\msf{u})$ is a choice of natural swap operation $c_{V,W}:V\ot W\overset{\sim}\to W\ot V$ on products of $\msf{u}$-representations.  The $c_{V,W}$ here are assumed to be natural in both $V$ and $W$, so that in total we have a natural isomorphism $c_{-,-}$ between the tensor product on $\rep(\msf{u})$ and its opposite, and are also assumed to satisfy the braid relations \cite[\S 1]{joyalstreet86}.

\begin{proposition}\label{prop:weak_tpp}
Suppose that the integration $\msf{U}\to \msf{u}$ satisfies any one of the conditions {\rm (S1)--(S3)} above.  Then for arbitrary $V$ and $W$ in $\rep(\msf{u})$, hypersurface support satisfies the following:
\begin{enumerate}
\item $\supp^{hyp}_\mbb{P}(V\ot W)\subset\big(\supp^{hyp}_\mbb{P}(V)\cap\supp^{hyp}_\mbb{P}(W)\big)$.
\item $\supp^{hyp}_\mbb{P}(V)=\supp^{hyp}_\mbb{P}(V^\ast)$.
\end{enumerate}
\end{proposition}

In the proof we employ the notion of a thick subcategory in a triangulated category.  The definition of such a subcategory is recalled in Section \ref{sect:nc_Hop} below.

\begin{proof}
In the local case (S2) we have that $V\ot W$ is in the thick subcategory generated by $V$ and also the thick subcategory generated by $W$, where here we consider specifically the thick subcategories in $D_{coh}(\msf{U}/(f))$ for varying hypersurface.  This implies the inclusion (1).
\par

Consider now the braided case (S3).  We note that at any function $f\in m_Z$ the category $\msf{U}/(f)\text{-mod}_{fg}$ is a left $\rep(\msf{u})$-module category, in the sense that for any $V$ in $\rep(\msf{u})$, and $\msf{U}/(f)$-module $M$, the action of $\msf{U}$ on $V\ot M$ descends to an action of $\msf{U}/(f)$.  The duality
\[
\Hom_{\msf{U}/(f)}(V\ot M,-)\cong \Hom_{\msf{U}/(f)}(M,{^\ast V}\ot-)
\]
implies that the endomorphism $V\ot-:\msf{U}/(f)\text{-mod}_{fg}\to \msf{U}/(f)\text{-mod}_{fg}$ associated to any object in $\rep(\msf{u})$ preserves perfection.  Here ${^\ast V}$ is the linear dual of $V$ with $\msf{u}$ acting via the inverse antipode $S^{-1}$ (see \cite[\S 2.10]{egno15}).  So we see that
\[
\supp_\mbb{P}^{hyp}(V\ot W)\subset \supp^{hyp}_\mbb{P}(W).
\]
Since $V\ot W\cong W\ot V$ in this case, we also have an inclusion into $\supp^{hyp}_\mbb{P}(V)$, establishing (1).
\par

Suppose now that $Z$ is a Hopf subalgebra in $\msf{U}$, as in (S1).  In this case $\rep(\msf{u})$ acts on both the left \emph{and the right} of the category $\msf{U}/(f)\text{-mod}_{fg}$.  So one again employs dualities, as in the braided case, to obtain (1).
\par

Point (1) now implies, for $\msf{u}$ satisfying any of the (S\#), that any object $L$ which is a summand of a product $W\ot V\ot W'$, there is an inclusion $\supp_\mbb{P}^{hyp}(L)\subset\supp_\mbb{P}^{hyp}(V)$.  In particular, $\supp_\mbb{P}^{hyp}(V^\ast)\subset \supp_\mbb{P}^{hyp}(V)$, as $V^\ast$ is a summand of the product $V\ot V^\ast\ot V$, by definition of the dual ${^\ast (V^\ast)}\cong V$~\cite[Definition 2.10.1]{egno15}.  One obtains the opposite inclusion similarly.  So we have (2). 
\end{proof}

The analogous equations (1) and (2), for cohomological support $\supp_\msf{Y}$, are known to hold in both the local and braided cases (S2) \& (S3).  However, in the case (S1) we obtain something new.  We apply Corollary \ref{cor:hopf_hyp} to obtain the following from Proposition \ref{prop:weak_tpp}.

\begin{corollary}\label{cor:weak_tpp2}
Suppose that $\msf{u}$ admits an integration $\msf{U}\to \msf{u}$ for which $Z$ is a Hopf subalgebra in $\msf{U}$ and that the map $\kappa$ of \eqref{eq:kappa} is a closed embedding.  Then cohomological support for $\rep(\msf{u})$ satisfies the following:
\begin{enumerate}
\item $\supp_\msf{Y}(V\ot W)\subset\big(\supp_\msf{Y}(V)\cap\supp_\msf{Y}(W)\big)$.
\item $\supp_\msf{Y}(V)=\supp_\msf{Y}(V^\ast)$.
\end{enumerate}
\end{corollary}

We note that the relations of Corollary~\ref{cor:weak_tpp2} are not obvious when $\rep(\msf{u})$ is not braided, and not even true in general.  One can see Example~\ref{ex:no_tpp} below, or~\cite{plavnikwitherspoon18} for a more extensive exposition.

\begin{lemma}\label{lem:surj0}
If $\msf{u}$ is such that the reduced spectrum $\Spec(\Ext^\ast_\msf{u}(k,k))_{\rm red}$ is isomorphic to an affine space $\mbb{A}^n_k$, then the map $\kappa:\msf{Y}\to \mbb{P}(m_Z/m_Z^2)$ of~\eqref{eq:kappa} is a closed embedding.
\end{lemma}

\begin{proof}
Take $E=\Ext^\ast_{\msf{u}}(k,k)_{\rm red}$ and take $E'$ the subalgebra generated by $E^2$.  By assumption, $E$ is isomorphic to a polynomial ring.  Since $E$ is graded, we may take homogenous generators $\xi_i$, $\deg(\xi_i)=d_i$, and have $E=k[\xi_1,\dots, \xi_d]$.  We know that the deformation map
\begin{equation}\label{eq:1075}
\Sym(\Sigma^{-2}(m_Z/m_Z^2)^\ast)=A_Z\to E
\end{equation}
is finite, by Theorem \ref{thm:fg}.  We note that $A_Z$ has image in the subalgebra $E'$ generated in degree $2$, so that $E$ must be finite over $E'$, and $E'$ must be finite over the image of $A_Z$.  We have $k\ot_{A_Z}E'=\Sym(E^2/{\rm im}(m_Z/m_Z^2)^\ast)$ and $k\ot_{E'}E=k[\xi_j:\deg(\xi_j)>2]$.  Finiteness of these fibers forces $E^2={\rm im}(m_Z/m_Z^2)^\ast$, and $E'=E$.  Hence we have surjectivity of~\eqref{eq:1075}.
\end{proof}

One applies Lemma~\ref{lem:surj0} to see that the relations of Corollary~\ref{cor:weak_tpp2} hold in a number of ``unipotent" and ``solvable" settings.  We just remark on one.

\begin{corollary}\label{cor:weak_uqb}
For $u_q(\mfk{b})$ the small quantum Borel, in arbitrary Dynkin type, at $q$ of odd order greater than the associated Coxeter number $h$, cohomological support satisfies
\[
\supp_\msf{Y}(V\ot W)\subset \big(\supp_\msf{Y}(V)\cap\supp_\msf{Y}(W)\big)
\]
and also $\supp_\msf{Y}(V)=\supp_\msf{Y}(V^\ast)$.
\end{corollary}

\begin{proof}
The distinguished subalgebra $Z_0$ in $U^{DK}_q(\mfk{b})$ is a Hopf subalgebra and, by~\cite{ginzburgkumar93}, the algebra of extensions $\Ext^\ast_{u_q(\mfk{b})}(\mbb{C},\mbb{C})$ is a polynomial ring in this case.  So the result follows by Lemma \ref{lem:surj0} and Corollary \ref{cor:weak_tpp2}.
\end{proof}

\begin{remark}
For $\msf{u}$ the restricted enveloping algebra $u^{\rm res}(L)$ of reductive $L$ in large characteristic, we expect the map $\kappa$ is simply the embedding of the (twisted) nilpotent cone $\mcl{N}^{(1)}$ into $L^{(1)}$, or rather the projectivization of this map.  Similarly, for the quantum group $u_q(\mfk{g})$ at large order $q$, $\kappa$ should also be the projectivized embedding of the nilpotent cone into $\mfk{g}$.  So, in particular, the hypersurface and cohomological supports should agree here.  In the finite characteristic setting, one could explicitly prove this result by identifying the Hochschild map employed in \cite{friedlanderparshall86II}, \cite[Proposition 5.2]{friedlanderparshall83}, with our deformation map of \eqref{eq:627}.
\end{remark}

\begin{remark}
In Section \ref{sect:last} we will be rather precise in our definition of the quantum Borel.  In particular, we acknowledge that there are varying choices of Borels with varying groups of grouplikes.  In Corollary~\ref{cor:weak_uqb} the particular choice of grouplikes for $u_q(\mfk{b})$ does not matter.
\end{remark}

\section{Background: a preamble to examples}
\label{sect:bg_examples}

We now move on to the second phase in this work.  In the following sections we use hypersurface support to prove the tensor product property for cohomological support in a number of ``solvable" examples.  We explain what precisely counts as a ``solvable" Hopf algebra in Section~\ref{sect:g_chev} below.  The proofs of the tensor product property in these varied contexts rely on certain recurring arguments which are both homological and (mildly) tensor categorical in nature.  We provide here the backgrounds needed to understand the examples considered in the latter sections of the text.
\par

We recall that a tensor category is a $k$-linear, abelian, rigid monoidal~\cite[\S 2.10]{egno15} category $\msc{C}$ in which all objects are of finite length and $\Hom$ sets are finite-dimensional.  We also require that the unit $\1\in \msc{C}$ is simple.  A tensor category $\msc{C}$ is \emph{finite} if it has finitely many simples and enough projectives, and $\msc{C}$ is \emph{fusion} if it is finite and semisimple.  The only tensor categories we are interested are those of the form $\rep(\msf{u})$ for a finite-dimensional Hopf algebra $\msf{u}$ (in which case $\rep(\msf{u})$ is finite), so one needn't concern themselves with the intricacies of the theory of tensor categories here.  We would claim, however, that the language and philosophies from the subject are useful.
\par

We are also interested in \emph{braided} tensor categories, which, just as in the case of $\rep(\msf{u})$ discussed in Section \ref{sect:weak_tpp}, are tensor categories $\msc{C}$ equipped with a natural swap operation $c_{V,W}:V\ot W\overset{\sim}\to W\ot V$ on objects \cite[\S 1]{joyalstreet86}.

\subsection{Module categories}
\label{sect:module_cats}

We refer the reader to~\cite[Definition 6.2]{ostrik03} for a precise definition of a module category.  Basically, a (left) module category over a tensor category $\msc{C}$ is a $k$-linear abelian category $\msc{M}$ with a biexact action functor $\ot:\msc{C}\times \msc{M}\to \msc{M}$ which is compatible with the tensor structure on $\msc{C}$, in the obvious ways.  One has the analogous notion of a bimodule category.
\par

We have already encountered our primary source of examples: Given $\rep(\msf{u})$ for a Hopf algebra $\msf{u}$, and $B$ a (left) $\msf{u}$-comodule algebra, we obtain an action of $\rep(\msf{u})$ on $B\text{-mod}$.  In particular, for arbitrary $V$ in $\rep(\msf{u})$ and $M$ in $B$-mod, we employing the coaction $B\to \msf{u}\ot B$ to provide a natural $B$-action on the tensor products $V\ot B$.  In this way, the linear tensor product $\ot=\ot_k$ provides our exact action functor $\ot:\rep(\msf{u})\times B\text{-mod}\to B\text{-mod}$, endowing $B\text{-mod}$ with a $\rep(\msf{u})$-module category structure.  When $B$ is a $\msf{u}$-bicomodule, we similarly find that $B\text{-mod}$ is a bimodule category over $\rep(\msf{u})$.
\par

By considering the adjunction $\Hom_B(V\ot M,-)=\Hom_B(M,{^\ast V}\ot-)$ we see that the action of $\rep(\msf{u})$ on such $B$-mod preserves the subcategory of compact objects, i.e.\ finitely presented modules.  So, when $B$ is Noetherian, we have that $B\text{-mod}_{fg}$ is a $\rep(\msf{u})$-module subcategory in $B\text{-mod}$.  The analogous claims hold in the bicomodule/$\rep(\msf{u})$-bimodule setting.  Such duality arguments also show that the action(s) of $\rep(\msf{u})$ on $B$-mod preserves perfect objects.

\subsection{Drinfeld centeralizers}\label{sect:ZC}

\begin{definition}[\cite{muger03}]\label{def:ZC}
Given a tensor subcategory $\msc{D}$ in a tensor category $\msc{C}$, we define the Drinfeld centralizer $Z^{\msc{D}}(\msc{C})$ to be the category of pairs $(V,\gamma_V)$ where $V$ is an object in $\msc{C}$ and $\gamma_V$ is a half braiding against $\msc{D}$.  Specifically, $\gamma_V:V\ot-\to -\ot V$ is a natural isomorphism between the functors $-\ot V,\ V\ot -:\msc{D}\to \msc{C}$ which satisfies the appropriate braid relations.
\par

The Drinfeld center $Z(\msc{C})$ of $\msc{C}$ is the centralizer of $\msc{C}$ against itself, i.e.\ the category of pairs of an object $V$ with a global half braiding $\gamma_V:V\ot -\to -\ot V$.
\end{definition}

Both $Z^\msc{D}(\msc{C})$ and $Z(\msc{C})$ inherit obvious tensor structures $(V,\gamma_V)\ot (W,\gamma_W)=(V\ot W,(\gamma_V\ot 1)(1\ot\gamma_W))$ from $\msc{C}$, and we have a surjective tensor functor
\begin{equation}\label{eq:forget}
\operatorname{Forget}:Z^\msc{D}(\msc{C})\to\msc{C}
\end{equation}
given by forgetting the half braiding~\cite[Proposition 3.39]{etingofostrik04}.  By surjective we mean that all objects of $\msc{C}$ are obtained by taking subquotients of objects from $Z^\msc{D}(\msc{C})$.
\par

In the Hopf setting, $Z(\rep(\msf{u}))$ is equivalent to the category of modules over the Drinfeld double $D(\msf{u})=\msf{u}\!\bowtie\! \msf{u}^\ast$ (see e.g.~\cite{montgomery93,kassel12}).  For $\rep(\Lambda)\to \rep(\msf{u})$ the tensor inclusion corresponding to a Hopf quotient $\msf{u}\to \Lambda$, the Drinfeld centralizer is identified with representations over the Hopf subalgebra $\msf{u}\!\bowtie\! \Lambda^\ast$ in $D(\msf{u})$.
\par

We prefer the categorical expression $Z^\msc{D}(\msc{C})$ to the algebraic one as we mean to emphasize a single point here: It is natural to consider, in some specific contexts, objects $V$ in $\rep(\msf{u})$ with a uniform means of moving $V$ past a chosen class of other objects in $\rep(\msf{u})$, under the tensor action.

\subsection{Bosonization and half-braidings}
\label{sect:bhb}

Consider $\msf{u}^+$ a finite-dimensional Hopf algebra in a braided tensor category $\rep(\Lambda)$ which, as the notation suggests, is the representation category of a quasitriangular Hopf algebra $\Lambda$.  Then the category $\rep_{\Lambda}(\msf{u}^+)$ of finite-dimensional left $\msf{u}^+$-modules in $\rep(\Lambda)$ forms a tensor category for which the composition
\[
\rep_{\Lambda}(\msf{u}^+)\overset{forget}\to \rep(\Lambda)\to Vect
\]
provides a canonical fiber functor.  Such modules are simply $\msf{u}^+$-modules with a compatible action of $\Lambda$, i.e.\ modules over the smash product $\msf{u}=\msf{u}^+\rtimes \Lambda$, and the tensor structure on $\rep_{\Lambda}(\msf{u}^+)$ induces a Hopf structure on $\msf{u}$ under which the subalgebra $\msf{u}^+$ is a left $\msf{u}$-comodule algebra.  The algebra $\msf{u}$, with its given Hopf structure, is the \emph{bosonization} of $\msf{u}^+$.  We consider the explicit examples of bosonized quantum complete intersections in Section~\ref{sect:fun-q} below.
\par

One can endow all objects in $\rep(\Lambda)$ with the trivial $\msf{u}^+$-action, via the counit $\msf{u}^+\to \1$, to obtain a tensor embedding
\[
\rep(\Lambda)\to \rep_{\Lambda}(\msf{u}^+)=\rep(\msf{u}).
\]
When $\msf{u}^+$ is local, and $\Lambda$ is semisimple, this map is an equivalence onto the fusion subcategory of semisimple objects in $\rep(\msf{u})$.  In the statement of the following lemma $c$ denotes the braiding on $\rep(\Lambda)$.

\begin{lemma}\label{lem:bos}
For $\msf{u}$ a bosonized Hopf algebra $\msf{u}=\msf{u}^+\rtimes\Lambda$ as above, there is a canonical tensor functor $\rep(\msf{u})\to Z^{\rep(\Lambda)}(\rep(\msf{u}))$ which is a section of the forgetful functor.  Specifically, for $V$ in $\rep(\msf{u})$ and $L$ in $\rep(\Lambda)$, the isomorphisms
\[
\gamma_{V,L}:V\ot L\to L\ot V
\]
given by the braiding $\gamma_{V,L}:=c_{V,L}$ on $\rep(\Lambda)$ provide a $\rep(\Lambda)$-centralizing structure on all objects in $\rep(\msf{u})$, and the operations $\gamma_{V,L}$ are natural in $V$ and $L$.
\end{lemma}

\begin{proof}
Both $V\ot L$ and $L\ot V$ have well-defined $\msf{u}$-actions, and the braiding $\gamma=c$ is a map of $\Lambda$-modules by construction.  So we need only check $\msf{u}^+$-linearity.  Since $\msf{u}^+$ acts trivially on $L$, the action of $\msf{u}^+$ on $V\ot L$ is just induced by the action on $V$, and the action on $L\ot V$ employs the braiding
\[
\operatorname{act}_{L\ot V}=(1\ot \operatorname{act}_V)(c_{\msf{u}^+,L}\ot 1):\msf{u}^+\ot L\ot V\to L\ot V.
\]
Naturality of the braiding, and the braid relation, then imply commutativity of the diagram
\[
\xymatrix{
\msf{u}^+\ot V\ot L\ar[r]^{1\ot c}\ar[d]_{\operatorname{act}_{V\ot L}} & \msf{u}^+\ot L\ot V\ar[d]^{\operatorname{act}_{L\ot V}}\\
V\ot L\ar[r]^c & L\ot V.
}
\]
The above diagram implies that $\gamma_{V,L}=c_{V, L}$ is a map of $\msf{u}^+$, and hence $\msf{u}$, representations.
\end{proof}

\subsection{Geometrically Chevalley algebras}
\label{sect:g_chev}

In the following definition we let $\Rep(\Lambda)$ denote the monoidal category of infinite-dimensional representations for a finite-dimensional Hopf algebra $\Lambda$.

\begin{definition}
Call an integrable Hopf algebra $\msf{u}$ geometrically Chevalley if
\begin{enumerate}
\item[(a)] $\msf{u}$ is the bosonization $\msf{u}=\msf{u}^+\rtimes \Lambda$ of a local Hopf algebra $\msf{u}^+$ in a (semisimple!) braided fusion category $\rep(\Lambda)$.
\item[(b)] $\msf{u}^+$ admits a deformation sequence $Z\to \msf{U}^+\to \msf{u}^+$ via algebras in $\Rep(\Lambda)$ such that 
\begin{enumerate}
\item[(b1)] $\msf{U}^+$ is a local Hopf algebra in $\Rep(\Lambda)$ which is of finite global dimension, as an associative algebra.
\item[(b2)] $Z$ is a central Hopf subalgebra in $\msf{U}^+$ which has trivial $\Lambda$-action and admits an algebra isomorphism $Z\cong k\b{y_1,\dots,y_n}$.
\end{enumerate}
\end{enumerate}
We call $\msf{U}=\msf{U}^+\rtimes \Lambda$ the corresponding Chevalley integration of $\msf{u}$, and have the corresponding integration $\msf{U}\to \msf{u}$ parametrized by the Hopf subalgebra $Z$ in $\msf{U}$.
\end{definition}

Examples of geometrically Chevalley algebras include bosonized quantum complete intersections, quantum Borels, and also the Drinfeld doubles $D(B_{(1)})$ for Borel subgroups $B\subset \mbb{G}$ in almost-simple algebraic group $\mbb{G}$ (in particular characteristics).  Restricted enveloping algebras $u^{\rm res}(\mfk{n})$ of nilpotent restricted Lie algebras are also geometrically Chevalley, with $\Lambda=k$.  In the case of $D(B_{(1)})$ we take $\Lambda=kT_1$, where $T\subset B$ is the torus, and the braiding on $\rep(\Lambda)=\rep(T_1)$ is the standard (trivial) symmetry.  These examples are all considered in detail in Sections~\ref{sect:fun-q} and \ref{sect:last} below.
\par

By Lemma~\ref{lem:bos}, all objects in $\rep(\msf{u})$ centralize the simples $\rep(\Lambda)$ in this case, and we have a canonical section $\rep(\msf{u})\to Z^{\rep(\Lambda)}(\rep(\msf{u}))$ for any geometrically Chevalley $\msf{u}$.

\section{Noncommutative Hopkins lemma for hypersurfaces}
\label{sect:nc_Hop}

Let $\mcl{T}$ be a triangulated category.  Recall that a thick subcategory in $\mcl{T}$ is a full triangulated subcategory which is closed under taking summands in $\mcl{T}$, and the thick subcategory $\langle X_i:i\in I\rangle$ \emph{generated} by a collection of objects $\{X_i\}_{i\in I}$ in $\mcl{T}$ is the smallest thick subcategory in $\mcl{T}$ which contains all the $X_i$.  The most basic example of a thick subcategory for us is that of perfect complexes $\operatorname{perf}(Q)=\langle Q\rangle$ in the derived category of modules over a given algebra $Q$.
\par

Here we suggest an analysis of thick subcategories for the hypersurface categories $D_{coh}(\msf{U}/(f))$--or rather, $D_{coh}(\msf{U}^+/(f))$ in the geometrically Chevalley instance--which provides a technical foundation for the work that follows.

\subsection{Hopkins Lemma and noncommutative hypersurfaces}
\label{sect:hop}

Consider $Q$ a \emph{commutative}, regular, local algebra and $Q/(f)$ a hypersurface algebra.  Then a version of Hopkins lemma~\cite[Proposition 5.8]{carlsoniyengar15} implies that, for $M$ non-perfect over $Q/(f)$, the trivial module $k$ is in the thick subcategory $\langle M\rangle$ generated by $M$ in the derived category of $Q/(f)$.  It seems clear that this precise result does \emph{not} hold for a noncommutative hypersurface algebra $Q/(f)$, as one needs to properly account for the noncommutativity in this case.
\par

In a geometrically Chevalley setting $\msf{U}^+\to \msf{u}^+$, we consider an alternative implication:
\begin{equation}\label{eq:1082}
M\text{ non-perfect in }\msf{U}^+/(f)\text{-mod}_{fg}\ \overset{\rm question}\Rightarrow\ k\in \langle \lambda\ot M:\lambda\in \operatorname{Irrep}(\Lambda)\rangle.
\end{equation}
We act on the left of $\msf{U}^+/(f)$-mod via the left $\msf{u}$, and hence $\Lambda$, comodule structure.  More generally, one may replace the $\lambda$ in $\operatorname{Irrep}(\Lambda)$ with objects in an arbitrary tensor subcategory $\msc{D}\subset \rep(\msf{u})$.  Our main approach to the tensor product property for the examples considered in Section~\ref{sect:O}--\ref{sect:last} below is to show that the implication~\eqref{eq:1082} \emph{does} hold in these cases, and use this implication to deduce the tensor product property for support.

\begin{lemma}\label{lem:gen_lem2}
Consider $\msf{u}$ geometrically Chevalley, with Chevalley integration $\msf{U}\to \msf{u}$.  Fix $\msc{D}\subset \rep(\msf{u})$ an arbitrary tensor subcategory, and consider functions $f\in m_Z$ with non-trivial reduction to $m_Z/m_Z^2$.
\par

Suppose that for any non-perfect, finitely generated, module $M$ over a hypersurface $\msf{U}^+/(f)$, $k$ is in the thick subcategory $\langle \lambda\ot M:\lambda\in \msc{D}\rangle$ generated by the $\msc{D}$-orbit of $M$ in $D_{coh}(\msf{U}^+/(f))$.  Then the hypersurface support $\supp^{hyp}_\mbb{P}$ on $D_{fin}(\msf{u})$ satisfies the tensor product property
\[
\supp^{hyp}_\mbb{P}(V\ot W)=\supp^{hyp}_\mbb{P}(V)\cap\supp^{hyp}_\mbb{P}(W).
\]
The analogous result holds when $\msf{u}$ is local and admits an integration $\msf{U}\to \msf{u}$ by a local Hopf algebra $(\msf{U}^+=)\msf{U}$.
\end{lemma}

\begin{proof}
In what follows, for local $\msf{u}$ we take $\Lambda=k$, $\rep(\Lambda)=Vect$.  In both the geometrically Chevalley and local case we have an inclusion of supports $\supp^{hyp}_\mbb{P}(V\ot W)\subset \supp^{hyp}_\mbb{P}(V)\cap\supp^{hyp}_\mbb{P}(W)$, by Proposition~\ref{prop:weak_tpp}.  So, we need only establish opposite inclusion.
\par

Since the positive subalgebra $\msf{u}^+$ is local, by hypothesis, the inclusion $\rep(\Lambda)\to \rep(\msf{u})$ is an isomorphism onto the fusion subcategory of semisimple modules.  Furthermore, this fusion subcategory generates $\rep(\msf{u})$ under extension.  Therefore, for any $\msc{D}$ as in the statement we have
\[
\langle \lambda\ot M:\lambda\in \msc{D}\rangle\subset \langle \lambda\ot M:\lambda\in \operatorname{Irrep}(\Lambda)\rangle.
\]
So it suffices to prove the result when $\msc{D}=\rep(\Lambda)$.
\par

We must prove the following: If $\msf{U}/(f)$ is a hypersurface along which $W$ is non-perfect, and the product $V\ot W$ is perfect, then $V$ must be perfect.  (Here $f$ is a function in the parametrizing algebra $Z$ with non-vanishing linear part.)  Now, imperfection of a module over $\msf{U}/(f)=(\msf{U}^+/(f))\rtimes\Lambda$ is equivalent to imperfection over the subalgebra $\msf{U}^+/(f)$.  So we are free to replace $\msf{U}/(f)$ with this local subalgebra $\msf{U}^+/(f)$, and attempt to establish the same claims regarding imperfection.
\par

By locality of $\msf{U}^+/(f)$, imperfection of a finitely generated module $M$ over $\msf{U}^+/(f)$ is equivalent to nonvanishing of $\Ext^\ast_{\msf{U}^+/(f)}(M,k)$ in arbitrarily high degrees.  We recall also that the $\msf{u}$-comodule structure on $\msf{U}^+/(f)$ implies a (left) action of $\rep(\msf{u})$ on $\msf{U}^+/(f)$-mod.

Consider such $f$, $V$, and $W$ as prescribed.  As $W$ is non-perfect along $\msf{U}^+/(f)$ we have
\[
k\in \langle \lambda\ot W:\lambda\in \operatorname{Irrep}(\Lambda)\rangle\ \subset D_{coh}(\msf{U}^+/(f)),
\]
by hypothesis.  Applying the endofunctor $V\ot-:D_{coh}(\msf{U}^+/(f))\to D_{coh}(\msf{U}^+/(f))$ and the centrality claim of Lemma~\ref{lem:bos} gives then
\[
V\in \langle V\ot\lambda\ot W:\lambda\in \operatorname{Irrep}(\Lambda)\rangle=\langle \lambda \ot V\ot W:\lambda\in \operatorname{Irrep}(\Lambda)\rangle.
\]
Since $V\ot W$ is perfect, this final thick subcategory is contained in $\operatorname{perf}(\msf{U}^+/(f))$, so that $V$ must be perfect.
\end{proof}

\begin{remark}
It suffices to work stably, so that one may replace the thick subcategory generated by the $\lambda\ot M$ with the thick subcategory generated by the $\lambda\ot M$ and $\msf{U}^+/(f)$ in equations~\eqref{eq:1082} and Lemma~\ref{lem:gen_lem2} above.
\end{remark}

\begin{remark}
It is possible that~\eqref{eq:1082} holds in general, although we would not claim such a result at this time.
\end{remark}

\section{The tensor product property for functions on finite group schemes}
\label{sect:O}

We consider representations of $\O(\mcl{G})$, for a finite group scheme $\mcl{G}$.  In this case we have $\rep(\O(\mcl{G}))=\Coh(\mcl{G})$, and we employ the latter notation for the sake of expediency.  We prove in Theorem~\ref{thm:Pio} below that (Hopfy) cohomological support for $\Coh(\mcl{G})$ has the tensor product property whenever $\mcl{G}$ is connected.  In the case in which $\mcl{G}$ is {\it not} connected, the tensor product property for $\Coh(\mcl{G})$ simply does not hold (see Example~\ref{ex:no_tpp}).  However, we prove in Corollary~\ref{cor:Pi} a variant of the tensor product property for non-connected $\mcl{G}$ which is, in a global sense, optimal.

Arguably, the case of $\Coh(\mcl{G})$ is not the most interesting application of the methods developed in this paper.  Indeed, many of the issues we've addressed in the present document are well-established in the commutative setting (see Remark~\ref{rem:CI}).  However, this very natural class of examples provides an excellent case study from which one can observe a number of important phenomenon.  We consider more novel situations in Sections~\ref{sect:fun-q} and~\ref{sect:last} below.

\subsection{Reducing to the connected component I}

Consider $\mcl{G}$ a finite group scheme, with identity component $\mcl{G}_o$.  We note that in this case the cohomology rings $\Ext^\ast_{\O(\mcl{G})}(k,k)$ and $\Ext^\ast_{\O(\mcl{G}_o)}(k,k)$ for the algebras of functions agree.  So we write unambiguously $\msf{Y}$ for the reduced projective spectrum of either of these rings.  The following lemma can be deduced from work of Plavnik and Witherspoon~\cite{plavnikwitherspoon18}, and we omit the proof.

\begin{lemma}[{\cite[Theorem 3.2]{plavnikwitherspoon18}}]\label{lem:coh1}
Let $\mcl{G}$ be a finite group scheme, with identity component $\mcl{G}_o\subset \mcl{G}$ and subgroup of closed points $\pi=\pi_0(\mcl{G})$.  The following are equivalent:
\begin{enumerate}
\item[(a)] Cohomological support for $\Coh(\mcl{G})$ has the tensor product property.
\item[(b)] Cohomological support for the identity components $\Coh(\mcl{G}_o)$ has the tensor product property and the group of closed points $\pi\subset \mcl{G}$ acts trivially on $\msf{Y}$, under the adjoint action.
\end{enumerate}
\end{lemma}

It is easy to construct examples of $\mcl{G}=\mcl{G}_o\rtimes \pi$ for which $\pi$ acts non-trivially on cohomology for $\mcl{G}_o$.  Indeed, one deduces from such pairs $(\mcl{G}_o,\pi)$ examples of tensor categories for which one has no such inclusion
\[
\supp_\msf{Y}(V\ot W)\nsubseteq \big(\supp_\msf{Y}(V)\cap\supp_\msf{Y}(W)\big).
\]
One can compare the following example with Corollary~\ref{cor:weak_tpp2}.

\begin{example}[{cf.\ \cite{bensonwitherspoon14}}]\label{ex:no_tpp}
Take $k=\overline{\mbb{F}}_p$.  Consider $\mbb{Z}_2$ acting on $\mbb{G}_{a(1)}\times \mbb{G}_{a(1)}$ by swapping the factors.  Take $\mcl{G}_o=\mbb{G}_{a(1)}\times\mbb{G}_{a(1)}$ and $\mcl{G}=\mcl{G}_o\rtimes \mbb{Z}_2$.  Let $\supp^{\mcl{G}_o}$ and $\supp^{\mcl{G}}$ denote the corresponding supports.  Since $\O(\mcl{G}_o)$ is cocommutative, its support $\supp^{\mcl{G}_o}$ is known to have the tensor product property~\cite{friedlandersuslin97}.
\par

Consider $V$ and $V'$ the $\O(\mcl{G}_o)$ representations given by projecting onto the second and first factor of $\O(\mbb{G}_{a(1)})$ respectively.  Then one has
\[
\Ext^\ast_{\O(\mcl{G}_o)}(k,k)_{\rm red}=k[x,y],\ \ \Ext^\ast_{\O(\mcl{G}_o)}(k,V)=k[x],\ \ \Ext^\ast_{\O(\mcl{G}_o)}(k,V')=k[y].
\]
We also have $V'=\rm{Ad}_{\sigma}(V)$ in $\Coh(\mcl{G})$, where $\mbb{Z}/2\mbb{Z}=\{1,\sigma\}$.  Thus
\[
\supp^\mcl{G}(V)=pt_x,\ \supp^\mcl{G}(\sigma\oplus k)=\mbb{P}^1,
\]
and
\[
\supp^\mcl{G}(V\ot(\sigma\oplus k))=\supp^\mcl{G}(V'\oplus V)=pt_x\cup pt_y,
\]
where $\mbb{P}^1$ is identified with the projective spectrum of $k[x,y]=\Ext^\ast(k,k)_{\rm red}$ and $pt_x$ and $pt_y$ are the images of the coordinate axes for $x$ and $y$.  In particular, if we take $\supp_\msf{Y}$ the cohomological support for $\O(\mcl{G})$, $V$ as above and $W=\sigma\oplus k$, then
\[
pt_x\cup pt_y=\supp_\msf{Y}(V\ot W)\nsubseteq\big(\supp_\msf{Y}(V)\cap\supp_\msf{Y}(W)\big)=pt_x.
\]
We note finally that the Hopfy support and hypersurface support agree for this example, by Lemma~\ref{lem:surjO}, so that the tensor product property for hypersurface support also fails to hold in this case.
\end{example}

One can similarly consider an $n$-fold product of $\mbb{G}_{a(1)}$'s equipped with the permutation action of $S_n$, and observe an explicit obstruction to the tensor product property in these cases as well.  A generalization of the above example, where $\O(\mcl{G})$ is replaced with an arbitrary Hopf algebra, can be found in~\cite[Proof of Corollary 3.4]{plavnikwitherspoon18}.

\subsection{Reducing to the connected component II}

We consider a finite group scheme $\mcl{G}$ with connected component $\mcl{G}_o$.  Suppose that support for $\Coh(\mcl{G}_o)$ satisfies the tensor product property, which is in fact the case by Theorem~\ref{thm:Pio} below.
\par

We have the action of the finite group $\pi=\pi_0(\mcl{G})$ of closed points, i.e.\ simple $\O(\mcl{G})$-representations, on $\Coh(\mcl{G})$ via conjugation $\rm{Ad}_\lambda(V)=\lambda\ot V\ot\lambda^{-1}$.  This action corresponds to the adjoint action of $\pi$ on the scheme $\mcl{G}$, in the sense that the tensor adjoint action permutes sheaves on $\mcl{G}$ via pushforward along the adjoint action automorphisms $\rm{Ad}_\lambda:\mcl{G}\to \mcl{G}$.
\par

To say that an object $V$ in $\Coh(\mcl{G})$ is equivariant with respect to this conjugation action of $\pi$ is to say that there is a compatible collection of natural isomorphisms $V\ot\lambda\cong \lambda\ot V$ for all $\lambda\in \pi$.  So we observe that a $\pi$-equivariant structure on $V$ is equivalent to a $\rep(\pi)$-centralizing structure, i.e.\ half-braiding, on $V$.
\par

In the following theorem we consider non-connected $\mcl{G}$, and let $\supp_\msf{Y}$ denote the cohomological support for $\Coh(\mcl{G})$.  We take $F:Z^{\Coh(\pi)}(\Coh(\mcl{G}))\to \Coh(\mcl{G})$ the forgetful functor.

\begin{lemma}[{cf.~\cite{plavnikwitherspoon18}}]\label{lem:1026}
Let $\mcl{G}$ be a finite group scheme and suppose that support for the identity component $\Coh(\mcl{G}_o)$ satisfies the tensor product property.  Take $\pi=\pi_0(\mcl{G})$ the group of closed points.
\par

For arbitrary $W$ in $\Coh(\mcl{G})$, and $V$ in the Drinfeld centralizer of $\Coh(\pi)$ in $\Coh(\mcl{G})$, there is an identification of supports
\[
\supp_\msf{Y}(F(V)\ot W)=\supp_\msf{Y}(F(V))\cap\supp_\msf{Y}(W).
\]
\end{lemma}

\begin{proof}
We may write $V=\oplus_\lambda (V_\lambda\ot \lambda)=\oplus_\lambda (\lambda\ot {_\lambda V})$ for representations $V_\lambda$ and ${_\lambda V}$ supported on $\mcl{G}_o$, and adopt similar expressions for $W$.  (Here the sums run over the simples $\lambda\in \pi$.)  Let $\supp^{\mcl{G}}$ and $\supp^{\mcl{G}_o}$ denote the cohomological supports for $\Coh(\mcl{G})$ and $\Coh(\mcl{G}_o)$ respectively.  We have
\[
\supp^{\mcl{G}}(V)=\cup_\lambda\supp^{\mcl{G}_o}({_\lambda V}),\ \ \supp^{\mcl{G}}(W)=\cup_\lambda \supp^{\mcl{G}_o}({_\lambda W}).
\]
We are assuming that $V$ centralizes the simples.  This implies an identification of the right handed and left handed supports for $V$,
\begin{equation}\label{eq:1046}
\cup_\lambda\supp^{\mcl{G}_o}({_\lambda V})=\supp^{\mcl{G}}(V)=\cup_\lambda\supp^{\mcl{G}_o}(V_\lambda)
\end{equation}
\par

We have that $\supp^{\mcl{G}}(V\ot W)$ is the support of cohomology
\[
\begin{array}{rl}
\Ext^\ast_{\Coh \mcl{G}}(k,\bigoplus_\lambda \lambda\ot(V\ot W))&\cong\Ext^\ast_{\Coh \mcl{G}}(k,\bigoplus_\lambda (V\ot\lambda\ot W))\\
&=\bigoplus_{\lambda,\mu}\Ext^\ast_{\Coh \mcl{G}_o}(k,V_\mu\ot {_{(\mu\lambda)^{-1}}W}),\\
&=\bigoplus_{\lambda,\mu}\Ext^\ast_{\Coh \mcl{G}_o}(k,V_\mu\ot {_{\lambda}W})
\end{array}
\]
where the sum runs over all invertibles $\lambda\in \pi$.  We employ the tensor product property for $\mcl{G}_o$ and equation~\eqref{eq:1046} to obtain
\[
\begin{array}{rl}
\supp^{\mcl{G}}(V\ot W)&=\bigcup_{\lambda,\mu}\supp^{\mcl{G}_o}(V_\mu\ot{_\lambda W})\\
&=\bigcup_{\lambda,\mu}\big(\supp^{\mcl{G}_o}(V_\mu)\cap\supp^{\mcl{G}_o}({_\lambda W})\big)\\
&=\big(\bigcup_\mu \supp^{\mcl{G}_o}(V_\mu)\big)\cap\supp^{\mcl{G}}(W)\ =\supp^{\mcl{G}}(V)\cap\supp^{\mcl{G}}(W).
\end{array}
\]
\end{proof}

\subsection{Support for coherent sheaves on finite group schemes}

Recall that a finite group scheme $\mcl{G}$ is called infinitesimal if it is connected.  Equivalently, $\mcl{G}$ is infinitesimal if $\O(\mcl{G})$ is local.  Recall also that for any embedding $\mcl{G}\to \mcl{H}$ into smooth $\mcl{H}$, we have the corresponding deformation sequence $\hat{\O}_{\mcl{H}/\mcl{G}}\to \hat{\O}_\mcl{H}\to \O(\mcl{G})$ (see Example \ref{ex:1}).  We take in this instance $Z=\hat{\O}_{\mcl{H}/\mcl{G}}$, and as usual $\msf{Y}=\Proj(\Ext_{\Coh(\mcl{G})}(k,k))_{\rm red}$.

\begin{lemma}\label{lem:surjO}
Let $\mcl{G}$ be an infinitesimal group scheme, with chosen embedding $\mcl{G}\to \mcl{H}$ into a smooth algebraic group.  Then the associated map $\kappa:\msf{Y}\to \mbb{P}(m_Z/m_Z^2)$ of~\eqref{eq:kappa} is a closed embedding.
\end{lemma}

\begin{proof}
Take $E=\Ext^\ast_{\Coh(\mcl{G})}(k,k)_{\rm red}$.  We have an algebra isomorphism $\O(\mcl{G})\cong k[x_1,\dots, x_n]/(x_i^{d_i}:1\leq i\leq n)$, abstractly~\cite{waterhouse}, so that $E$ is a graded polynomial ring generated in degree $2$, $E=\Sym(E^2)$.  Therefore $\kappa$ is a closed embedding by Lemma~\ref{lem:surj0}.
\end{proof}

\begin{theorem}\label{thm:Pio}
Let $\mcl{G}$ be an infinitesimal group scheme.  Then for arbitrary $V$ and $W$ in $\Coh(\mcl{G})$ we have
\[
\supp_\msf{Y}(V\ot W)=\supp_\msf{Y}(V)\cap\supp_\msf{Y}(V),
\]
where $\supp_\msf{Y}$ denotes the cohomological support.
\end{theorem}

\begin{proof}
Fix an embedding $\mcl{G}\to \mcl{H}$ into a smooth algebraic group.  By the weak tensor product property (and Lemma~\ref{lem:surjO}) we have a containment
\[
\supp_\msf{Y}(V\ot W)\subset \big(\supp_{\msf{Y}}(V)\cap\supp_{\msf{Y}}(W)\big).
\]
We want to show now that if a given closed point in $\msf{Y}$ is in both $\supp_\msf{Y}(V)$ and $\supp_\msf{Y}(W)$, then it is in $\supp_\msf{Y}(V\ot W)$.  Equivalently, according to Corollary~\ref{cor:hopf_hyp}, we want to show that if $\hat{\O}_\mcl{H}/(f)$ is a hypersurface along which both $V$ and $W$ are non-perfect, then the product $V\ot W$ is also non-perfect over $\hat{\O}_\mcl{H}/(f)$.
\par

Let us fix such a function $f\in \hat{\O}_{\mcl{H}/\mcl{G}}$ and take $\O=\hat{\O}_\mcl{H}$.  It is well-known that imperfection of an object $M$ over the local (commutative) hypersurface $\O/(f)$ implies that the trivial module $k$ is in the thick subcategory $\langle M\rangle$ generated by $M$, in $D_{coh}(\O)$ (see e.g.~\cite{takahashi10}).  By exactness of the operation $V\ot-$ this implies $V\in \langle V\ot M\rangle$.  Hence, perfection of the product $V\ot M$ implies
\[
V\in \langle V\ot M\rangle \subset \operatorname{perf}(\O/(f)),
\]
so that $V$ is seen to be perfect.  Considering the case $M=W$, we observe the desired implication
\[
x\in \supp_\msf{Y}(W)\ \text{while }x\notin \supp_\msf{Y}(V\ot W)\ \Rightarrow\ x\notin \supp_\msf{Y}(V),
\]
i.e.\ if $W$ is not perfect over a given hypersurface, while $V\ot W$ is perfect, then $V$ {\it must} be perfect over that hypersurface.
\end{proof}

\begin{remark}
In the proof, one can provide a more direct argument to find that $k\in \langle M\rangle$ whenever $M$ is non-perfect over a commutative local hypersurface, using arguments of Carlson-Iyengar~\cite{carlsoniyengar15}.  We will follow such a line of reasoning in the proof of Theorem~\ref{thm:fun1} below (see also Remark~\ref{rem:ok}).
\end{remark}

\begin{remark}\label{rem:CI}
Theorem~\ref{thm:Pio} can alternatively be proved by employing the classification of thick subcategories in $\Coh(\mcl{G})$~\cite{carlsoniyengar15}.  One uses this classification to reduce to the case of a product of Carlson modules, then observes the desired result by applying~\cite[Corollary 4.1]{pevtsovawitherspoon09}.
\end{remark}

By applying Lemma~\ref{lem:1026} we obtain a result for general finite group schemes.

\begin{corollary}\label{cor:Pi}
Let $\mcl{G}$ be an arbitrary finite group scheme, and let $\pi$ denote the discrete subgroup of closed points in $\mcl{G}$.  For $W$ arbitrary in $\Coh(\mcl{G})$, and $V$ in the Drinfeld centralizer $Z^{\Coh(\pi)}(\Coh(\mcl{G}))$, we have an identification
\begin{equation}\label{eq:1138}
\supp_\msf{Y}(F(V)\ot W)=\supp_\msf{Y}(F(V))\cap\supp_\msf{Y}(W),
\end{equation}
where $\supp_\msf{Y}$ denotes the cohomological support for $\Coh(\mcl{G})$ and $F$ is the forgetful functor~\eqref{eq:forget}.
\end{corollary}

We would suggest from examples such as Example~\ref{ex:no_tpp} that the conditions on $V$ in the above result are, in a global sense, optimal.  More specifically, the tensor product property simply cannot hold in the non-connected case, and the above corollary gives (arguably) the strongest sense in which the tensor product property does hold for coherent sheaves on arbitrary $\mcl{G}$.

\begin{remark}
Due to our particular left/right conventions (see Section \ref{sect:side}), the $\rep(\pi)$-centralizing object $V$ in Corollary~\ref{cor:Pi} must appear on the \emph{left}.  However, if we take $V$ in the full Drinfeld center $Z(\Coh(\mcl{G}))$, then the analogous claim~\eqref{eq:1138} is side-independent.
\end{remark}

\subsection{A remark on centrality and nonbraided TPP}
\label{sect:Z_supp}

We propose an alternate version of the tensor product property which is more robust that the usual version \eqref{eq:tpp}, and is satisfied by all examples which currently appear in the literature.  Our particular framing here is informed by $4$-dimensional topological field theory \cite{brochierjordansnyder,brochierjordansafronovsnyder}, although we don't elaborate on this point here.  From a more pragmatic perspective, we want to provide a version of the tensor product property which \emph{is} satisfied by the examples $\Coh(\mcl{G})$ considered above.
\par

One can view any nonbraided tensor category $\msc{C}$ as a tensor category \emph{over} its Drinfeld center $Z(\msc{C})$.  We recall that a tensor category $\msc{C}$ over a given braided tensor category $\msc{Z}$ is a tensor category equipped with a central tensor functor $F:\msc{Z}\to \msc{C}$.  Formally, such central $F$ consists of a choice of tensor functor $F_0:\msc{Z}\to \msc{C}$ and a braided lift of this functor to the Drinfeld center $F_1:\msc{Z}\to Z(\msc{C})$.  Informally, we simply mean that for arbitrary $V$ in $\msc{Z}$ and $W$ in $\msc{C}$ there is a commutativity relation $F(V)\ot W\cong W\ot F(V)$ which is natural in $V$ and $W$.

\begin{example}
For the quantum Borel $u_q(\mfk{b})$, the forgetful functor $\rep u_q(\mfk{g})\to \rep u_q(\mfk{b})$ from the associated quantum group admits a central structure $\rep u_q(\mfk{g})\to Z(\rep u_q(\mfk{b}))$.  This follows from the explicit form of the $R$-matrix for $u_q(\mfk{g})$.
\end{example}

For nonbraided $\msc{C}$ we consider again the support $\supp_\msf{Y}(V)=\Supp_{\msf{Y}} \Ext_{\msc{C}}^\ast(V,V)^\sim$, but now adopt a tensor relation which expresses a ``linearity" property with respect to the action of the Drinfeld center $Z(\msc{C})$,
\begin{equation}\label{eq:1228}
\supp_\msf{Y}(F(V)\ot W)\overset{\rm question}=\supp_\msf{Y}(F(V))\cap \supp_\msf{Y}(V).
\end{equation}
Here $V$ is in $Z(\msc{C})$, $W$ is in $\msc{C}$, and $F:Z(\msc{C})\to \msc{C}$ is the forgetful functor.  The implicit half-braiding on objects in $Z(\msc{C})$ implies that there is no sidedness in the expression~\eqref{eq:1228}.  One considers the relation~\eqref{eq:1228} as a replacement for the usual tensor product property \eqref{eq:tpp}.

More generally, for $F:\msc{Z}\to \msc{C}$ any central tensor functor we can consider the relation analogous to \eqref{eq:1228} for $\msc{Z}$ acting on $\msc{C}$.  We leave the proof of the following Lemma to the interested reader.

\begin{lemma}\label{lem:c_tpp}
The following are equivalent:
\begin{enumerate}
\item The relation \eqref{eq:1228} holds for all $V$ in $Z(\msc{C})$ and $W$ in $\msc{C}$.
\item The analogous relation \eqref{eq:1228} holds for any braided $\msc{Z}$ equipped with a central tensor functor $F:\msc{Z}\to \msc{C}$.
\end{enumerate}
Furthermore, when $\msc{C}$ itself admits a braiding then either of the above two points are equivalent to the usual tensor product property \eqref{eq:tpp} for $\msc{C}$.
\end{lemma}

\begin{definition}
We say that support for a tensor category $\msc{C}$ satisfies the \emph{centralized tensor product property} if equation \eqref{eq:1228} holds at all $V$ in $Z(\msc{C})$ and $W$ in $\msc{C}$.
\end{definition}

Corollary \ref{cor:Pi} implies that, for sheaves on non-connected $\mcl{G}$ for example, the standard tensor product property~\eqref{eq:tpp} generally fails while the central version~\eqref{eq:1228} still holds.  More explicitly, support for the category $\Coh(\mcl{G})$ of coherent sheaves on the smash product $\mcl{G}=(\mbb{G}_{a(1)}\times \mbb{G}_{a(1)})\rtimes \mbb{Z}/2\mbb{Z}$ was shown to \emph{not} have the tensor product property \eqref{eq:tpp}, in Example \ref{ex:no_tpp}.  But we have shown that support for $\Coh(\mcl{G})$, for this particular $\mcl{G}$, does have the central tensor product property outlined above.  This alternate tensor product property also holds in all of the (counter)examples considered in~\cite{bensonwitherspoon14,plavnikwitherspoon18}, where the standard tensor product property was shown to fail.  
\par

Furthermore, for pointed Hopf algebras with a ``bad" choice of grouplikes the centralized tensor product property \eqref{eq:1228} is, seemingly, a more appropriate relation to consider, in comparing with the usual version \eqref{eq:tpp}.  One can see our explicit analysis of quantum complete intersections with various choices of grouplikes given in Theorems \ref{thm:fun1} and \ref{thm:qci} below.

\section{Bosonized quantum complete intersections}
\label{sect:fun-q}

We consider support for some non-classical analogs of functions on infinitesimal group schemes, over the complex numbers.  We focus on a class of examples associated to skew polynomial rings with odd order skewing parameters.  We refer to these examples as quantum complete intersections, although the nomenclature \emph{quantum linear spaces} is also common \cite{andruskiewitschschneider98,bonteanikshych}.

\subsection{Bosonized quantum complete intersections}
\label{sect:qci_1}

Consider a skew polynomial ring
\[
\msf{U}^+=\mbb{C}\langle\!\langle x_1,\dots,x_n\!\rangle\!\rangle/(x_ix_j-q_{ij}x_jx_i:i\neq j),\ \ q_{ij}=q^{a_{ij}},
\]
defined by a choice of root of unity $q$ of odd order, and an integer matrix $[a_{ij}]$ which is skew symmetric away from the diagonal $a_{ij}=-a_{ji}$ and of constant value $a_{ii}=1$ along the diagonal.  Fix $l=\operatorname{ord}(q)$.
\par

Take now $G^\vee$ a finite abelian group equipped with a bilinear form $q^{(-,-)}:G^\vee\times G^\vee\to \mbb{C}^\times$ such that, for distinguished elements $e_1,\dots, e_n$ and $\chi_1,\dots, \chi_n$ in $G^\vee$, we have
\begin{equation}\label{eq:1549}
q^{(e_i,e_j)}=q^{a_{ij}}\ \ \text{and}\ \ q^{(e_i,\chi_j)}=q^{(e_i,e_j)_{\rm alt}}\ \ \forall\ \ 1\leq i\leq n.
\end{equation}
Here $(-,-)_{\rm alt}$ is the alternating form $(e_i,e_j)_{\rm alt}=\frac{1}{2}(a_{ij}-a_{ji})$.
\par

One can construct such $G^\vee$ by simply considering a rank $2n$ elementary abelian $l$-group with free generators $e_1,\dots, e_n, \chi_1,\dots, \chi_n$, and form defined on these generators as above.  On the other hand, when the determinant of the matrix $[a_{ij}]$ is invertible mod $l$, the required elements $\chi_j$ already exist in the rank $n$ elementary abelian $l$-group generated by the $e_i$, with form defined in accordance with \eqref{eq:1549}.
\par

In any case, we fix $G^\vee$ as above and take $\Lambda=\mbb{C}G$ the group ring of the character group for $G^\vee$.  The form on $G^\vee$ endows the tensor category $\rep(\Lambda)$ with a braided structure.  In terms of homogeneous vectors, the braiding on $\rep(\Lambda)$ is given explicitly by
\[
c_{V,W}:V\ot W\to W\ot V,\ \ v\ot w\mapsto q^{(\deg v,\deg w)}w\ot v.
\]
The algebra $\msf{U}^+$ then becomes an algebra in $\rep(\Lambda)$, or rather in its $\operatorname{Ind}$-completion $\Rep(\Lambda)$, by taking each generator $x_i\in \msf{U}^+$ to be of $G^\vee$-degree $e_i$.
\par

The coproduct $\Delta(x_i)=x_i\ot 1+1\ot x_i$ on $\msf{U}^+$ endows it with the structure of a Hopf algebra in $\Rep(\Lambda)$, and we have the quotient Hopf algebra
\[
\msf{u}^+=\msf{U}^+/(x_i^l:1\leq i\leq n)
\]
in $\rep(\Lambda)$.  By applying bosonization (Section~\ref{sect:bhb}) we obtain Hopf algebras $\msf{U}=\msf{U}^+\rtimes \Lambda$ and $\msf{u}=\msf{u}^+\rtimes \Lambda$, in $Vect$, and the quotient map $\msf{U}\to \msf{u}$ deforms $\msf{u}$ along the trivial commutative Hopf subalgebra $Z=\mbb{C}\b{x_1^l,\dots, x_n^l}$.  (By trivial, we mean that $\Lambda$ acts trivially on $Z$.)  So $\msf{u}$ is geometrically Chevalley with Chevalley integration $\msf{U}\to \msf{u}$.  We refer to $\msf{u}$ constructed in this manner as a bosonized algebra of functions on a quantum complete intersection, or just \emph{a bosonized quantum complete intersection}.
\par

To be clear, as an algebra
\begin{equation}\label{eq:1152}
\msf{u}=\big(\mbb{C}\langle x_1,\dots, x_n\rangle/(x_ix_j-q_{ij}x_jx_i,\ x_i^l)_{i,j}\big)\rtimes \Lambda
\end{equation}
and the coproduct is given by the formula $\Delta(x_i)=x_i\ot 1+K_i\ot x_i$, where $K_i$ is the character $K_i=q^{(e_i,-)}:G^\vee\to \mbb{C}^\times$.

\begin{remark}
The name ``quantum complete intersection" is somewhat unsatisfactory, as it does not reference the Hopf structure on this algebra.  Such algebras are probably more accurately understood as $q$-analogs of functions on first Frobenius kernels $(\mbb{G}_{a}^n)_{(1)}$ in additive group schemes $\mbb{G}_a^n$ of varying rank.
\end{remark}

\begin{remark}
There is some ambiguity in our presentation of the Hopf algebras $\msf{u}$, as we only specify the grouplikes within a certain range of parameters.  We consider the ``standard" choice of grouplikes for such a quantum complete intersection in Section \ref{sect:qci_2} below.
\end{remark}

\begin{theorem}\label{thm:fun1}
Consider $\msf{u}$ a bosonized quantum complete intersection, with grouplikes $G$ as above.  Then for arbitrary $V$ and $W$ in $\rep(\msf{u})$, we have
\[
\supp_\msf{Y}(V\ot W)=\supp_\msf{Y}(V)\cap\supp_\msf{Y}(W).
\]
\end{theorem}

We prove Theorem~\ref{thm:fun1} in Section~\ref{sect:q-prf} below.  Before giving the proof, let us compare the above result with our findings for functions on non-connected group schemes $\O(\mcl{G})$, Corollary~\ref{cor:Pi}.  We refer to the situation considered in the present section as the ``$q$-symmetric" situation.
\par

In both the non-connected case of Corollary~\ref{cor:Pi} and the $q$-symmetric case, orbits of objects under the tensor-action of the simples appear in the proof.  Furthermore, the centralizing category $Z^{\rep(\Lambda)}(\rep(\msf{u}))$ also appears in the proof of Theorem~\ref{thm:fun1} (and in the statements of Theorem \ref{thm:qci} below).  There is, however, a fundamental distinction between these two cases.  Namely, in the $q$-symmetric setting, the parametrizing subalgebra $Z\subset \msf{U}$ is a Hopf subalgebra in $\msf{U}$.  This implies that modules over each hypersurfaces $\msf{U}/(f)$ have a \text{bimodule} structure over $\msf{u}$.  Such normality does \emph{not} hold in the case of functions on a non-connected group scheme, in general.
\par

Indeed, in the case of non-connected $\mcl{G}$ the ``hard" inclusion
\[
\big(\supp_\msf{Y}(V)\cap\supp_\msf{Y}(W)\big)\subset \supp_\msf{Y}(V\ot W)
\]
still holds on $\Coh(\mcl{G})$.  It is the opposite inclusion that fails to hold.  The opposite inclusion in the $q$-symmetric setting is forced by (co)normality of the deformation $\msf{U}\to \msf{u}$ (Corollary~\ref{cor:weak_tpp2}).  In considering Theorem \ref{thm:qci} below, which makes a claim mirroring almost exactly that of Corollary \ref{cor:Pi}, we would still propose that the situations of \ref{thm:qci} and \ref{cor:Pi} are rather dissimilar.  And so, we view the uniform nature of Corollary \ref{cor:Pi} and Theorem \ref{thm:qci} as a testament to the general framework outlined in Section \ref{sect:Z_supp}.

\subsection{Proof of Theorem~\ref{thm:fun1}}
\label{sect:q-prf}

\begin{proof}[Proof of Theorem~\ref{thm:fun1}]
We first note that $\Ext^\ast_\msf{u}(\mbb{C},\mbb{C})_{\rm red}$ is a polynomial ring \cite{berghoppermann08}, so that the map $\kappa:\msf{Y}\to \mbb{P}$ is a closed embedding by Lemma~\ref{lem:surj0}.  Hence we have an inclusion of supports $\supp_\msf{Y}(V\ot W)\subset \big(\supp_\msf{Y}(V)\cap\supp_\msf{Y}(W)\big)$, by Corollary~\ref{cor:weak_tpp2}.  We seek the opposite inclusion.

Fix $f\in m_Z$ with non-trivial linear component.  According to Lemma~\ref{lem:gen_lem2}, it suffices to prove that $\mbb{C}$ is in the thick subcategory $\langle \lambda\ot M:\lambda\in \operatorname{Irrep}(\Lambda)\rangle$ for finite non-perfect $M$ over $\msf{U}^+/(f)$.  Fix $R=\msf{U}^+/(f)$, $Q=\msf{U}^+$, $Z=\mbb{C}\b{f}$, and consider the corresponding deformation sequence $Z\to Q\to R$.  Fix also $T$ to be the central subalgebra $\mbb{C}\b{x_1^l,\dots,x_n^l}$ in $Q$.
\par

Take $K_Q=Q\ot_Z K_Z$ the corresponding dg resolution of $R$, with $K_Z$ the Koszul resolution of $k$ over $Z$, as in Section \ref{sect:koszul}.  We have the reduction $\mbb{C}\ot_Q K_Q$, which is isomorphic to the dual numbers $\mbb{C}[\varepsilon]=\mbb{C}[d_f]/(d_f^2)$ with generator in degree $-1$, and may consider the reduction
\[
\mbb{C}\ot^{\rm L}_Q-:K_Q\text{-dgmod}\to \mbb{C}[\varepsilon]\text{-dgmod}.
\]
We consider here, specifically, dg $K_Q$-modules with finite cohomology over $T$ and dg $\mbb{C}[\varepsilon]$-modules with finite-dimensional cohomology.
\par

By adjunction we see that
\[
\Ext^\ast_{K_Q}(M,\mbb{C})=\Ext^\ast_{\mbb{C}[\varepsilon]}(\mbb{C}\ot^{\rm L}_Q M,\mbb{C})
\]
so that $\mbb{C}\ot^{\rm L}_Q M$ is non-perfect over the dual numbers whenever $M$ is non-perfect over $K_Q$.  By a dg version of Hopkins Theorem~\cite[Theorem 4.4]{carlsoniyengar15} for $\mbb{C}[\varepsilon]$, we have then
\begin{equation}\label{eq:1244}
\mbb{C}\in \langle \mbb{C}\ot^{\rm L}_QM\rangle\ \ \text{whenever}\ \ M\text{ is non-perfect over }K_Q,
\end{equation}
where $\langle L\rangle$ denotes the thick subcategory in $D_{fin}(\mbb{C}[\varepsilon])$ generated by the given object $L$.  We claim that this implies
\[
\mbb{C}\in \langle \lambda\ot M:\lambda\in \operatorname{Irrep}(\Lambda)\rangle\subset D_{coh}(K_Q).
\]
Let us prove this claim.
\par

We have the $q$-exterior dg algebra
\[
E=\wedge^\ast_q\mbb{C}\{d_{x_1},\dots, d_{x_n}\}=\mbb{C}\langle d_{x_i}:1\leq i\leq n\rangle/(d_{x_i}d_{x_j}+q_{ij}d_{x_j}d_{x_i}, d_{x_i}^2)
\]
in the braided category of chain complexes over $\rep(\Lambda)$, where the $d_{x_i}$ are of cohomological degree $-1$ and respective $G^\vee$-degree $e_i$.  We consider the corresponding Koszul resolution $F=(E\ot Q, d_F)$ of $\mbb{C}$ over $Q$, with explicit differential
\[
d_F(d_{x_{i_r}}\dots d_{x_{i_1}}\ot a)=\sum_j(-1)^{j+1}(\prod_{t<j}q_{i_ji_t})d_{x_{i_r}}\dots\hat{d_{x_{i_j}}}\dots d_{x_{i_1}}\ot x_{i_j}a.
\]
This complex is in fact a dg {\it bimodule} over $Q$ via symmetry of $Q$.  Namely, we have the well-defined left action of $Q$ on $F$ given by the formula
\begin{equation}\label{eq:1625}
b\cdot(v\ot w):=q^{(\deg(b),\deg(v))_{\rm alt}}v\ot bw=q^{(\deg(b),\chi_v)}v\ot bw,
\end{equation}
where $(-,-)_{\rm alt}$ is as in \eqref{eq:1549} and $\chi_v=\sum_j \chi_{i_j}$ for $v=\prod_j d_{x_{i_j}}$.  Hence tensoring by $F$ provides an endofunctor
\[
F\ot^{\rm L}_Q-=F\ot_Q-:D_{coh}(K_Q)\to D_{coh}(K_Q)
\]
which is isomorphic to the endofunctor $\mbb{C}\ot^{\rm L}_Q-$, which we interpret as the composition of the reduction to $\mbb{C}[\varepsilon]$ composed with restriction $D_{fin}(\mbb{C}[\varepsilon])\to D_{coh}(K_Q)$.
\par

Let $\msc{D}\subset \rep(\Lambda)$ denote the tensor subcategory generated by the simples associated to the characters $\chi_j$.  By construction of $F$ and equation \eqref{eq:1625}, we have for any dg $K_Q$-module $M$,
\[
\begin{array}{l}
\mbb{C}\ot^{\rm L}_QM\cong F\ot^{\rm L}_QM\in \langle E^i\ot M:0\leq i\leq n\rangle\\
\hspace{4cm}\subset \langle \chi\ot M:\chi\text{ a simples in }\msc{D}\rangle\\
\hspace{5cm}\subset \langle \lambda\ot M:\lambda\text{ a simples in }\rep(\Lambda)\rangle
\end{array}
\]
By~\eqref{eq:1244}, imperfection of $M$ over $K_Q$ then implies
\[
\mbb{C}\in \langle \lambda\ot M:\lambda\in \operatorname{Irrep}(\Lambda)\rangle,
\]
as desired.  The tensor product property now follows as an application of Lemma~\ref{lem:gen_lem2}.
\end{proof}

\begin{remark}\label{rem:1586}
The deformation $Q\to R$ employed in the proof does not fit into the setup of Section \ref{sect:setup}, as $Q$ is not finite over the parametrizing algebra $Z=k\b{f}$ in this case.  (Although, $Q$ is finite over complete local central subalgebra $T\supset Z$.)  However, all we use in this case is the associated Koszul resolution $K_Q\to R$, which is a valid construction for any deformation $Q\to R$ with complete, formally smooth parametrizing algebra $Z$.
\end{remark}

\begin{remark}\label{rem:ok}
By considering Koszul duality $D_{fin}(\mbb{C}[\varepsilon])\cong D_{coh}(\mbb{C}[t])$, $\deg(t)=2$, and the fact that $\mbb{C}[t]$ is a PID, one can basically classify thick subcategories in $D_{fin}(\mbb{C}[\varepsilon])$ by hand.  (The classification is rather easy for the stable category $D_{sing}(\mbb{C}[\varepsilon])$, which is sufficient for our needs.)  That is to say, one doesn't actually have to understand the dg version of Hopkins' theorem, in full regalia, in order to prove of Theorem~\ref{thm:fun1}.
\end{remark}

\subsection{Changing grouplikes (standard QCIs)}
\label{sect:qci_2}

Consider again the positive algebras of Section \ref{sect:qci_1}, which we relabel as
\[
\msf{A}_q^+(P)=\mbb{C}\langle\!\langle x_1,\dots,x_n\!\rangle\!\rangle/(x_ix_j-q_{ij}x_jx_i)_{i,j}
\]
\[
\text{and}\ \ \msf{a}^+_q(P)=\mbb{C}\langle x_1,\dots, x_n\rangle/(x_ix_j-q_{ij}x_jx_i,\ x_i^l)_{i,j}.
\]
Here $q$ is a root of unity of odd order $l$ and $P=[a_{ij}]$ is an integer matrix for which $q_{ij}=q^{a_{ij}}$.  As before we assume $a_{ij}=-a_{ji}$ off the diagonal, and all $a_{ii}=1$.  Having fixed such $P$, we adopt the more efficient notation $A_q^+=A_q^+(P)$ and $\msf{a}^+_q=\msf{a}^+_q(P)$.
\par

Consider $G^{\rm st}$ the rank $n$ elementary abelian $l$-group with generators $K_i$, and let $G^{\rm st}$ act on $\msf{A}^+_q$ and $\msf{a}^+_q$ in the expected way $K_i\cdot x_j=q_{ij}x_j$.  Then the smash products
\[
\msf{A}_q=\msf{A}^+_q\rtimes G^{\rm st}\ \ \text{and}\ \ \msf{a}_q=\msf{a}^+_q\rtimes G^{\rm st}
\]
become Hopf algebras under the coproduct $\Delta(x_i)=x_i\ot 1+K_i\ot x_i$.
\par

The Hopf algebra $\msf{A}_q$ integrates $\msf{a}_q$, with parametrizing Hopf subalgebra $Z=\mbb{C}\b{x_1^l,\dots, x_n^l}$, and the positive (coideal) subalgebra $\msf{A}^+_q$ deforms $\msf{a}^+_q$ along $\Spf(Z)$ as well.

The Hopf algebra(s) $\msf{a}_q$ constructed as above are what most people would refer to as \emph{the} bosonized quantum complete intersection, or \emph{the} quantum linear space, associated to a root of unity $q$ and matrix $[a_{ij}]$.  The algebras $\msf{u}$ of Section \ref{sect:qci_1} and $\msf{a}_q$ of course differ only in the choice of grouplikes and, by considering the delta characters $\delta_j=q^{(-,e_j)}\chi_j^{-1}$ in $G^\vee$, we see that there are Hopf inclusion
\[
\operatorname{incl}:\msf{A}_q\to \msf{U}\ \ \text{and}\ \ \operatorname{incl}:\msf{a}_q\to \msf{u}.
\]
These inclusions are the identity on their respective positive parts, and send the $K_i$ in $G^{\rm st}$ to the $K_i$ in $G$.
\par

We let $\Lambda^{\rm st}$ and $\Lambda$ denote the group rings of $G^{\rm st}$ and $G$ respectively, where $G$ is as in Section \ref{sect:qci_1}.

\begin{theorem}\label{thm:qci}
Consider $\msf{a}_q=\msf{a}_q(P)$ the bosonized quantum complete intersection associated to given odd order $q$ and matrix $P=[a_{ij}]$.   For $V$ in the Drinfeld centralizer $Z^{\rep(\Lambda^{\rm st})}(\rep \msf{a}_q)$, and arbitrary $W$ in $\rep(\msf{a}_q)$, we have
\[
\supp_\msf{Y}(F(V)\ot W)=\supp_\msf{Y}(F(V))\cap\supp_\msf{Y}(W).
\]
\end{theorem}

As before, $F:Z^{\rep(\Lambda^{\rm st})}(\rep \msf{a}_q)\to \rep(\msf{a}_q)$ is the forgetful functor.

\begin{proof}
We further simplify our notation and take $\msf{A}=\msf{A}_q(P)$ and $\msf{A}^+=\msf{A}^+_q(P)$.  The comultiplication on $\msf{A}$ provides $\msf{A}^+$ with a $\Lambda^{\rm st}$-comodule algebra structure,
\[
\rho:\msf{A}^+\overset{\Delta}\longrightarrow \msf{A}\ot \msf{A}^+\overset{\pi\ot 1}\longrightarrow \Lambda^{\rm st}\ot \msf{A}^+.
\]
This coaction descends to a coaction on all hypersurface algebras $\msf{A}^+/(f)$, for $f\in m_Z$.  Via this coaction we get a $\rep(\Lambda^{\rm st})$-action on the category of modules over any hypersurface algebra $\msf{A}^+/(f)$, as well as its corresponding derived category.
\par

Consider $f\in m_Z$ with nonvanishing reduction in $m_Z/m_Z^2$.  We claim that any non-perfect, finite, $M$ over $\msf{A}^+/(f)$ is such that
\begin{equation}\label{eq:139}
\mbb{C}\in \langle \mu\ot M:\mu\in \operatorname{Irrep}(\Lambda^{\rm st})\rangle\subset D_{coh}(\msf{A}^+/(f)).
\end{equation}
It was already shown in the proof of Theorem \ref{thm:fun1} that for such $M$
\begin{equation}\label{eq:143}
\mbb{C}\in \langle \lambda\ot M:\lambda\in \operatorname{Irrep}(\Lambda)\rangle\subset D_{coh}(\msf{U}^+/(f))=D_{coh}(\msf{A}^+/(f)).
\end{equation}
Now, we have the restriction map $\rep(\Lambda)\to \rep(\Lambda^{\rm st})$, and claim that for any character $\lambda\in \operatorname{Irrep}(\Lambda)$ with restriction $\mu\in\operatorname{Irrep}(\Lambda^{\rm st})$, the $\msf{A}^+=\msf{U}^+$-modules $\lambda\ot M$ and $\mu\ot M$ are literally equal.  Indeed, this just follows from the fact that the comultiplication on $\msf{U}$ sends the positive subalgebra $\msf{U}^+$ into the subalgebra $(\msf{U}^+\rtimes G^{\rm st})\ot \msf{U}^+$, so that the each endofunctor $\lambda\ot-$ of $\msf{U}^+$-mod depends only on the restriction of $\lambda$ to $G^{\rm st}$.  (The point here is that the grouplikes $K_i$ live in the subgroup $G^{\rm st}\subset G$.)
\par

Consider finally a product $V\ot W$ of a $\rep(\Lambda^{\rm st})$-centralizing representation $V$ with an arbitrary representation $W$ over $\msf{a}_q$.  Suppose that $W$ is non-perfect over a given hypersurface $\msf{A}/(f)$ while $V\ot W$ is perfect.  One uses \eqref{eq:139}, and argues as in Lemma \ref{lem:gen_lem2}, to see that $V$ must be perfect in this case.  Rather, we find that hypersurface support for $\msf{a}_q$ satisfies the stated tensor product property.  Since (reduced) cohomology for $\msf{a}_q$ is an affine space, cohomological support also satisfies this tensor product property.
\end{proof}

\section{($q$-)Regular sequences}

We consider a particular noncommutativization of the notion of a regular sequence.  Such sequences are employed in proofs of the tensor product property for cohomological support, in the examples of Section~\ref{sect:last}. 

\subsection{A choice of ambient braided fusion category $\msc{D}$}

Fix $G^\vee$ a finite abelian group with a chosen (exponentiated) form $q^{(-,-)}:G^\vee\times G^\vee\to k^\times$.  We are free to assume $\operatorname{ord}(q)|\exp(G^\vee)$ and the form $(-,-)$ can be seen as an additive form which takes values in $\mbb{Z}/\exp(G^\vee)\mbb{Z}$.  We view $G^\vee$ as the group of characters for its dual group in characteristic $0$, or for the dual Hopf algebra in finite characteristic.
\par

We fix
\[
\msc{D}=\text{The braided fusion category $Vect_{G^\vee}$, with braiding given by the form }q^{(-,-)}.
\]
Recall that $Vect_{G^\vee}$ is the fusion category of $G^\vee$-graded vector spaces.  The irreducibles in $\msc{D}$ are explicitly the $1$-dimensional vector spaces supported at a given character.  As before, in terms of homogeneous vectors, the braiding on $\msc{D}$ is given by
\[
c_{V,W}:V\ot W\to W\ot V,\ \ v\ot w\mapsto q^{(\deg v,\deg w)}w\ot v.
\]
We impose no symmetry condition on the form $(-,-)$, and we remark that the case $\msc{D}=Vect$, $G=\{1\}$, is still of interest.

\subsection{Regular sequences}
\label{sect:qreg}

Consider $A$ a dg algebra in $\msc{D}$, or more precisely the category of complexes over the $\operatorname{Ind}$-category $\operatorname{Ind}\msc{D}$ of possibly infinite-dimensional $G^\vee$-graded vector spaces.  So, we allow $A$ to be infinite-dimensional, provided it splits into character spaces $A=\oplus_\chi A_\chi$.

We let $D(A)$ denote the derived category of arbitrary dg modules.  When $A$ is finite over a central Noetherian subalgebra $T$ of cocycles, one can replace $D(A)$ with $D_{coh}(A)$, the subcategory of dg modules with finitely generated cohomology over $T$, in all of the results of this section.

\begin{definition}[\cite{kirkmankuzzhang15}]\label{def:reg_seq}
A $q$-regular sequence $x_1,\dots, x_n$ in $A$ is a sequence such that
\begin{enumerate}
\item[(a)] the $x_i$ are homogenous cocycles, with respect to the character and cohomological gradings, and
\item[(b)] for each $1\leq j\leq n$, $x_j$ is a non-zero divisor in $A/(x_{j+1},\dots, x_n)$ which is $\chi_j$-central for some character $\chi_j$, in the sense that
\[
bx_j=\pm\!\ q^{(\deg b, \chi_j)}x_jb
\]
for all (homogenous) $b\in A/(x_{j+1},\dots, x_n)$, where $\pm=(-1)^{|x_j||b|}$.
\end{enumerate}
In the case $\msc{D}=Vect$, $q=1$, we refer to such a sequence simply as a regular sequence.
\end{definition}

Given such $A$ with fixed $q$-regular sequence $\{x_1,\dots, x_n\}$ we take $A_j=A/(x_{j+1},\dots, x_n)$, and have a sequence of projections
\[
A_0\leftarrow A_1\leftarrow\dots \leftarrow A_n
\]
of dg algebra objects in $\msc{D}$.

We consider the categories of dg (bi)modules over the $A_j$.  For any object $V$ in $\msc{D}$, and dg $A_j$-bimodule $M$, we have the new bimodule $V\ot M$ with right action induced by that of $M$ and left action given by
\[
a\cdot(v\ot b)=q^{(\deg a,\deg v)}v\ot ab.
\]
Similarly, for any left dg module $M$ we have the new module $V\ot M$.  In this way $A_j\text{-dgmod}$ and $A_j\text{-dgbimod}$ become left module categories over $\msc{D}$.  Similarly, $\msc{D}$ acts on the derived categories $D(A_j)$.

\begin{lemma}\label{lem:1415}
For any $r\leq s$, the reduction $A_r\ot_{A_s}-:A_s\text{-\rm dgmod}\to A_r\text{-\rm dgmod}$ can be given a $\msc{D}$-linear structure, so that it is a map of $\msc{D}$-module categories, as can be its derived counterpart $A_r\ot^{\rm L}_{A_s}-:D(A_s)\to D(A_r)$.
\end{lemma}

By a $\msc{D}$-linear structure on a functor $F:\msc{M}\to \msc{N}$ between $\msc{D}$-module categories we mean a choice of natural isomorphism of functors $F(V\ot-)\cong V\ot F(-)$, for each $V$ in $\msc{D}$, which is additionally associative and natural in $V$.  Formally speaking, triangulated categories cannot serve as module categories over $\msc{D}$, as they are not abelian.  However, by an action of $\msc{D}$ on a triangulated category $\msc{T}$ here we just mean a choice of associative action map $\ot:\msc{D}\times \msc{T}\to \msc{T}$ which preserves triangles in $\msc{T}$ and sends exact sequences in $\msc{D}$ to triangles in $\msc{T}$.  We would refer to such $\msc{T}$ simply as \emph{triangulated} module categories.

\begin{proof}[Proof of Lemma~\ref{lem:1415}]
For any $V$ in $\msc{D}$ and dg $A_s$-module $M$, we have the natural map
\[
V\ot (A_r\ot M)\to  A_r\ot_{A_s} (V\ot M),\ \ v\ot( a\ot m)\mapsto q^{-(\deg a,\deg v)} a\ot (v\ot m).
\]
Since the quotient $A_s\to A_r$ is a map in $\msc{D}$, and thus preserves the character grading, the above natural isomorphism respects the relations of the tensor product over $A_s$, and hence descends to a well-defined map
\begin{equation}\label{eq:1655}
V\ot (A_r\ot_{A_s}M)\overset{\sim}\to  A_r\ot_{A_s} (V\ot M).
\end{equation}
One similarly constructs the inverse to see that \eqref{eq:1655} is an isomorphism.  This isomorphism is seen to be natural in $V$ and $M$, and also respects the tensor structure on $\msc{D}$, so that we embue the reduction functors $F_{r,s}=A_r\ot_{A_s}-$ with the structure of a map of $\msc{D}$-module categories.  One proves the result for the derived categories similarly.
\end{proof}

\subsection{Regular sequences and thick subcategories}

Take $A$ as in Section~\ref{sect:qreg}, with a fixed $q$-regular sequence $\{x_1,\dots, x_n\}$.  For an object $M$ in $A_j\text{-dgmod}$ we let $\langle M\rangle$ denote the thick subcategory in $D(A_j)$ generated by $M$.  We let
\[
\res_r^s:D(A_r)\to D(A_s)
\]
denote the restriction functor along the projection $A_s\to A_r$, for $r<s$.  It is apparent that restriction commutes with the operators $V\ot-$, for $V$ in $\msc{D}$, so that each $\res_r^s$ is a map of $\msc{D}$-module categories.

\begin{lemma}\label{lem:1440}
Consider $A$ as in Section~\ref{sect:qreg} with fixed $q$-regular sequence.  For arbitrary $M$ in $D(A_r)$, there is a containment
\[
\langle \res_{r-1}^r(A_{r-1}\ot_{A_{r}}^{\rm L}M)\rangle\subset \langle \lambda\ot M: \lambda\in\operatorname{Irred}(\msc{D})\rangle.
\]
of thick subcategories in $D(A_r)$.
\end{lemma}

\begin{proof}
Take $\lambda_r=k1_r$ the irreducible in $\msc{D}$ corresponding to the character $\chi_r$, and take $m$ the cohomological degree of $x_r$.  Consider the complex of $A_r$-bimodules
\[
K_{r-1}^r=\left(\Sigma^{m+1} (\lambda_r\ot A_r)\to A_r, d\right),\ \ d(1_r)=x_r.
\]
The fact that the differential is $A_r$-linear on the left follows by the commutativity hypothesis imposed on $x_r$, by the defintion of a $q$-regular sequence.  Note that $K_{r-1}^r$ is semi-free over $A_r$ on the right, and hence the endomorphism $K_{r-1}^r\ot_{A_r}-$ on $A_r$-dgmod preserves quasi-isomorphisms.
\par

Since $x_r$ is a non-zero divisor in $A_r$, the reduction $A_r\to A_{r-1}$ induces a quasi-isomorphism $K_{r-1}\to A_{r-1}$ of dg $A_r$-bimodules.  It follows that we have a natural isomorphism
\[
\res_{r-1}^r\circ(A_{r-1}\ot_{A_{r}}^{\rm L}-) \cong K_{r-1}^r\ot_{A_{r}}^{\rm L}-\cong K_{r-1}^r\ot_{A_{r}}-
\]
of triangulated endomorphisms of $D(A_r)$.  Since
\[
K_{r-1}^r\ot_{A_r} M=\operatorname{cone}(\lambda_r\ot M\overset{d}\to M),
\]
we have
\[
 \res_{r-1}^r(A_{r-1}\ot_{A_{r}}^{\rm L}M)\cong K_{r-1}^r\ot M\in \langle \lambda\ot M: \lambda\in\operatorname{Irred}\msc{D}\rangle.
\]
The proposed containment of thick subcategories follows.
\end{proof}

\begin{lemma}\label{lem:thick_subcats}
Consider $A$ as in Section~\ref{sect:qreg} with fixed $q$-regular sequence, and let $\res$ denote the restriction functor along the projection $A\to A_0$.  For arbitrary $M$ in $A\text{-dgmod}$, there is a containment
\[
\langle \res(A_0\ot_{A}^{\rm L}M)\rangle\subset \langle \lambda\ot M: \lambda\in\operatorname{Irred}(\msc{D})\rangle.
\]
\end{lemma}

\begin{proof}
We claim that there is such a containment 
\[
\langle \res_{r}^s(A_r\ot_{A_s}^{\rm L}M)\rangle\subset \langle \lambda\ot M: \lambda\in\operatorname{Irred}(\msc{D})\rangle
\]
for arbitrary $r< s$ and $M$ over $A_s$.  We proceed by induction on the difference $s-r$.  The base case $r=s-1$ is covered by Lemma~\ref{lem:1440}.  Suppose now that the result holds for $\res_{r}^s(A_{r}\ot_{A_s}^{\rm L}M)$ and consider the desired containment for $\res_{r-1}^s(A_{r-1}\ot_{A_s}^{\rm L}M)$.
\par

Take $N=A_{r}\ot_{A_s}^{\rm L}M$.  Then
\[
A_{r-1}\ot_{A_s}^{\rm L}M\cong A_{r-1}\ot_{A_r}^{\rm L}N.
\]
and $\res_{r-1}^r(A_{r-1}\ot_{A_s}^{\rm L}M)\in \langle \lambda \ot N:\lambda\in\operatorname{Irred}\msc{D}\rangle$, by Lemma~\ref{lem:1440}.  By Lemma~\ref{lem:1415} we have
\[
\langle \lambda\ot N:\lambda\in\operatorname{Irred}\msc{D}\rangle=\langle A_r\ot^{\rm L}_{A_s}(\lambda\ot M):\lambda\in\operatorname{Irred}\msc{D}\rangle
\]
so that $\res_{r-1}^r(A_{r-1}\ot_{A_s}^{\rm L}M)\in \langle A_r\ot^{\rm L}_{A_s}(\lambda\ot M)\rangle_\lambda$.  By restricting further to $A_s$ we find, by our induction hypothesis,
\[
\begin{array}{rl}
\res_{r-1}^s(A_{r-1}\ot_{A_s}^{\rm L}M) & \in \langle \res_r^s A_r\ot^{\rm L}_{A_s}(\lambda\ot M):\lambda\in\operatorname{Irred}\msc{D}\rangle\\
&\hspace{5mm} \subset \langle \mu\ot\lambda\ot M:\lambda,\mu\in\operatorname{Irred}\msc{D}\rangle=\langle \lambda\ot M:\lambda\in\operatorname{Irred}\msc{D}\rangle.
\end{array}
\]
So we have the proposed containment.
\end{proof}

\subsection{Regular sequences and deformations}

Consider $Z\to Q\to R$ a deformation sequence.  Suppose additionally that all of the given algebras are algebras in $\msc{D}$, and that $Z$ is trivial in the sense that $Z$ is supported at the identity in $G^\vee$.  Then the Koszul resolution $K_Q$ is a dg algebra in $\msc{D}$.

\begin{lemma}\label{lem:1951}
If $\{x_1,\dots, x_n\}$ is a $q$-regular sequence in $Q$, then $\{x_1,\dots, x_n\}$ is also a $q$-regular sequence in the Koszul resolution $K_Q$ of $R$.
\end{lemma}

\begin{proof}
Certainly the $x_j$ remain homogenous cocycles in $K_Q$, and have the same commutativity relations in $K_Q/(x_{j+1},\dots, x_n)$, so we need only verify that each $x_j$ is a non-zero divisor in the corresponding quotient.  This final property only has to do with the (non-dg) algebra structure of $K_Q$.  If we take $Q_j=Q/(x_{j+1},\dots, x_n)$, then as an algebra we have
\[
K_Q/(x_{j+1},\dots, x_n)\cong Q_j\ot\wedge^\ast(m_Z/m_Z^2),
\]
and $x_j$ is seen to be a non-zero divisor in this algebra as it is a free module over $Q_j$.
\end{proof}

\section{Restricted enveloping algebras, height 1 doubles, and the quantum Borel}
\label{sect:last}

We employ regular sequences, in conjunction with hypersurface support, to obtain the tensor product property for cohomological support in a number of geometrically Chevalley examples (see Section~\ref{sect:g_chev}).

\subsection{A general lemma}

The following general result was essentially argued in the proof of Theorem~\ref{thm:fun1}, and can be seen as a reduction of the materials of Section~\ref{sect:hop}.

\begin{lemma}\label{lem:gen_lem}
Consider $\msf{u}$ geometrically Chevalley, with Chevalley integration $\msf{U}\to \msf{u}$.  Fix $\msc{D}$ an arbitrary tensor subcategory in $\rep(\msf{u})$.  For $f\in m_Z$ with non-trivial reduction to $m_Z/m_Z^2$, let $K_f$ denote the corresponding Koszul resolution of the hypersurface $\msf{U}^+/(f)$.
\par

Suppose that, at arbitrary such $f\in m_Z$, the endomorphism
\[
k\ot_{\msf{U}^+}^{\rm L}-:D(K_f)\to D(K_f)
\]
is such that $k\ot^{\rm L}_{\msf{U}^+} M\in\langle \lambda\ot M:\lambda\in \msc{D}\rangle$, for each $M$ in $D(K_f)$.  Then hypersurface support for $D_{fin}(\msf{u})$ satisfies the tensor product property
\[
\supp^{hyp}_\mbb{P}(V\ot W)=\supp^{hyp}_\mbb{P}(V)\cap\supp^{hyp}_\mbb{P}(W).
\]
The same result holds when $\msf{u}$ is local and admits a local integration $(\msf{U}^+=)\msf{U}\to \msf{u}$.
\end{lemma}

\begin{proof}
We claim that the inclusion $k\ot^{\rm L}_{\msf{U}^+} M\in\langle \lambda\ot M:\lambda\in \msc{D}\rangle$, at arbitrary dg modules $M$, implies the necessary inclusion $k\in \langle \lambda\ot M:\lambda\in \msc{D}\rangle$ whenever $M$ is non-perfect over the hypersurface $\msf{U}^+/(f)$.  One then applies Lemma~\ref{lem:gen_lem2} to obtain the result.
\par

Via the equivalence $D_{coh}(\msf{U}^+/(f))\to D_{coh}(K_f)$ provided by restriction, it suffices to show that the given inclusion implies $k\in \langle \lambda\ot M:\lambda\in \msc{D}\rangle$ for non-perfect $M$ in $D_{coh}(K_f)$.  Take such non-perfect $M$.  As in the proof of Theorem~\ref{thm:fun1}, one factors the reduction $k\ot^{\rm L}_{\msf{U}^+}-$ as the composite
\[
D_{coh}(K_f)\overset{k\ot^{\rm L}_{\msf{U}^+}-}\longrightarrow D_{fin}(k[\varepsilon])\overset{\rm res}\to D_{coh}(K_f)
\]
and employs the understanding of thick ideals in $D_{fin}(k[\varepsilon])$ from~\cite[Theorem 4.4]{carlsoniyengar15} to find that
\[
k\in \langle k\ot^{\rm L}_{\msf{U}^+} M\rangle\ \subset D_{coh}(K_f)
\]
at non-perfect $M$.  Our assumption now implies the desired inclusion $k\in \langle \lambda\ot M:\lambda\in \msc{D}\rangle$ at non-perfect $M$.
\end{proof}

\subsection{Nilpotent restricted Lie algebras over $\overline{\mbb{F}}_p$}
\label{sect:u_nilp}

Consider $\mfk{n}$ a nilpotent Lie algebra with a restricted structure, in finite characteristic.  Consider any basis $x_1,\dots, x_n$ of $\mfk{n}$ which is compatible with the lower central series $\mfk{n}\supset \mfk{n}^1\supset\dots \supset \mfk{n}^d=0$, in the sense that the final $m_{d-1}$ elements form a basis for $\mfk{n}^{d-1}$, the following $m_{d-2}+m_{d-1}$ elements form a basis for $\mfk{n}^{d-2}$, etc.  Then, each elements $x_r$ is a central non-zero divisor in $U(\mfk{n})_r=U(\mfk{n})/(x_{r+1},\dots, x_n)$, and also in the completion $\hat{U}(\mfk{n})$ along the kernel of the projection $U(\mfk{n})\to u^{\rm res}(\mfk{n})$.  Hence the ordered sequence $\{x_1,\dots, x_n\}$ is a regular sequence in $\hat{U}(\mfk{n})$ which generates the kernel of the augmentation $\hat{U}(\mfk{n})\to k$.
\par

We now recover an essential result of Friedlander and Parshall.

\begin{theorem}[\cite{friedlanderparshall87,suslinfriedlanderbendel97b}]\label{thm:ures}
For an arbitrary nilpotent restricted Lie algebra $\mfk{n}$, the hypersurface support $\supp^{hyp}_\mbb{P}$ on $\rep(u^{\rm res}(\mfk{n}))$ has the tensor product property.  Supposing additionally that the reduced spectrum of cohomology for $u^{\rm res}(\mfk{n})$ is an affine space (e.g.\ supposing the $p$-th power map on $\mfk{n}$ is $0$), then the cohomological support for $u^{\rm res}(\mfk{n})$ has the tensor product property
\[
\supp_{\msf{Y}}(V\ot W)=\supp_{\msf{Y}}(V)\cap\supp_{\msf{Y}}(W).
\]
\end{theorem}

\begin{proof}
The fact that the augmentation ideal in $\hat{U}$ is generated by a regular sequence implies a containment
\[
k\ot_{\hat{U}}^{\rm L}M\in \langle M\rangle\ \subset D_{coh}(K_f),
\]
at arbitrary dg modules $M$ over the resolution $K_f$ of any hypersurface $\hat{U}/(f)$, by Lemma~\ref{lem:thick_subcats}.  So Lemma~\ref{lem:gen_lem}, considered in the case $\msc{D}=Vect$, gives us the tensor product property for hypersurface support.  In the case that the spectrum of cohomology is an affine space, the cohomological support is identified with the hypersurface support by Lemma~\ref{lem:surjO}.
\end{proof}

\subsection{Solvable height 1 doubles over $\overline{\mbb{F}}_p$}

Consider $U$ a smooth unipotent group scheme over $k=\overline{\mbb{F}}_p$.  Suppose that $U$ is normal in a larger group $B$ which admits a quasi-logarithm.  We recall that a quasi-logarithm is the information of an isomorphism of $\mfk{b}=\operatorname{Lie}(B)$-algebras $\hat{\O}_{B,1}\cong k\b{\mfk{b}^\ast}$, so that we have an isomorphism of algebras (see Example~\ref{ex:3})
\[
\hat{D}(U_{(1)}):=\hat{\O}_B\rtimes \hat{U}(\mfk{n})\cong k\b{\mfk{b}^\ast}\rtimes \hat{U}(\mfk{n}),
\]
where $U(\mfk{n})$ acts via the coadjoint action of $u^{\rm res}(\mfk{n})$-action on the generators $\mfk{b}^\ast$ (see e.g.~\cite{kazhdanvarshavsky06}).  We may filter $\mfk{b}^\ast$ by subrepresentations $\mfk{b}^\ast\supset F_{-1}\mfk{b}^{\ast}\supset\dots\supset F_{-d}\mfk{b}^{\ast}=0$ so that each subquotient is a trivial $\mfk{n}$-representation, and consider a corresponding ordered basis $y_1,\dots, y_n$ for $\mfk{b}^\ast$ with the final $m_d$ elements providing a basis for $F_{-d+1}\mfk{b}^\ast$, the final $m_{d-1}+m_d$ elements forming a basis for $F_{-d+2}\mfk{b}^\ast$, etc.  We choose also a basis $x_1,\dots, x_m$ for $\mfk{n}$ as in Section~\ref{sect:u_nilp}.
\par

Then the sequence $x_1,\dots, x_n,y_1,\dots, y_n$ forms a regular sequence in $\hat{D}(U_{(1)})$ which generates the kernel of the augmentation $\hat{D}(U_{(1)})\to k$.  An instance in which the ambient group $B$ admits a quasi-logarithm is the case in which $B$ is a Borel of an almost-simple (smooth) algebraic group $\mbb{G}$~\cite[Corollary 6.4]{friedlandernegron18}, in very good characteristic for $\mbb{G}$.

\begin{theorem}\label{thm:D_unip}
Let $U$ be a unipotent subgroup in an almost-simple algebraic group $\mbb{G}$ over $\overline{\mbb{F}}_p$, which is normalized by a maximal torus.  If $p$ is very good for $\mbb{G}$ then the hypersurface support for the height $1$ double $D(U_{(1)})$ satisfies the tensor product property.  If $p-1>\dim(U)$ then the cohomological support satisfies the tensor product property
\[
\supp_\msf{Y}(V\ot W)=\supp_\msf{Y}(V)\cap\supp_\msf{Y}(W).
\]
\end{theorem}

\begin{proof}
In the case of very good $p$ we have a quasi-logarithm, and hence the augmentation ideal for $\hat{D}$ is generated by a regular sequence.  It follows that for arbitrary $M$ over the Koszul resolution $K_{\hat{D}}$ we have $k\ot_{\hat{D}}^{\rm L}M\in \langle M\rangle$, by Lemma~\ref{lem:thick_subcats}.  So by Lemma~\ref{lem:gen_lem} the hypersurface support for the double $D(U_{(1)})$ satisfies the tensor product property.  When $p-1>\dim(U)$ the reduced spectrum of cohomology is an affine space~\cite[Theorem 6.10]{friedlandernegron18}.  So the hypersurface support and the cohomological support agree, by Lemma~\ref{lem:surjO}.
\end{proof}

Consider now $B\subset \mbb{G}$ a Borel subgroup in almost-simple $\mbb{G}$.  We have $D(B_{(1)})=(\O(B_{(1)})\rtimes kU_{(1)})\rtimes k\mbb{G}_{m(1)}^r$, where $U$ is the unipotent radical in $B$.  We can consider the integration
\[
\hat{D}(B_{(1)}):=\big(\hat{\O}_B\rtimes \hat{U}(\mfk{n})\big)\rtimes k\mbb{G}_{m(1)}^r
\]
of $D(B_{(1)})$.  We note that $\rep(\mbb{G}_{m(1)}^r)=Vect_{G^\vee}$ as a braided (symmetric) category, for $G^\vee$ a rank $r$ elementary abelian $p$-group with trivial form.  So $D(B_{(1)})$ is geometrically Chevalley and, with $\msc{D}=\rep(\mbb{G}_{m(1)}^r)$, fits into the framework of Section~\ref{sect:qreg}.

\begin{theorem}\label{thm:D_borel}
Let $B$ be a Borel subgroup in an almost-simple algebraic group $\mbb{G}$ over $\overline{\mbb{F}}_p$.  If $p$ is very good for $\mbb{G}$, then the hypersurface support for the height $1$ double $D(B_{(1)})$ satisfies the tensor product property.  If $p-1>\dim(B)$, then the cohomological support satisfies the tensor product property as well.
\end{theorem}

\begin{proof}
As remarked above, when $p$ is very good for $\mbb{G}$, $B$ admits a quasi-logarithm, and when $p-1>\dim(B)$ the reduced cohomology ring for $D(B_{(1)})$ is a polynomial ring~\cite[Theorem 6.10]{friedlandernegron18}.  So, by Lemma~\ref{lem:surjO}, we need only show that the hypersurface support satisfies the tensor product property.  By semisimplicity of $k\mbb{G}_{m(1)}^r$, a $\hat{D}$-module is perfect if and only if its restriction to the Hopf subalgebra $\hat{D}^+=\hat{\O}\rtimes \hat{U}(\mfk{n})$ is perfect.  One now proceeds as in the proof of Theorem~\ref{thm:D_unip}.
\end{proof}

\subsection{The quantum Borel in small quantum $\operatorname{SL}_n$}
\label{sect:q_borel}

Take $q\in \mbb{C}^\times$ of odd order $l$.  We first consider the simply-connected form of the quantum Borel, then deal with general quantum Borels in type $A$ in Section \ref{sect:q_borel2}.

By the simply-connected form of the quantum Borel $u_q(\mfk{b})$ we mean the non-negative subalgebra in the small quantum group $u_q(\mfk{sl}_n)=u_q(\operatorname{SL}_n)$.  (See~\cite{negron} and in particular~\cite[Section 9]{negron}.)  So, $u_q(\mfk{b})$ has its usual positive generators $E_\alpha$, and the grouplikes are given by the character group $G=(P/lQ)^\vee$ of the quotient $P/lQ$ of the weight lattice by the $l$-th scaling of the root lattice.  The standard toral elements $K_\alpha$ are precisely the function $K_\alpha=q^{(\alpha,-)}:P/lQ\to \mbb{C}^\times$, where $(-,-)$ is the normalized Killing form, $(\alpha,\beta)=d_\alpha \langle \alpha,\beta\rangle$, $d_\alpha=|\alpha|^2/|\text{short root}|^2$.  (It happens to be the case that all roots are the same length here, so that all $d_\alpha=1$.)
\par

Let us consider first the situation in type $A_2$.  For $u_q(\mfk{b})$ the quantum Borel in $u_q(\mfk{sl}_3)$ we have the standard basis of the positive subalgebras $u_q(\mfk{n})$ and $U^{DK}_q(\mfk{n})$ provided by ordered monomials in the root vectors
\[
E_\alpha,\ E_\beta,\ E_{\alpha+\beta}:=E_\alpha E_\beta-q^{-1}E_\beta E_\alpha
\]
\cite{lusztig90,deconcinikac91}.  One can check by hand, via the Serre relations, that $E_{\alpha+\beta}$ is $\chi$-central in $U^{DK}_q(\mfk{n})$, for the character $\chi$ of the grouplikes given by $\chi(K_\alpha)=q$, $\chi(K_\beta)=q^{-1}$.  In terms of classes $\{\bar{\omega}_\alpha\}_\alpha$ of the fundamental weights $\{\omega_\alpha\}_\alpha$ in $P/lQ$, $\chi=\bar{\omega}_\alpha-\bar{\omega}_\beta$.  So, by considering the monomial basis of $U^{DK}_q(\mfk{n})$ in terms of the $E_\gamma$, we see that the sequence $\{E_\alpha,E_\beta,E_{\alpha+\beta}\}$ provides a $q$-regular sequence in $U^{DK}_q(\mfk{n})$ which generates the kernel of the augmentation ideal.  The root vectors also provide a $q$-regular sequence for the completion $\hat{U}_q^{DK}(\mfk{n})$ with respect to degree.
\par

The above argument generalizes in type $A$ to provide the following.

\begin{lemma}[\cite{andruskiewitschschneider02}]\label{lem:AS}
The augmentation ideal in the completed De Concini-Kac algebra $\hat{U}_q^{DK}(\mfk{n})\subset \hat{U}_q^{DK}(\mfk{b})$ in type $A$, at arbitrary odd order $q$, is generated by a $q$-regular sequence.
\end{lemma}

\begin{proof}
We have the standard generators $E_{i,i+1}:=E_{\alpha_i}$ of $U_q^{DK}(\mfk{n})$ and explicitly define the root vectors $E_{i,j}$, $1\leq i<j\leq n$, by
\[
E_{i,j}=E_{i,j-1}E_{j-1,j}-q^{(\gamma,\nu)}E_{j-1,j}E_{i,j-1},
\]
where $E_{i,j}$ and $E_{i,j-1}$ are of respective degrees $\gamma$ and $\nu$ under the root lattice grading.  We order the set of roots $\{E_{i,j}:1\leq i<j\leq n\}$ with respect to the lex(icographic) order, so that $E_{i,j}<E_{i',j'}$ whenever $i<i'$ or $i=i'$ and $j<j'$.
\par

Having established the above explicit expressions, and ordering, we let $E_\gamma$ denote the unique such vector $E_{i,j}$ of $Q$-degree $\gamma$, for $\gamma$ a positive root, and adopt the lex ordering on the set $\{E_\gamma:\gamma\in \Phi^+\}$ established above.  In~\cite[Lemmas 6.4, 6.7, 6.8]{andruskiewitschschneider02} Andruskiewitsch and Schneider calculate explicitly the $q$-commutators of the root vectors to find
\[
E_{\gamma}E_\nu-q^{(\gamma,\nu)}E_{\gamma}E_{\nu}\in \text{the ideal in $U_q^{DK}(\mfk{n})$ generated by $E_\xi$ with }ht(\xi)>ht(\gamma),\ ht(\nu),
\]
whenever $E_\gamma< E_\nu$ in the lex ordering.  (See in particular \cite[Equations (6-19) and (6-25)]{andruskiewitschschneider02}.)  By considering the basis of $U^{DK}_q(\mfk{n})$ in terms of monomials in the root vectors, it follows that any ordering $\{E_{\gamma_1},\dots,E_{\gamma_m}\}$ of the root vectors which is compatible with the height, in the sense that $ht(\gamma_i)\leq ht(\gamma_{i+1})$, provides a $q$-regular sequence in $U^{DK}_q(\mfk{n})$ which generates the kernel of the augmentation $U^{DK}_q(\mfk{n})\to \mbb{C}$.  The specific character $\chi_\gamma$ associated to $E_\gamma$ is the explicit sum of fundamental weights
\[
\chi_\gamma=\sum_{\alpha\ \text{with}\ E_\alpha<E_\gamma}(\alpha,\gamma)\bar{\omega}_\alpha-\sum_{\beta\ \text{with}\ E_\gamma<E_\beta}(\beta,\gamma)\bar{\omega}_\beta\ \in G^\vee.
\]
By exactness of completion $\mbb{C}\b{E_\gamma^l:\gamma\in \Phi^+}\ot_{\mbb{C}[E_\gamma^l:\gamma\in \Phi^+]}-$, it follows that the ordered set $\{E_{\gamma_1},\dots,E_{\gamma_n}\}$ also provides a $q$-regular sequence in the completed algebra $\hat{U}_q^{DK}(\mfk{n})$ which generates the kernel of the augmentation.
\end{proof}

One applies Lemma~\ref{lem:thick_subcats} and Lemma~\ref{lem:gen_lem} to find.

\begin{theorem}\label{thm:borel}
For $u_q(\mfk{b})$ the (simply-connected) quantum Borel in type $A_n$, at arbitrary odd order $q$, the hypersurface support $\supp^{hyp}_\mbb{P}$ for $\rep(u_q(\mfk{b}))$ satisfies the tensor product property
\[
\supp^{hyp}_\mbb{P}(V\ot W)=\supp^{hyp}_\mbb{P}(V)\cap\supp^{hyp}_\mbb{P}(W).
\]
\end{theorem}

We recall that at $q$ of order greater than the Coxeter number $h$ for $\mfk{sl}_{n+1}$, the cohomology ring for $u_q(\mfk{b})$ is isomorphic to functions on the nilpotent radical $\mfk{n}\cong \mbb{A}^{|\Phi^+|}_\mbb{C}$ \cite{ginzburgkumar93}.  Indeed, in this case $\kappa$ is actually an isomorphism from $\msf{Y}$ onto $\mbb{P}$.  So we obtain the tensor product for cohomological support at such parameters $q$, by applying Lemma \ref{lem:surj0} and Corollary \ref{cor:hopf_hyp}.

\begin{corollary}\label{cor:borel}
Consider $u_q(\mfk{b})$ the (simply-connected) quantum Borel in type $A_n$, at $q$ of odd order $>h$. For arbitrary $V$ and $W$ in $\rep(u_q(\mfk{b}))$ we have
\begin{equation}\label{eq:1629}
\supp_\msf{Y}(V\ot W)=\supp_\msf{Y}(V)\cap\supp_\msf{Y}(W).
\end{equation}
\end{corollary}

We note that the arguments employed in the proof of Theorem~\ref{thm:borel} are specific to type $A$.  In particular, direct calculation indicates that none of the root vectors $\{E_\gamma:\gamma\in \Phi^+\}\subset U_q^{DK}(\mfk{b})$, in types $D_4$ and $B_2$, are skew central, and so no ordering of the positive roots produces a $q$-regular sequence.  One therefore needs a more robust approach to the quantum Borel in other Dynkin type.

\subsection{Arbitrary Borels in type $A$}

Consider $X$ any intermediate lattice $Q\subset X\subset P$ between the root and weight lattice in type $A$, and take $X^M$ the radical of the $q$-exponentiated Killing form $q^{(-,-)}$ on $X$.  Let $\mbb{G}$ be the corresponding algebraic group for $X$, and $B$ be the positive Borel in $\mbb{G}$.
\par

We have the corresponding small quantum group $u_q(\mbb{G})$ \cite[Section 9]{negron} and the quantum Borel $u_q(B)$ in $u_q(\mbb{G})$, which is explicitly the smash product of the nilpotent subalgebra $u_q(\mfk{n})$ with the character group $(X/X^M)^\vee$ of the quotient $X/X^M$.  As stated above, we have taken $u_q(\mfk{b})(=u_q(B_{sc}))$ to be the quantum Borel in the simply-connected form $u_q(\operatorname{SL}_n)$.

\begin{theorem}\label{thm:borel_A}
Consider $\mbb{G}$ an arbitrary almost-simple algebraic group in type $A$, $q$ an odd order root of unity, and $u_q(B)$ the quantum Borel in $u_q(\mbb{G})$.  Then hypersurface support for $u_q(B)$ satisfies the tensor product property.  Furthermore, when $\ord(q)>h$ cohomological support for $u_q(B)$ also satisfies the tensor product property
\[
\supp_\msf{Y}(V\ot W)=\supp_\msf{Y}(V)\cap \supp_\msf{Y}(W).
\]
\end{theorem}

Before giving the proof of Theorem \ref{thm:borel_A}, we provide the following lemma, which reframes our analysis of support for the simply-connected form $u_q(\mfk{b})$.

\begin{lemma}\label{lem:eh}
Consider $\msf{U}_{sc}=\hat{U}^{DK}_q(\mfk{b})$ the (completed) De Concini-Kac integration of the simply-connected form $u_q(\mfk{b})$ in type $A$, and $f\in m_Z$ with non-trivial reduction to $m_Z/m_Z^2$.  Then for $M$ finitely generated and non-perfect over $\msf{U}_{sc}/(f)$ we have
\begin{equation}\label{eq:AHHHA}
k\in\langle \lambda\ot M:\lambda\in \operatorname{Irrep}(G_{sc})\rangle\ \subset D_{coh}(\msf{U}_{sc}/(f)).
\end{equation}
\end{lemma}

The point here is that we have replaced the positive subalgebra $\msf{U}^+$ with the Hopf algebra $\msf{U}_{sc}$ in our analysis of support for hypersurfaces (cf.\ Lemma \ref{lem:gen_lem2}).

\begin{proof}
Fix $f$ as in the statement of the lemma, and let $G_{sc}$ denote the grouplikes in $u_q(\mfk{b})$.  Let $M'$ be a finite and non-perfect module over $\msf{U}^+/(f)$, which we view also as a dg module over the corresponding Koszul resolution $K_f$.  Consider the reduction $K_f\to k\ot_{\msf{U}^+}K_f=k[\varepsilon]$.  The dg version of Hopkins lemma employed in the proof of Theorem \ref{thm:fun1}, \cite[Theorem 4.4]{carlsoniyengar15}, therefore implies that
\[
k\in \langle k\ot_{\msf{U}^+}M'\rangle\ \subset D_{fin}(k[\varepsilon]).
\]
By Lemmas \ref{lem:AS} and \ref{lem:1951}, the kernel of the reduction $K_f\to k[\varepsilon]$ is generated by a $q$-regular sequence.  So by Lemma~\ref{lem:thick_subcats} we have that 
\[
k\in \langle \lambda\ot M':\lambda\in \operatorname{Irrep}(G_{sc})\rangle \subset D_{coh}(\msf{U}^+/(f)).
\] 

We consider the right adjoint
\[
R:\msf{U}^+/(f)\text{-mod}_{fg}\to \msf{U}_{sc}/(f)\text{-mod}_{fg}
\]
to the restriction functor $T:\msf{U}_{sc}/(f)\text{-mod}_{fg}\to \msf{U}^+/(f)\text{-mod}_{fg}$.  Note that both $T$ and $R$ are exact maps of left $\rep(G_{sc})$-module categories \cite[\S 3.3]{etingofostrik04}.  When $M'$ is the restriction $M'=TM$ of a $\msf{U}_{sc}/(f)$-module $M$ we have
\[
R(TM)=\oplus_{\lambda'\in \operatorname{Irrep}(G_{sc})} M\ot\lambda'\cong \oplus_{\lambda'\in \operatorname{Irrep}(G_{sc})} \lambda'\ot M,
\]
where for the final identity we use the centralizing property of $\msf{U}_{sc}$-modules against the irreducibles (Lemma \ref{lem:bos}).  So by considering the derived functor $R:D_{coh}(\msf{U}^+/(f))\to D_{coh}(\msf{U}_{sc}/(f))$ we find for non-perfect $M$ over $\msf{U}_{sc}/(f)$
\[
\begin{array}{rl}
\oplus_{\lambda\in \operatorname{Irrep}(G_{sc})}\lambda=R(k)& \in R(\langle \lambda\ot TM:\lambda\in \operatorname{Irrep}(G_{sc})\rangle)\\
& \hspace{1cm}\subset \langle R(\lambda\ot TM):\lambda\in \operatorname{Irrep}(G_{sc})\rangle\\
& \hspace{1cm}=\langle \lambda\ot \lambda'\ot M:\lambda,\lambda'\in \operatorname{Irrep}(G_{sc})\rangle\\
& \hspace{1cm}=\langle \lambda\ot M:\lambda\in \operatorname{Irrep}(G_{sc})\rangle\ \subset\ D_{coh}(\msf{U}_{sc}/(f)).
\end{array}
\]
Since thick subcategories are closed under taking summands we have finally
\[
k\in\langle \lambda\ot M:\lambda\in \operatorname{Irrep}(G_{sc})\rangle \subset D_{coh}(\msf{U}_{sc}/(f))
\]
whenever $M$ is non-perfect over $\msf{U}_{sc}/(f)$.
\end{proof}

We now prove our theorem.

\begin{proof}[Proof of Theorem \ref{thm:borel_A}]
Write $G_{sc}$ for the grouplikes in $u_q(\mfk{b})$ and $G$ for the grouplikes in $u_q(B)$.  Since the $q$-exponentiated Killing form on $P\times lQ$ is identically $1$, we have an inclusion $lQ\subset X^M$.  So we observe group maps
\[
\xymatrixrowsep{2.5mm}
\xymatrix{
 & X/lQ\ar[r]^{\rm incl}\ar[dl]_{\rm surj} & P/lQ\\
 X/X^M	&  &
}\ \ \overset{\vee}\rightsquigarrow\ \ 
\xymatrix{
G_{sc}\ar[r]^{\rm surj} & G'\\
 	& & G\ar[ul]_{\rm incl},
}
\]
where $G'$ is the character group of $X/lQ$.  For $\msf{u}'$ the Hopf algebra $u_q(\mfk{n})\rtimes G'$, we then have a corresponding surjection $u_q(\mfk{b})\to \msf{u}'$ and inclusion $u_q(B)\to \msf{u}'$.
\par

Note that the Killing form on $P/lQ$ restricts to a form on the intermediate quotient $X/lQ$, and that the kernel of the quotient $X/lQ\to X/X^M$ is the radical of this form on $X/lQ$.  So each of the Hopf algebras $u_q(\mfk{b})$, $\msf{u}'$, and $u_q(B)$ are constructed via a uniform bosonization process, using the $q$-exponentiated Killing form.  These Hopf algebras also have consistent integrations provided by taking the appropriate smash product with $\msf{U}^+=\hat{U}_q^{DK}(\mfk{n})$.  Finally, when $\operatorname{ord}(q)>h$ the spectrum of cohomology for each of these Hopf algebras is an affine space \cite{ginzburgkumar93}, so that the hypersurface and cohomological supports agree in this case.  Therefore to prove our result it suffices to show that hypersurface support for $u_q(B)$ splits over tensor products.
\par

Write $\msf{U}_{sc}$, $\msf{U}'$, and $\msf{U}$ for the integrations of $u_q(\mfk{b})$, $\msf{u}'$, and $u_q(B)$ respectively.  Consider $f\in m_Z$ any function with nonzero reduction to $m_Z/m_Z^2$.
\par

Consider $M$ non-perfect over the algebra $\msf{U}'/(f)$, and restrict along the projection $\msf{U}_{sc}/(f)\to \msf{U}'/(f)$ to consider $M$ as a non-perfect $\msf{U}_{sc}/(f)$-module.  We have the map
\[
L=\mbb{C}G\ot_{\mbb{C}G_{sc}}-:\msf{U}_{sc}/(f)\text{-mod}_{fg}\to \msf{U}'/(f)\text{-mod}_{fg}
\]
which is an exact adjoint to restriction.  This functor simply sends a $\msf{U}_{sc}/(f)$-module, which is graded by $P/lQ$, to its $\msf{U}'/(f)$-summand which is supported on the subgroup $X/lQ$.  So, in particular, for $M$ restricted from a $\msf{U}'/(f)$-module we have
\[
L(\lambda\ot M)=\left\{\begin{array}{ll}
\lambda\ot M &\text{if }\lambda\in X/lQ\\
0 &\text{otherwise}.
\end{array}\right.
\]
Applying $L$ to the formula \eqref{eq:AHHHA} of Lemma \ref{lem:eh} then gives, for such non-perfect $M$,
\begin{equation}\label{eq:AAAAAAAAAAH}
k\in\langle \mu\ot M:\mu\in \operatorname{Irrep}(G')\rangle \subset D_{coh}(\msf{U}'/(f)).
\end{equation}
One therefore argues as in the proof of Lemma \ref{lem:gen_lem2} to find that hypersurface support for the intermediate algebra $\msf{u}'$ satisfies the tensor product property.
\par

We consider now the inclusion $u_q(B)\to \msf{u}'$.  One can argue exactly as in the proof of Theorem \ref{thm:qci} to see that hypersurface support for $u_q(B)$ satisfies the tensor product property.  In particular, we observe that the formula \eqref{eq:AAAAAAAAAAH} implies the analogous formula with $\msf{U}'$ replaced by the positive subalgebra $\msf{U}^+$, and follows precisely the arguments of Theorem \ref{thm:qci}.
\end{proof}

\begin{remark}
Supposing that hypersurface support for the simply-connected form $u_q(\mfk{b})$, in arbitrary Dynkin type, satisfies the tensor product property, the arguments of Theorem \ref{thm:borel_A} show that supports for all Borels $u_q(B)$ away from the simply-connected form also satisfy the tensor product property.
\end{remark}

\subsection{A conjecture in general type}
\label{sect:q_borel2}

We conclude with the obvious conjecture.

\begin{conjecture}\label{conj:borel}
Let $\mbb{G}$ be an arbitrary almost-simple algebraic group, and $u_q(B)$ be the quantum Borel in $u_q(\mbb{G})$, at $q$ of arbitrary odd order.  Then cohomological support $\supp_\msf{Y}$ for $\rep(u_q(B))$ satisfies the tensor product property~\eqref{eq:1629}.
\end{conjecture}

In the sequel \cite{negronpevtsova2} we expect to show that the conclusions of Theorem \ref{thm:borel_A}, in type $A$, hold at $q$ of arbitrary odd order.  This informs our omission of a bound on the order of $q$ in the statement of Conjecture \ref{conj:borel}.

\renewcommand{\O}{\oldO}
\bibliographystyle{abbrv}

\end{document}